\documentclass[]{article}
\usepackage[a4paper, margin=3cm]{geometry}
\usepackage{amsmath}
\usepackage{amssymb}
\usepackage{amsthm}
\theoremstyle{plain}
\newtheorem{theorem}{Theorem}[section]
\newtheorem{lemma}[theorem]{Lemma}
\newtheorem{proposition}[theorem]{Proposition}

\theoremstyle{definition}
\newtheorem{assumption}[theorem]{Assumption}
\newtheorem{definition}[theorem]{Definition}
\newtheorem{example}[theorem]{Example}

\theoremstyle{remark}
\newtheorem{remark}[theorem]{Remark}

\usepackage{epstopdf}
\usepackage{makecell,multirow,diagbox,rotating,tabularx}
\usepackage{color}
\usepackage{booktabs}
\usepackage{algorithm}
\usepackage{algorithmic}
\usepackage{url}
\usepackage{cite}
\usepackage[hidelinks]{hyperref}
\usepackage{authblk} 
\usepackage{cleveref}
\RequirePackage{bm}
\usepackage{tikz}
\newcommand{\alglabel}[1]{\textcolor{gray}{\textit{$\triangleright$~#1}}}

\newcommand{\email}[1]{\href{mailto:#1}{#1}}

\newenvironment{keywords}{\paragraph{Keywords:}}{}

\title{Limiting Stationarity of Regularized Gap-Function Reformulations for Bilevel Optimization with Unbounded Multipliers
}

\author{
	Xiaoning Bai\thanks{Department of Mathematics, Southern University of Science and Technology, Shenzhen 518055, People's Republic of China. (\email{12331002@mail.sustech.edu.cn}).} \quad
	Shangzhi Zeng\thanks{National Center for Applied Mathematics Shenzhen, and Department of Mathematics, Southern University of Science and Technology, Shenzhen 518055, People's Republic of China. (\email{zengsz@sustech.edu.cn}).} \quad
	Jin Zhang\thanks{Corresponding author. Department of Mathematics, and National Center for Applied Mathematics Shenzhen, Southern University of
		Science and Technology, Shenzhen 518055, People's Republic of China. (\email{zhangj9@sustech.edu.cn}).}
}

\date{}

\begin{document}
	
\maketitle

\begin{abstract}
	Value-function-type reformulations have generated a broad class of methods for bilevel optimization. However, the corresponding value-function-type constraints are inherently degenerate and generally fail to satisfy standard constraint qualifications, so the associated multiplier sequences may be unbounded and bounded-multiplier convergence analyses become inapplicable. We study this issue for the regularized gap-function reformulation of bilevel problems with constrained convex lower-level programs. We prove that accumulation points of approximate stationary sequences are C-stationary for the corresponding Karush--Kuhn--Tucker-based mathematical program with complementarity constraints (MPCC), even when the multiplier sequence associated with the regularized gap-function constraint is unbounded. The result holds under Mangasarian--Fromovitz constraint qualification (MFCQ) for the upper- and lower-level constraint systems and MPCC-MFCQ at the limiting MPCC point, without any constraint qualification on the regularized gap-function constraint itself. We further provide an example showing that approximate stationary points of the standard regularized gap-function reformulation may converge to a point that is C-stationary but not M-stationary. To guarantee M-stationarity, we introduce a slack-based two-parameter penalty formulation preserving exact multiplier--slack complementarity and establish M-stationarity under a domination condition on the penalty parameters.  We develop an inexact slack-penalty method with adaptive penalty updates and feasibility correction, whose accumulation points are M-stationary under the stated assumptions.	
\end{abstract}


\section{Introduction}\label{sec1}

Bilevel optimization models hierarchical decision processes in which an upper-level decision is constrained by the solution set of a lower-level problem, with applications in operations research, economics, machine learning, and data-driven decision making \cite{bard2013practical,dempe2002foundations,dempe2020bilevel,colson2007overview,franceschi2018bilevel}.

In this paper, we consider the bilevel optimization problem
\begin{equation}\label{BP}
	\begin{aligned}
		(\text{BP}) \qquad \min_{x \in \mathbb{R}^n, y \in \mathbb{R}^m} \  & F(x, y) \\
		\text{s.t.}\quad\ \ \, &y \in S(x), \; G(x,y) \leq 0,
	\end{aligned}
\end{equation}
where \(S(x)\) denotes the solution set of the lower-level problem 
\begin{equation}\label{LL}
	\begin{aligned}
		(P_x) \qquad \min_{y \in \mathbb R^m} \ & f(x,y) \\
		\mathrm{s.t.} \,\ 
		&g(x,y)\le 0,
	\end{aligned}
\end{equation}	
Here \(F:\mathbb R^n\times\mathbb R^m\to\mathbb R\), \(f:\mathbb R^n\times\mathbb R^m\to\mathbb R\), \(G=(G_1,\ldots,G_p):\mathbb R^n\times\mathbb R^m\to\mathbb R^p\) and \(g=(g_1,\ldots,g_q):\mathbb R^n\times\mathbb R^m\to\mathbb R^q\). We denote the upper-level constraint set by $\Omega:=\{(x,y)\in\mathbb R^n\times\mathbb R^m\mid G(x,y)\le0\}$. The feasible region of the upper-level variable is $X:=\{x\in\mathbb R^n\mid \exists\ y\in\mathbb R^m\ {\rm s.t.}\ (x,y)\in\Omega\}$, and $\Gamma(x):=\{y\in\mathbb R^m\mid g(x,y)\le0\}$ is the feasible region of the lower-level problem.

The principal difficulty lies in the implicit solution-set constraint
\(y\in S(x)\), which may render the bilevel feasible region nonsmooth,
nonconvex, and difficult to characterize even when all problem data are smooth.
A classical approach is to replace the implicit constraint \(y\in S(x)\) by the Karush--Kuhn--Tucker (KKT) system, yielding the following single-level reformulation:
\begin{equation}\label{eq:kkt-reform}
	\begin{aligned}
		\min_{x \in\mathbb R^n,\,y\in \mathbb R^m,\,z\in\mathbb R^{q}}
		\quad & F(x,y) \\
		\mathrm{s.t.}\qquad\quad \ \,
		& G(x,y) \le 0, \ \nabla_y \mathcal L(x,y,z)=0, \\
		& 0 \le z\;\perp\; -g(x,y) \ge 0,
	\end{aligned}
\end{equation}
where $\mathcal L(x,y,z):=f(x,y)+z^\top g(x,y)$ is the lower-level Lagrangian, and \(z\in\mathbb R_+^q\) is a lower-level multiplier. This problem is a mathematical program with complementarity constraints (MPCC); see, e.g., \cite{luo1996mathematical}. 

Under suitable convexity-type assumptions on the lower-level problem and some constraint qualifications, the MPCC reformulation is equivalent to the original bilevel problem \cite{dempe2012bilevel}. To address the failure of standard constraint qualifications caused by complementarity constraints, 
MPCC-tailored constraint qualifications and stationarity conditions have been developed; see, e.g., \cite{scheel2000mathematical,ye1999optimality,jane2005necessary,flegel2005m,kanzow2013new,dempe2012karush,kaming2025approximate}.

Recently, a variety of algorithms have been proposed for solving bilevel programs based on their MPCC reformulations \cite{hu2023improved,li2022bilevel,bennett2008bilevel,dempe2020bilevel,dempe2025duality}. Although the MPCC reformulation is natural, directly handling the full KKT system can be computationally demanding from an algorithmic perspective, particularly in large-scale settings. Indeed, linearizing the lower-level stationarity equation requires differentiating \(\nabla_y\mathcal L\), and hence evaluating second-order derivatives of the lower-level Lagrangian. 

Motivated by these considerations,  value-function-type reformulations have received increasing attention in the design of bilevel algorithms. The classical value-function reformulation, studied in the early works \cite{outrata1990numerical,ye1995optimality}, replaces the implicit constraint \(y\in S(x)\) by lower-level feasibility together with a value-function inequality, leading to
\[
	\min_{x\in\mathbb R^n,y\in\mathbb R^m}
		\quad F(x,y) \quad
		\mathrm{s.t.}\ G(x,y)\le 0,\ g(x,y)\le 0,\ f(x,y)-v(x)\le 0.
\]
Here, $v(x):=\inf_{u\in\mathbb R^m}\{f(x,u)\mid g(x,u)\le0\}$ is the value function of the lower-level problem. Since \(f(x,y)\ge v(x)\) for every \(y\in\Gamma(x)\), the constraint \(f(x,y)-v(x)\le0\) exactly characterizes lower-level optimality.

The value-function reformulation and its variants have been studied from both theoretical and algorithmic perspectives. Theoretical work has examined their equivalence properties, constraint qualifications, and necessary optimality conditions \cite{dempe2013bilevel,ye2010new,bai2025optimality}. On the algorithmic side, recent developments initially focused on unconstrained lower-level models. Representative examples include deterministic value-function-based methods that approximate the lower-level value function through finitely many gradient steps \cite{liu2022bome,shen2023penalty,shen2025penalty,sow2022primal}, stochastic function-value-gap penalty methods \cite{kwon2023fully,kwon2024penalty}, and federated value-function methods \cite{yang2025first,zhang2025federated}. Value-function and Moreau-envelope reformulations have also led to interior-point, difference-of-convex, inexact first-order, and smoothed stochastic primal--dual methods \cite{gao2022value,gao2024,liu2024moreau,bai2026alternating,lu2025tsp}.

For constrained lower-level programs, a number of algorithms exploit particular constraint structures, including equality constraints \cite{xiao2023alternating}, linear equality and inequality constraints \cite{khanduri2023linearly,kornowski2024first}, and strongly convex lower-level problems with coupled convex constraints \cite{jiang2024barrier}. General nonlinear lower-level constraints have been approached through hypergradient and stochastic-gradient methods \cite{shi2024double,giovannelli2025inexact}, penalty and minimax reformulations \cite{lu2024first}, proximal Lagrangian value functions \cite{yaoconstrained}, and augmented Lagrangian value-function reformulations \cite{nie2025augmented}. Regularized gap functions provide another value-function-type approach to general nonlinear lower-level constraints: they represent lower-level optimality through a regularized primal--dual gap. In particular, the regularized gap function introduced in \cite{yao2025overcoming} is defined as
\begin{equation}\label{eq:defG}
\mathcal G_\gamma(x,y,z):= \max_{\theta \in \mathbb R^{m}, \lambda \in\mathbb R_+^{q}} \Big\{
\mathcal{L}(x,y,\lambda) - \frac{1}{2\gamma_2}\| \lambda - z\|^2 - \mathcal{L}(x,\theta,z)- \frac{1}{2\gamma_1}\| \theta - y\|^2 \Big\},
\end{equation}
where \(z\in\mathbb R_+^q\), \(\gamma=(\gamma_1,\gamma_2)>0\), and \(\mathcal L(x,y,z):=f(x,y)+z^\top g(x,y)\) is the lower-level Lagrangian. 
Under the convexity assumptions and the existence of lower-level KKT multipliers, \(\mathcal G_\gamma(x,y,z)\ge0\), and \(\mathcal G_\gamma(x,y,z)=0\) is equivalent to \(y\in S(x)\) and \(z\in\mathcal M(x,y)\) \cite{yao2025overcoming}. The regularized gap function is continuously differentiable and encodes
lower-level feasibility and optimality through a single scalar constraint. This leads to the regularized gap-function reformulation
\begin{equation}\label{reformulation_GP_orig}
	\begin{aligned}
		\text{(GP)} \ \min_{x \in\mathbb R^n,\,y\in \mathbb R^m,\,z\in\mathbb R_+^{q}} \quad & F(x,y) \\
		\text{s.t.}\qquad\ \quad\, & \mathcal G_\gamma(x,y,z)\le0,\quad G(x,y)\le0,
	\end{aligned}
\end{equation}

Although value-function-type reformulations have led to many algorithms, these methods often solve penalized or relaxed problems, with penalty parameters tending to infinity or relaxation parameters tending to zero \cite{liu2024moreau,yao2025overcoming}. In standard nonlinear programming, suitable constraint qualifications often provide the multiplier boundedness needed to pass approximate first-order conditions to the limit. However, value-function-type constraints in bilevel optimization are intrinsically degenerate and generally fail to satisfy standard constraint qualifications \cite{bai2022directional}. Consequently, their associated multiplier sequences may be unbounded, in which case the conventional bounded-multiplier analysis is no longer applicable. For bilevel programs with affine constraints and linear lower-level problems, the challenge associated with unbounded multiplier sequences in the value-function reformulation has been addressed in \cite{mehlitz2023asymptotic}. In more general settings, however, the limiting stationarity information retained by approximate first-order sequences with unbounded multipliers remains insufficiently understood. In this paper, we address this issue for the regularized gap-function reformulation of bilevel problems with constrained convex lower-level programs. We show that, under suitable constraint qualifications for the original constraint systems and the associated KKT reformulation viewed as an MPCC, meaningful MPCC stationarity can still be recovered without imposing any constraint qualification involving the regularized gap-function constraint itself.

Against this background, we investigate the MPCC stationarity properties of
accumulation points of sequences satisfying approximate first-order conditions
for the regularized gap-function reformulation. The main contributions of this
paper are summarized as follows.
\begin{itemize}
	
		\item We formulate an approximate KKT condition that captures the common
	limiting first-order structure of penalty approximation of
	the regularized gap-function reformulation. Under Mangasarian--Fromovitz constraint qualification (MFCQ) for the upper- and
	lower-level constraint systems and MPCC-MFCQ at the limiting MPCC point,
	we prove that every accumulation point is C-stationary for the associated
	KKT-based MPCC reformulation, even when the multiplier sequence associated
	with the regularized gap-function constraint is unbounded.
	
		\item We construct an example showing that the preceding C-stationarity
	conclusion is generally sharp: stationary points of the standard
	penalty formulation may converge to a point that is C-stationary but
	not M-stationary for the associated MPCC. Motivated by this limitation, we
	introduce a slack-based two-parameter penalty formulation that preserves
	exact multiplier--slack complementarity. Under suitable constraint
	qualifications and a domination condition between the two penalty parameters, we prove that accumulation points satisfying the
	resulting approximate KKT condition are M-stationary for the associated
	KKT-based MPCC reformulation.

		\item We develop an inexact regularized gap-function bilevel slack-penalty
	method (iG-BSPM), which permits inexact solution of the proximal lower-level
	subproblem through an implementable residual criterion and incorporates
	adaptive penalty updates and feasibility correction. We prove that the
	 iterates generate an approximate KKT sequence for the
	slack-based formulation. Consequently, under the stated assumptions, every
	accumulation point is M-stationary for the associated
	MPCC reformulation  under MPCC-MFCQ.
\end{itemize}

This paper is organized as follows. Section~\ref{sec2} introduces the standing assumptions, basic properties of the regularized gap function, and the MPCC constraint qualifications and stationarity concepts used in the analysis. Section~\ref{sec3} introduces an approximate KKT condition for the regularized gap-function reformulation and establishes C-stationarity of its accumulation points. Section~\ref{sec4} introduces a slack-based two-parameter penalty formulation and its approximate KKT condition, and establishes M-stationarity of its accumulation points. Section~\ref{sec5} develops iG-BSPM and establishes its convergence
properties. The paper concludes with a summary of the main results.

\section{Preliminaries}\label{sec2}

\subsection{Notations and Basic Assumptions}

Let $\mathbb{R}^d$ denote the $d$-dimensional Euclidean space, and let $\mathbb{B}$ denote the closed unit ball centered at the origin. The nonnegative and positive orthants in $\mathbb{R}^d$ are denoted by $\mathbb{R}^d_+$ and $\mathbb{R}^d_{++}$, respectively. The standard inner product and Euclidean norm are denoted by $\langle \cdot, \cdot \rangle$ and $\|\cdot\|$, respectively.
For a vector \( w \in \mathbb{R}^d \), \([w]_+\) denotes the componentwise positive part, i.e., $([w]_+)_i=\max\{w_i,0\},\ i=1,\ldots,d$. For a vector \( w \in \mathbb{R}^d \) and a closed set \( D \subset \mathbb{R}^d \), the distance from \( w \) to \( D \) is defined as \(\mathrm{dist}(w, D) = \min_{u \in D} \|w - u\|\), the Euclidean projection operator onto \( D \) is denoted by \(\mathrm{Proj}_D(w):=\arg\min_{u\in D}\|u-w\|\). When \(D\) is nonconvex, \(\operatorname{Proj}_D(w)\) denotes an arbitrary selected element of the projection set whenever it is used as a point. The Cartesian product of two sets \( D_1 \) and \( D_2 \) is denoted by \( D_1 \times D_2 \). For a nonempty closed set \(D\), the Fr\'echet normal cone to \(D\) at \(\bar w\) and the Mordukhovich (limiting) normal cone to \(D\) at \(\bar w\) are defined by
\[
\begin{aligned}
	\widehat{\mathcal N}_D(\bar w)
	&:=
	\left\{
	\xi\in\mathbb R^d
	\;\middle|\;
	\langle \xi,w-\bar w\rangle
	\le o(\|w-\bar w\|)
	\quad \text{as } w\to \bar w,\ w\in D
	\right\},\\
	\mathcal N_D(\bar w)
	&:=
	\left\{
	\xi\in\mathbb R^d
	\;\middle|\;
	\exists\, w^k\to \bar w,\ \exists\, \xi^k\to \xi
	\text{ such that }
	w^k\in D,\ 
	\xi^k\in \widehat{\mathcal N}_D(w^k)
	\right\},
\end{aligned}
\]
where \( o(w) \) denotes a function satisfying \( o(w)/w\to 0 \) as \( w \downarrow 0 \). When \(D\) is convex, \(\mathcal N_D(\bar w)\) reduces to the classical convex normal cone.

The following assumptions are the standing assumptions throughout the paper. 

\begin{assumption}\label{asup_stationarity_upper}
	The upper-level objective function \(F\) is continuously differentiable on an open set containing \(\Omega\). For each \(i=1,\ldots,p\), the upper-level constraint function \(G_i\) is continuously differentiable and convex on an open set containing \(\Omega\).
\end{assumption}

\begin{assumption}\label{asup_stationarity_lower}
	For every \(x\in X\), the functions \(f(x,\cdot)\) and \(g_i(x,\cdot)\), \(i=1,\ldots,q\), are convex on \(\mathbb R^m\). Moreover, \(f\) and \(g_i\), \(i=1,\ldots,q\), are twice continuously differentiable on an open set containing \(X\times\mathbb R^m\).
\end{assumption}

\subsection{Basic Properties of the Regularized Gap Function}
This subsection recalls several basic properties of the regularized gap function used below.

The result \cite[Lemma 2.1]{yao2025overcoming} shows that the regularized gap function is nonnegative and that its vanishing exactly characterizes lower-level optimality together with the lower-level KKT multiplier condition.

\begin{lemma}\label{lem:gap_nonnegative_exact}
	Under the standing assumptions, let \(\gamma_1,\gamma_2>0\), for any \(x\in X\) and \(y\in\mathbb R^m\), define
	\[
	\mathcal M(x,y):=\left\{\lambda\in\mathbb R_+^{q}\ \middle|\ 0=\nabla_y f(x,y)+\nabla_y g(x,y)^\top\lambda,\ \lambda^\top g(x,y)=0\right\}.
	\]
	Then, for any \((x,y,z)\in X\times\mathbb R^m\times\mathbb R_+^q\), one has \(\mathcal G_\gamma(x,y,z)\ge0\). Moreover, for any \((x,y,z)\in X\times\mathbb R^m\times\mathbb R_+^q\), $\mathcal G_\gamma(x,y,z)\le 0$  if and only if $y\in S(x)$ and $z\in\mathcal M(x,y)$.
\end{lemma}

Using the definition of \(\Omega\), the regularized gap-function reformulation can be written compactly as
\begin{equation}\label{reformualtion_GP}
	\text{(GP)} \qquad
	\min_{(x,y,z)\in\Omega\times\mathbb R_+^{q}}
	\quad F(x,y)
	\quad
	\text{s.t.}\quad
	\mathcal G_\gamma(x,y,z)\le0.
\end{equation}

The preceding property yields the equivalence between the gap-function reformulation and the original bilevel problem, provided that lower-level multipliers exist at all feasible points; see \cite[Theorem 2.3]{yao2025overcoming}.

\begin{proposition}\label{prop:gap_reform_exact}
	Under the standing assumptions, let \(\gamma_1,\gamma_2>0\). Assume further that, for any feasible point \((x,y)\) of \eqref{BP}, a lower-level multiplier exists at \((x,y)\), i.e., \(\mathcal M(x,y)\neq\emptyset\). Then the regularized gap-function reformulation \eqref{reformualtion_GP} is equivalent to the bilevel optimization problem \eqref{BP}. Specifically, \((\bar{x},\bar{y})\) is an optimal solution to \eqref{BP} and \(\bar{z}\in\mathcal M(\bar{x},\bar{y})\) if and only if \((\bar{x},\bar{y},\bar{z})\) is an optimal solution to \eqref{reformualtion_GP}.
\end{proposition}

The function $\mathcal G_\gamma$ is continuously differentiable when \(f\) and \(g\) are continuously differentiable. Under the standing assumptions, the function \(\theta\mapsto f(x,\theta)+z^\top g(x,\theta)+\|\theta-y\|^2/(2\gamma_1)\) is coercive and \(1/\gamma_1\)-strongly convex for every \((x,y,z)\in\Omega\times\mathbb R_+^q\). Hence, the minimizer $\theta^*(x,y,z)$ defined below exists and is unique. The next result is adapted from \cite[Lemma 2.2]{yao2025overcoming}.

\begin{lemma}\label{smooth_G}
	Under the standing assumptions, let \(\gamma_1,\gamma_2>0\). Then \(\mathcal G_\gamma\) is continuously differentiable on \(\Omega\times\mathbb R_+^{q}\), and for any \((x,y,z)\in\Omega\times\mathbb R_+^{q}\),
	\begin{equation}\label{Ggradient}
		\nabla \mathcal G_\gamma(x,y,z) 
		=	\begin{pmatrix}
			\nabla_x f(x,y) + \nabla_x g(x,y)^\top \lambda^*  \\
			\nabla_y f(x,y) + \nabla_y g(x,y)^\top \lambda^*  \\
			-\left(z - \lambda^*\right)/\gamma_2
		\end{pmatrix}-
		\begin{pmatrix}
			\nabla_x f(x,\theta^*) + \nabla_x g(x,\theta^*)^\top z \\
			\left(y - \theta^*\right)/\gamma_1 \\
			g(x, \theta^*)
		\end{pmatrix},
	\end{equation}
	where $\theta^*$ and $\lambda^*$ denote $\theta^*(x,y,z)$ and $\lambda^*(x,y,z)$, respectively, defined as
	\begin{equation}\label{eq:deftheta}
		\begin{aligned}
			\theta^*(x,y,z) &:= \underset{\theta \in \mathbb R^m}{\mathrm{argmin}}  \left\{ f(x,\theta) + z^{T} g(x, \theta) + \frac{1}{2\gamma_1}\| \theta - y\|^2  \right\}, \\
			\lambda^*(x,y,z) &:= \underset{\lambda \in \mathbb R_+^{q}}{\mathrm{argmax}} \left\{  f(x, y) + \lambda^{T} g(x, y) - \frac{1}{2\gamma_2}\| \lambda - z\|^2 \right\} = \mathrm{Proj}_{\mathbb R_+^{q}}\left( z + \gamma_2 g(x,y)  \right).
		\end{aligned}
	\end{equation}
\end{lemma}

The next lemma shows that \(\mathcal G_\gamma(x^k,y^k,z^k)\to0\) implies \(\theta_k^*-y^k\to0\) and \(\lambda_k^*-z^k\to0\), where \(\theta_k^*:=\theta^*(x^k,y^k,z^k)\) and \(\lambda_k^*:=\lambda^*(x^k,y^k,z^k)\).

\begin{lemma}\label{lem:vanishing_gap_residuals}
	Under the standing assumptions, let \(\gamma_1,\gamma_2>0\), \(\{(x^k,y^k,z^k)\}\subset\Omega\times\mathbb R_+^q\), and define \(\theta_k^*:=\theta^*(x^k,y^k,z^k)\) and \(\lambda_k^*:=\lambda^*(x^k,y^k,z^k)\). If \(\mathcal G_\gamma(x^k,y^k,z^k)\to0\), then 
	\[
	\theta_k^*-y^k\to0, \quad \lambda_k^*-z^k\to0.
	\]
\end{lemma}

\begin{proof}
	For each \(k\), define $\theta_k^*:=\theta^*(x^k,y^k,z^k)$, and $\lambda_k^*:=\lambda^*(x^k,y^k,z^k)$. By the definition of \(\mathcal G_\gamma\),
	\[
	\mathcal G_\gamma(x^k,y^k,z^k)
	=
	\max_{\lambda\in\mathbb R_+^q}
	\left\{
	\mathcal L(x^k,y^k,\lambda)
	-\frac{1}{2\gamma_2}\|\lambda-z^k\|^2
	\right\}
	-
	\min_{\theta\in\mathbb R^m}
	\left\{
	\mathcal L(x^k,\theta,z^k)
	+\frac{1}{2\gamma_1}\|\theta-y^k\|^2
	\right\}.
	\]
	
	Since \(z^k\in\mathbb R_+^q\) is feasible for the maximization problem in \(\lambda\), and the minimization problem defining \(\theta_k^*\) is \(1/\gamma_1\)-strongly convex with \(y^k\) feasible, we obtain
	\begin{equation}\label{eq:theta_residual_bound}
		\begin{aligned}
			\frac{1}{2\gamma_1}\|\theta_k^*-y^k\|^2\le
			\mathcal L(x^k,y^k,z^k)
			-
			\left(
			\mathcal L(x^k,\theta_k^*,z^k)
			+\frac{1}{2\gamma_1}\|\theta_k^*-y^k\|^2
			\right)\le
			\mathcal G_\gamma(x^k,y^k,z^k).
		\end{aligned}
	\end{equation}
	Similarly, since \(y^k\) is admissible for the minimization problem in \(\theta\), and since the maximization problem defining \(\lambda_k^*\) is \(1/\gamma_2\)-strongly concave with \(z^k\in\mathbb R_+^q\) feasible, we obtain
	\begin{equation}\label{eq:lambda_residual_bound}
		\frac{1}{2\gamma_2}\|\lambda_k^*-z^k\|^2\le\left(
			\mathcal L(x^k,y^k,\lambda_k^*)
			-\frac{1}{2\gamma_2}\|\lambda_k^*-z^k\|^2
			\right)		-		\mathcal L(x^k,y^k,z^k) \le
			\mathcal G_\gamma(x^k,y^k,z^k).
	\end{equation}
	Since \(\mathcal G_\gamma(x^k,y^k,z^k)\to0\), the two estimates \eqref{eq:theta_residual_bound} and \eqref{eq:lambda_residual_bound} imply $\theta_k^*-y^k\to0$, $\lambda_k^*-z^k\to0$. 
\end{proof}

Since $\lambda^*(x,y,z)$ admits a closed-form expression, substituting it into the maximization problem in the definition of $\mathcal G_\gamma(x,y,z)$ yields the representation
\begin{equation}\label{vg_eq2}
	\mathcal G_\gamma(x,y,z) = f(x,y) + \mathcal{P}_{\gamma_2}(x,y,z) - M_{\mathcal L}^{\gamma_1}(x,y,z).
\end{equation}
Here,
\[
	\mathcal{P}_{\gamma_2}(x,y,z) := \frac{1}{2\gamma_2} \left(\| [z + \gamma_2 g(x,y)]_+\|^2 - \|z\|^2\right),
\]
which is the classical Powell--Hestenes--Rockafellar (PHR) penalty function \cite{hestenes1969multiplier,powell1969method,rockafellar1976augmented} for the lower-level constraints, and
\begin{equation}\label{def_moreau_envelope}
	M_{\mathcal L}^{\gamma_1}(x,y,z) := \min_{\theta \in \mathbb R^m} \left\{ \mathcal L(x,\theta,z) + \frac{1}{2\gamma_1}\| \theta - y\|^2 \right\},
\end{equation}
which is the Moreau envelope \cite{moreau1965proximite} of the lower-level Lagrangian \(\mathcal L\), for fixed \(x\) and \(z\). Since the squared positive-part function is continuously differentiable, \(\mathcal P_{\gamma_2}\) is continuously differentiable whenever \(g\) is continuously differentiable. 

In the algorithmic analysis, \(\theta^*(x,y,z)\) is computed only approximately. The next estimate quantifies how the stationarity residual controls both the distance to \(\theta^*(x,y,z)\) and the Moreau-envelope error.
\begin{lemma}\label{error_bound}
	Suppose Assumption~\ref{asup_stationarity_lower} holds, and let \(\gamma_1>0\). Then, for any \((x,y,z)\in X\times\mathbb R^m\times\mathbb R_+^q\) and \(\theta\in\mathbb R^m\), it holds that
	\[
	\|\theta-\theta^*(x,y,z)\|
	\le
	\gamma_1
	\left\|
	\nabla_y f(x,\theta)
	+\nabla_y g(x,\theta)^\top z
	+\frac{\theta-y}{\gamma_1}
	\right\|,
	\]
	where $\theta^*(x,y,z):=\operatorname{argmin}_{\theta\in\mathbb R^m}\left\{f(x,\theta)+z^\top g(x,\theta)+\|\theta-y\|^2/(2\gamma_1)	\right\}$. Moreover,
	\[
	\mathcal L(x,\theta,z)
	+\frac{1}{2\gamma_1}\|\theta-y\|^2
	-M_{\mathcal L}^{\gamma_1}(x,y,z)
	\le
	\frac{\gamma_1}{2}
	\left\|
	\nabla_y f(x,\theta)
	+\nabla_y g(x,\theta)^\top z
	+\frac{\theta-y}{\gamma_1}
	\right\|^2.
	\]
\end{lemma}

\begin{proof}
	Fix \((x,y,z)\in X\times\mathbb R^m\times\mathbb R_+^q\). By Assumption~\ref{asup_stationarity_lower} and \(z\in\mathbb R_+^q\), the function \(\theta\mapsto\mathcal L(x,\theta,z)\) is convex. Hence $\mathcal L(x,\theta,z)+\frac{1}{2\gamma_1}\|\theta-y\|^2$ is \(1/\gamma_1\)-strongly convex. Therefore, it has a unique minimizer, denoted by \(\theta^*=\theta^*(x,y,z)\), and the first-order optimality condition gives $0= \nabla_y\mathcal L(x,\theta^*,z) +(\theta^*-y)/\gamma_1$.
		
	We first prove the error bound. Since the above strongly convex function has a \(1/\gamma_1\)-strongly monotone subdifferential, we obtain
	\[
	\frac{1}{\gamma_1}\|\theta-\theta^*\|^2
	\le
	\left\langle \nabla_y\mathcal L(x,\theta,z)+\frac{\theta-y}{\gamma_1},\theta-\theta^*\right\rangle
	\le
	\left\|\nabla_y\mathcal L(x,\theta,z)+\frac{\theta-y}{\gamma_1}\right\| \|\theta-\theta^*\|.
	\]
	If \(\theta=\theta^*\), the desired error bound holds trivially. Otherwise, dividing the preceding inequality by \(\|\theta-\theta^*\|\) gives $\|\theta-\theta^*\|\le \gamma_1\|\nabla_y\mathcal L(x,\theta,z)+(\theta-y)/\gamma_1\|$.
	
	It remains to prove the value estimate. By the \(1/\gamma_1\)-strong convexity of $\mathcal L(x,\theta,z)+\frac{1}{2\gamma_1}\|\theta-y\|^2$, we have
	\[
	\begin{aligned}
		&\quad \mathcal L(x,\theta,z)
		+\frac{1}{2\gamma_1}\|\theta-y\|^2
		-
		\mathcal L(x,\theta^*,z)
		-\frac{1}{2\gamma_1}\|\theta^*-y\|^2\\
		&\le
		\left\langle \nabla_y\mathcal L(x,\theta,z)+\frac{\theta-y}{\gamma_1},\theta-\theta^*\right\rangle
		-
		\frac{1}{2\gamma_1}\|\theta-\theta^*\|^2\\
		&\le
		\left\|\nabla_y\mathcal L(x,\theta,z)+\frac{\theta-y}{\gamma_1}\right\| \|\theta-\theta^*\|
		-
		\frac{1}{2\gamma_1}\|\theta-\theta^*\|^2\\
		&\le
		\frac{\gamma_1}{2}\left\|\nabla_y\mathcal L(x,\theta,z)+\frac{\theta-y}{\gamma_1}\right\|^2,
	\end{aligned}
	\]
	where the second inequality follows from the Cauchy--Schwarz inequality, and the last inequality follows from the elementary estimate \(ab-a^2/(2\gamma_1)\le \gamma_1 b^2/2\). Using the definition of \(M_{\mathcal L}^{\gamma_1}(x,y,z)\) yields
	\[
	\mathcal L(x,\theta,z)
	+\frac{1}{2\gamma_1}\|\theta-y\|^2
	-M_{\mathcal L}^{\gamma_1}(x,y,z)
	\le
	\frac{\gamma_1}{2}
	\left\|
	\nabla_y\mathcal L(x,\theta,z)
	+\frac{\theta-y}{\gamma_1}
	\right\|^2.
	\]
	The desired estimate follows again from
	\(\nabla_y\mathcal L(x,\theta,z)=\nabla_y f(x,\theta)+\nabla_y g(x,\theta)^\top z\).
\end{proof}

\subsection{MPCC Reformulation, Constraint Qualifications, and Stationarity Notions}
This subsection recalls the MPCC structure of the KKT reformulation introduced in Section~\ref{sec1}. We state the complementarity geometry and index sets, the MPCC constraint qualification used below, and the MPCC stationarity notions.

The complementarity condition \(0\le z \perp -g(x,y)\ge0\) in \eqref{eq:kkt-reform} can be written as \((z,-g(x,y))\in\mathcal C\), where
\begin{equation}\label{def_C}
\mathcal C:=\{(z,s)\in\mathbb R_+^q\times\mathbb R_+^q\mid z_i s_i=0,\ i=1,\ldots,q\}
\end{equation}
is the standard complementarity set. For any \((\bar z,\bar s)\in\mathcal C\), we define the complementarity index sets by
\[
\mathcal I_{+0}(\bar z,\bar s):=\{i\mid \bar z_i>0,\ \bar s_i=0\},\quad
\mathcal I_{0+}(\bar z,\bar s):=\{i\mid \bar z_i=0,\ \bar s_i>0\},\quad
\mathcal I_{00}(\bar z,\bar s):=\{i\mid \bar z_i=0,\ \bar s_i=0\}.
\]
With a slight abuse of notation, when \((\bar z,\bar s)=(\bar z,-g(\bar x,\bar y))\) is induced by a feasible point \(\bar w:=(\bar x,\bar y,\bar z)\) of the MPCC reformulation \eqref{eq:kkt-reform}, we write
\begin{equation}\label{MPCC_index_set}
	\begin{aligned}
		\mathcal I_{+0}(\bar w)&:=\{i\mid \bar z_i>0,\ -g_i(\bar x,\bar y)=0\},\ \mathcal I_{0+}(\bar w):=\{i\mid \bar z_i=0,\ -g_i(\bar x,\bar y)>0\},\\
		\mathcal I_{00}(\bar w)&:=\{i\mid \bar z_i=0,\ -g_i(\bar x,\bar y)=0\}.
	\end{aligned}
\end{equation}
We also define the active upper- and lower-level constraint index sets by
\[
I_G(\bar x,\bar y):=\{j\mid G_j(\bar x,\bar y)=0\},\qquad
I_g(\bar x,\bar y):=\{i\mid g_i(\bar x,\bar y)=0\}.
\]

The limiting normal-cone formula for \(\mathcal C\) follows from the formula for \(\operatorname{gph}N_{\mathbb R_+^q}\) in \cite[Proposition~3.7]{ye2000constraint}.

\begin{lemma}\label{lem:normal_cone_complementarity}
	For any \((\bar z,\bar s)\in\mathcal C\), the limiting normal cone to \(\mathcal C\) at \((\bar z,\bar s)\) is given by
	\[
	\mathcal N_{\mathcal C}(\bar z,\bar s)
	=
	\left\{
	(\xi^z,\xi^s)\in\mathbb R^{q}\times \mathbb R^{q}
	\ \middle|\
	\xi^z_{\mathcal I_{+0}}=0,\ \xi^s_{\mathcal I_{0+}}=0,\ 
	\left[
	(\xi^z_i<0,\ \xi^s_i<0)\ \text{or}\ \xi^z_i\xi^s_i=0
	\right]\ \forall\ i\in\mathcal I_{00}
	\right\},
	\]
	where $\mathcal I_{+0}:=\mathcal I_{+0}(\bar z,\bar s)$, $\mathcal I_{0+}:=\mathcal I_{0+}(\bar z,\bar s)$, $\mathcal I_{00}:=\mathcal I_{00}(\bar z,\bar s)$ .
\end{lemma}

We next recall the constraint qualifications used for the KKT reformulation. Since standard nonlinear-programming constraint qualifications fail at feasible points of complementarity systems, MPCC-tailored constraint qualifications are needed. We shall use Mangasarian--Fromovitz constraint qualification (MFCQ) in its no-nonzero-abnormal-multiplier form and its local stability property.

Let \(g:\mathbb R^d\to\mathbb R^p,\ h:\mathbb R^d\to\mathbb R^q\) be continuously differentiable, and define a smooth equality--inequality system $\{w\mid h(w)=0,\ g(w)\le0\}$, denote the active inequality index set by \(I_g(\bar w):=\{i\mid g_i(\bar w)=0\}\). Since the subsequent analysis is formulated in terms of multiplier sequences and normal-cone inclusions, we use the following no-nonzero-abnormal-multiplier form of MFCQ.

\begin{definition}[MFCQ]\label{def:mfcq_general}
	Let \(\bar w\in\{w\mid h(w)=0,\ g(w)\le0\}\). We say that the Mangasarian--Fromovitz constraint qualification holds at \(\bar w\) for the system \(\{w\mid h(w)=0,\ g(w)\le0\}\), if
	\[
	\mu^h\in\mathbb R^q,\ 
	\mu^g_i\ge0,\ i\in I_g(\bar w),\ 
	\sum_{j=1}^q\mu^h_j\nabla h_j(\bar w)+\sum_{i\in I_g(\bar w)}\mu^g_i\nabla g_i(\bar w)=0
	\Longrightarrow
	\mu^h=0,\  \mu^g_i=0,\ i\in I_g(\bar w).
	\]
\end{definition}

The local stability of MFCQ is a standard consequence of the stability of Robinson's constraint qualification and its equivalence to MFCQ for smooth equality--inequality systems; see Bonnans and Shapiro~\cite[Remark~2.88 and Corollary~2.101]{bonnans2000perturbation}.

\begin{lemma}\label{lem:mfcq_stability}
	Let \(\bar w\in\{w\mid h(w)=0,\ g(w)\le0\}\), and MFCQ holds at \(\bar w\) for the system \(\{w\mid h(w)=0,\ g(w)\le0\}\), then there exists a neighborhood \(U\) of \(\bar w\) such that MFCQ holds at every feasible point \(\hat{w}\in\{w\mid h(w)=0,\ g(w)\le0\}\cap U\).
\end{lemma}

We now apply this constraint qualification to the MPCC reformulation. The relevant nonlinear program is the tightened nonlinear program (TNLP) obtained by fixing the local complementarity structure at the reference point \(\bar w\):
\begin{equation}\label{eq:tnlp}
	\begin{aligned}
		\min_{x \in\mathbb R^n,\,y\in \mathbb R^m,\,z\in\mathbb R^{q}} & F(x,y) \\
		\mathrm{s.t.}\quad\ \ \quad & G(x,y) \le 0, \\
		& \nabla_y{\mathcal L}(x,y,z)=0, \\
		& -g_i(x,y) = 0 \;\; (i \in \mathcal{I}_{+0}(\bar w) \cup \mathcal{I}_{00}(\bar w)),
		\quad z_i = 0 \;\; (i \in \mathcal{I}_{0+}(\bar w) \cup \mathcal{I}_{00}(\bar w)),\\
		& -g_i(x,y) \ge 0 \;\; (i \in \mathcal{I}_{0+}(\bar w)),
		\quad z_i \ge 0 \;\; (i \in \mathcal{I}_{+0}(\bar w)).
	\end{aligned}
\end{equation}

We now introduce MPCC-MFCQ for the KKT reformulation, namely, MFCQ applied to the TNLP \eqref{eq:tnlp} \cite{scheel2000mathematical,kanzow2013new}.

\begin{definition}[MPCC-MFCQ]\label{def:mpcc_mfcq}
	Let \(\bar w:=(\bar x,\bar y,\bar z)\) be a feasible point of the MPCC reformulation \eqref{eq:kkt-reform}. We say that MPCC-MFCQ holds at \(\bar w\) if the MFCQ holds at \(\bar w\) for the TNLP \eqref{eq:tnlp} defined by \(\bar w\).
\end{definition}

By the no-nonzero-abnormal-multiplier characterization of MFCQ, MPCC-MFCQ excludes nonzero abnormal multiplier systems for the TNLP \eqref{eq:tnlp}; see, e.g., \cite{ye2020constraint}. We use the following explicit multiplier form to rule out normalized limits of unbounded multiplier sequences. 

\begin{lemma}[No abnormal multiplier under MPCC-MFCQ]\label{lem:mpcc_mfcq_no_abnormal_multiplier}
	Let \(\bar w:=(\bar x,\bar y,\bar z)\) be a feasible point of the MPCC reformulation \eqref{eq:kkt-reform}. Suppose that MPCC-MFCQ holds at \(\bar w\). Define
	\[
	I_G:=\{i\mid G_i(\bar x,\bar y)=0\},\quad
	I_g:=\{i\mid g_i(\bar x,\bar y)=0\},\quad
	I_z:=\{i\mid \bar z_i=0\}.
	\]
	If multipliers $\mu_i\ge0\ (i\in I_G)$, $\omega\in\mathbb R^m$, $\eta_i\in\mathbb R\ (i\in I_g)$, $\nu_i\in\mathbb R\ (i\in I_z)$ satisfy
	\begin{equation}\label{eq:mpcc_mfcq_xy_component_system}
		0=
		\nabla_y\nabla_{x,y}\mathcal L(\bar x,\bar y,\bar z)\omega
		+\sum_{i\in I_g}\eta_i\nabla g_i(\bar x,\bar y)
		+\sum_{i\in I_G}\mu_i\nabla G_i(\bar x,\bar y),
	\end{equation}
	and
	\begin{equation}\label{eq:mpcc_mfcq_y_gradient_component_system}
		\nabla_y g_i(\bar x,\bar y)^\top\omega=0\quad (i\in I_g\setminus I_z),
		\qquad
		\nabla_y g_i(\bar x,\bar y)^\top\omega-\nu_i=0\quad (i\in I_z),
	\end{equation}
	then
	\[
	\mu_i=0\ (i\in I_G),\qquad
	\omega=0,\qquad
	\eta_i=0\ (i\in I_g),\qquad
	\nu_i=0\ (i\in I_z).
	\]
	Consequently, no nonzero multiplier tuple satisfying \eqref{eq:mpcc_mfcq_xy_component_system}--\eqref{eq:mpcc_mfcq_y_gradient_component_system} can exist under MPCC-MFCQ.
\end{lemma}

We recall the MPCC stationarity notions used below. The three notions considered below share the same basic stationarity system and differ in the sign and complementarity conditions imposed on the multipliers for the biactive index set \(\mathcal I_{00}(\bar{w})\). We use the standard concepts of strong (S-), Mordukhovich (M-), and Clarke (C-) stationarity, following \cite{scheel2000mathematical,jane2005necessary,flegel2005m,kanzow2013new}.

\begin{definition}[MPCC stationarity]\label{def:stationarity}
	Let \(\bar w:=(\bar x,\bar y,\bar z)\) be a feasible point of the MPCC reformulation \eqref{eq:kkt-reform}. We say that \(\bar w\) is \emph{weakly stationary} if there exist multipliers \(\mu\in\mathbb R_+^p\), \(\omega\in\mathbb R^m\), and \(\eta\in\mathbb R^q\) such that
	\begin{equation}\label{eq:mpcc-stationarity}
		\begin{aligned}
			0&=
			\nabla F(\bar x,\bar y)
			+\nabla G(\bar x,\bar y)^\top\mu
			+\nabla_y\nabla_{x,y}\mathcal L(\bar x,\bar y,\bar z)\omega
			+\nabla g(\bar x,\bar y)^\top\eta,\\
			\mu &\ge 0, \quad G(\bar{x},\bar{y}) \le 0, \quad \mu_jG_j(\bar x,\bar y)=0,\ j=1,\ldots,p, \\
			\eta_i &= 0 \quad  \forall\ i \in \mathcal{I}_{0+}(\bar w), \qquad
			\nabla_y g_i(\bar x,\bar y)^\top\omega=0 \quad  \forall\ i \in \mathcal{I}_{+0}(\bar w),
		\end{aligned}
	\end{equation}	
	We say that \(\bar w\) is \emph{strongly stationary} (S-stationary), \emph{Mordukhovich-stationary} (M-stationary), and \emph{Clarke-stationary} (C-stationary), respectively, if there exist \(\mu\in\mathbb R_+^p\), \(\omega\in\mathbb R^m\), and \(\eta\in\mathbb R^q\) satisfying \eqref{eq:mpcc-stationarity} and, for every \(i\in\mathcal I_{00}(\bar w)\), the condition below holds:
	\begin{equation*}
		\begin{aligned}
			&\eta_i\ge0,\quad
			\nabla_y g_i(\bar x,\bar y)^\top\omega\ge0,\\
			&\text{either }
			\eta_i>0\ \text{and}\
			\nabla_y g_i(\bar x,\bar y)^\top\omega>0,
			\quad\text{or}\quad
			\eta_i\nabla_y g_i(\bar x,\bar y)^\top\omega=0,\\
			&\eta_i\nabla_y g_i(\bar x,\bar y)^\top\omega\ge0,
		\end{aligned}
	\end{equation*}
	respectively.
\end{definition}

In the TNLP multiplier system, the multiplier \(\nu_i\) corresponds to the fixed constraint \(z_i=0\). In the stationarity notions above, this multiplier is eliminated through the \(z\)-component stationarity equation, namely $\nu_i=\nabla_y g_i(\bar x,\bar y)^\top\omega$, for any $i\in \mathcal I_{0+}(\bar w)\cup \mathcal I_{00}(\bar w)$, which is precisely the zero-\(z\) index set.

The above stationarity notions differ only in their treatment of the biactive index set \(\mathcal I_{00}\), and satisfy 
\[
\text{S-stationarity} \quad\Longrightarrow\quad \text{M-stationarity} \quad\Longrightarrow\quad \text{C-stationarity}.
\]
However, S-stationarity is not a necessary optimality condition for MPCCs in general; it usually requires strong MPCC-tailored constraint qualifications, such as MPCC-LICQ. M-stationarity is weaker than S-stationarity but is often the sharper first-order necessary condition available under weaker constraint qualifications. In contrast, C-stationarity is often too weak to exclude certain first-order descent directions at biactive indices; see, e.g., \cite{scheel2000mathematical,outrata1999optimality,jane2005necessary,flegel2005m,kanzow2013new}.

\section{Approximate KKT Condition and Convergence to \\ C-Stationary Points}\label{sec3}
In this section, we establish limiting stationarity properties for approximate first-order points of the regularized gap-function reformulation \eqref{reformualtion_GP}. We introduce an approximate KKT condition for \eqref{reformualtion_GP} and relate points satisfying this condition to C-stationarity of the KKT reformulation \eqref{eq:kkt-reform}.
Since \(\mathcal G_\gamma(x,y,z)\ge0\), two natural approximation schemes for \eqref{reformualtion_GP} are the penalty formulation
\begin{equation}\label{eq:penalized}
	\min_{(x,y,z) \in \Omega \times \mathbb R_+^{q} } \quad F(x, y) + \rho \mathcal G_\gamma(x,y,z).
\end{equation}
where \(\rho\) is the penalty parameter, and the relaxed formulation
\begin{equation}\label{eq:relax}
	\min_{(x,y,z)\in\Omega\times\mathbb R_+^{q}}
	\quad F(x,y)
	\quad
	\text{s.t.}\quad
	\mathcal G_\gamma(x,y,z)\le s_\mathcal G,
\end{equation}
where \(s_\mathcal G>0\) is the relaxation parameter.

Because these approximation problems are generally nonconvex, one typically obtains only approximate stationary points. We therefore study accumulation points generated as the penalty parameter $\rho \to \infty$ or the relaxation parameter $s_\mathcal G \to 0$. The common first-order form of \eqref{eq:penalized} and \eqref{eq:relax} leads to the following approximate KKT condition for the gap-function reformulation \eqref{reformualtion_GP}.

\begin{definition}
	We say that a point $(\bar{x}, \bar{y})$ satisfies the approximate KKT condition for (GP) in \eqref{reformualtion_GP} if there exist sequences $\{(x^k, y^k, z^k)\} \subset \mathbb{R}^{n+m+q}$, $\{ \varepsilon^k \} \subset \mathbb{R}^{n+m+q}$, and $\{\rho^k\} \subset \mathbb{R}_{++}$ such that
	\begin{equation}\label{approximate_kkt}
		\begin{aligned}
			&\lim_{k\to\infty} (x^k, y^k) = (\bar{x}, \bar{y}), \quad \lim_{k\to\infty} \varepsilon^k = 0, \quad (x^k,y^k,z^k) \in \Omega \times \mathbb{R}_+^{q}, \quad \lim_{k\to\infty} \mathcal{G}_\gamma(x^k,y^k,z^k) = 0, \\
			&\varepsilon^k \in \nabla F(x^k,y^k) + \rho^k\nabla \mathcal{G}_\gamma(x^k,y^k,z^k) + \mathcal{N}_{\Omega \times \mathbb{R}_+^{q}}(x^k,y^k,z^k),
		\end{aligned}
	\end{equation}
	where, for notational simplicity, the gradient \(\nabla F(x^k,y^k)\) is embedded into the \((x,y,z)\)-space by assigning zero entries to the components of absent variables.
\end{definition}

Given an approximate KKT point $(\bar{x},\bar{y})$ of (GP), our objective is to characterize its corresponding stationarity properties. If the multiplier sequence $\{\rho^k\}$ for $\mathcal{G}_\gamma$ contains a bounded subsequence, the conclusion follows directly by passing to the limit in \eqref{approximate_kkt}. Under standard continuity assumptions on the problem data and their derivatives, this yields a classical KKT-type stationarity condition for the reformulation \eqref{reformualtion_GP}. The remaining case is $\rho^k \to \infty$.
In this unbounded-multiplier case, we analyze the limiting stationarity condition via the MPCC reformulation \eqref{eq:kkt-reform}.

Throughout this section, we use the shorthand \(w=(x,y,z)\). For a feasible point \(\bar w=(\bar x,\bar y,\bar z)\) of the MPCC reformulation \eqref{eq:kkt-reform}, we use the MPCC index-set notation introduced in \eqref{MPCC_index_set}. In addition, for any \(z\in\mathbb R_+^q\), define
\[
I_+(z):=\{i\mid z_i>0\},
\quad
I_0(z):=\{i\mid z_i=0\}.
\]

The next lemma gives the estimates used to compare the approximate KKT condition for (GP) with the stationarity system of the MPCC reformulation \eqref{eq:kkt-reform}.

\begin{lemma}\label{lem:approximate_kkt_normal_cone_free}
	Suppose that $(\bar{x}, \bar{y})$ satisfies the approximate KKT condition for $\mathrm{(GP)}$. Let\linebreak[4] $\{(x^k, y^k, z^k)\}\subset \mathbb{R}^{n+m+q}$, $\{ \varepsilon^k \} \subset \mathbb{R}^{n+m+q}$, and $\{\rho^k\} \subset \mathbb{R}_{++}$ satisfy \eqref{approximate_kkt}. If the MFCQ holds at $\bar{y}$ for the lower-level constraint system $\{ y \mid g(\bar{x}, y) \le 0 \}$, then $\{z^k\}$ is bounded.
	
	Furthermore, if $\bar{z} \in \mathbb{R}_+^q$ is an accumulation point of the sequence $\{z^k\}$, then $\bar{z}_i g_i(\bar{x}, \bar{y}) = 0$ for all $i = 1, \ldots , q$, and $\bar{w} := (\bar{x}, \bar{y}, \bar{z})$ is a feasible point for both $\mathrm{(GP)}$ in \eqref{reformualtion_GP} and the MPCC reformulation in \eqref{eq:kkt-reform}.
	
	Suppose, in addition, that the MFCQ holds at $(\bar{x}, \bar{y})$ for the upper-level inequality system $\{ (x,y) \mid G(x,y) \le 0 \}$. By partitioning $\varepsilon^k = (\varepsilon_{x,y}^k, \varepsilon_{z}^k)$ with $\varepsilon_{x,y}^k \in \mathbb{R}^{n+m}$ and $\varepsilon_z^k \in \mathbb{R}^q$, there exists a multiplier sequence $\{\mu^k\} \subset \mathbb{R}_+^p$ such that $\mu_i^k G_i(x^k,y^k) = 0$ for $i=1,\ldots,p$, and
	\begin{equation}\label{eq:old-gap-gradient-combined}
		\varepsilon_{x,y}^k + \rho^k r^k = \nabla F(x^k,y^k) + \nabla G(x^k,y^k)^\top \mu^k + \nabla_y\nabla_{x,y}{\mathcal{L}}(x^k,y^k,z^k)\omega^k + \nabla g(x^k,y^k)^\top \eta^k,
	\end{equation}
	where
	\[
	\lambda_k^* := \lambda^*(x^k,y^k, z^k), \quad \theta_k^* := \theta^*(x^k,y^k, z^k), \quad  \omega^k := -\rho^k(\theta_k^* - y^k), \quad \eta^k := \rho^k(\lambda_k^* - z^k),
	\]
	and $r^k \in \mathbb{R}^{n+m}$ is a residual vector satisfying \(r^k=0\) whenever \(\theta_k^*=y^k\), and \(\|r^k\|/|\theta_k^*-y^k\|\to0\) otherwise, as $k \to \infty$. Moreover, the following conditions hold:
	\begin{equation}\label{eq:approximate_kkt_opened_z}
		\begin{aligned}
			g_i(x^k,\theta_k^*) &= \frac{\lambda_{k,i}^* - z_i^k}{\gamma_2} - \frac{\varepsilon_{z,i}^k}{\rho^k}, \quad &&i \in I_+(z^k),\\
			g_i(x^k,\theta_k^*) &\le \frac{\lambda_{k,i}^* - z_i^k}{\gamma_2} - \frac{\varepsilon_{z,i}^k}{\rho^k}, \quad &&i \in I_0(z^k),
		\end{aligned}
	\end{equation}
	where $\lambda_{k,i}^*$ and $\varepsilon_{z,i}^k$ denote the $i$-th components of the vectors $\lambda_k^*$ and $\varepsilon_z^k$, respectively.
\end{lemma}

\begin{proof}
	We first establish the boundedness of the sequence $\{z^k\}$. Suppose, for the sake of contradiction, that $\{z^k\}$ is unbounded, i.e., $\|z^k\| \to \infty$. We may pass to a subsequence, if necessary, to assume that 
	\[
	\frac{z^k}{\|z^k\|} \to \hat{z}, \quad  \hat{z} \ge 0, \quad  \|\hat{z}\| = 1.
	\]
	By the first-order optimality condition for the subproblem \eqref{eq:deftheta} defining $\theta_k^* := \theta^*(x^k,y^k,z^k)$, we obtain
	\begin{equation}\label{eq:stationary-condition-theta}
		0 = \nabla_y f(x^k,\theta_k^*) + \nabla_y g(x^k,\theta_k^*)^\top z^k + \frac{1}{\gamma_1}(\theta_k^* - y^k).
	\end{equation}
	Since $\lim_{k\to\infty}\mathcal{G}_\gamma(x^k,y^k,z^k) = 0$, Lemma \ref{lem:vanishing_gap_residuals} forces $\theta_k^* - y^k \to 0$. Combining with $y^k \to \bar{y}$ yields $\theta_k^* \to \bar{y}$. Dividing \eqref{eq:stationary-condition-theta} by $\|z^k\|$, passing to the limit as $k \to \infty$, and invoking the continuity of $\nabla_y f$ and $\nabla_y g$, we deduce
	\begin{equation}\label{eq:zero-normal-g}
		\nabla_y g(\bar{x},\bar{y})^\top \hat{z} = 0.
	\end{equation}
	
	Next, we demonstrate that \eqref{eq:zero-normal-g} is incompatible with the MFCQ for the lower-level constraint system $\{ y \mid g(\bar{x}, y) \le 0 \}$. By Lemma~\ref{lem:vanishing_gap_residuals}, we have \(\lambda_k^*-z^k\to0\). Equivalently, for every $i = 1, \dots, q$,
	\begin{equation}\label{eq:limit-z}
		\mathrm{Proj}_{\mathbb{R}_+}\left( z_i^k + \gamma_2 g_i(x^k,y^k) \right) - z_i^k \to 0.
	\end{equation}
	For any fixed index $i \notin I_g(\bar{x},\bar{y})$, we have $g_i(\bar{x},\bar{y}) < 0$. By continuity, $g_i(x^k,y^k) < 0$ for all sufficiently large $k$. The limit \eqref{eq:limit-z} forces $z_i^k \to 0$, and consequently $\hat{z}_i = 0$ for every $i \notin I_g(\bar{x},\bar{y})$. Therefore, \eqref{eq:zero-normal-g} reduces to
	\[
	\sum_{i \in I_g(\bar{x},\bar{y})} \hat{z}_i \nabla_y g_i(\bar{x},\bar{y}) = 0.
	\]
	Because $\hat{z} \ge 0$, $\hat{z}_i = 0$ for all $i \notin I_g(\bar{x},\bar{y})$, and $\|\hat{z}\| = 1 \neq 0$, this contradicts the MFCQ for the lower-level constraint system $g(\bar{x},\cdot) \le 0$ at $\bar{y}$. Hence, the sequence $\{z^k\}$ must be bounded.

	Now let \(\bar z\in\mathbb R_+^q\) be an arbitrary accumulation point of \(\{z^k\}\). Passing to a subsequence if necessary, we may assume that \(z^k\to\bar z\). Since \((x^k,y^k)\in\Omega\), \(z^k\in\mathbb R_+^q\), and \((x^k,y^k,z^k)\to(\bar x,\bar y,\bar z)\), the closedness of \(\Omega\) and \(\mathbb R_+^q\) yields \((\bar x,\bar y)\in\Omega\) and \(\bar z\in\mathbb R_+^q\). In particular, this guarantees $G(\bar{x},\bar{y}) \le 0$. Moreover, by the continuity of \(\mathcal G_\gamma\),
	\[
	\mathcal G_\gamma(\bar x,\bar y,\bar z)
	=
	\lim_{k\to\infty}
	\mathcal G_\gamma(x^k,y^k,z^k)
	=0.
	\]
	Therefore, \(\bar w:=(\bar x,\bar y,\bar z)\) is feasible for \(\mathrm{(GP)}\). By Lemma~\ref{lem:gap_nonnegative_exact}, we have $\bar y\in S(\bar x)$, and $\bar z\in\mathcal M(\bar x,\bar y)$. Consequently, the relations $g(\bar{x},\bar{y}) \le 0$, $\bar{z} \ge 0$, $\bar{z}^\top g(\bar{x},\bar{y}) = 0$, and
	\[
	\nabla_y f(\bar{x},\bar{y}) + \nabla_y g(\bar{x},\bar{y})^\top \bar{z} = \nabla_y \mathcal{L}(\bar{x},\bar{y},\bar{z}) = 0
	\]
	must hold. These are precisely the feasibility conditions of the MPCC reformulation \eqref{eq:kkt-reform} for $\bar{w} := (\bar{x},\bar{y},\bar{z})$.
	
	It remains to establish the multiplier representation. The $(x,y)$-component of the approximate KKT inclusion \eqref{approximate_kkt} yields
	\begin{equation}\label{eq:approximate_kkt_split}
		\varepsilon_{x,y}^k \in \nabla F(x^k,y^k) + \rho^k\nabla_{x,y}\mathcal{G}_\gamma(x^k,y^k,z^k) + \mathcal{N}_{\Omega}(x^k,y^k).
	\end{equation}
	Since \(G\) is continuously differentiable and MFCQ holds at \((\bar x,\bar y)\) for the upper-level inequality system $G(x,y) \le 0$, Lemma~\ref{lem:mfcq_stability} implies that the MFCQ also holds for this system at $(x^k,y^k)$ for all sufficiently large $k$. Therefore, the standard normal-cone representation for smooth inequality systems applies (see, e.g., \cite[Theorem 6.14 and Example 6.40]{rockafellar1998variational}):
	\[
	\mathcal{N}_{\Omega}(x^k,y^k) = \left\{ \nabla G(x^k,y^k)^\top \mu \ \middle|\ \mu \in \mathbb{R}_+^p,\ \mu_i G_i(x^k,y^k) = 0,\ i=1,\ldots,p \right\}.
	\]
	Substituting this representation into \eqref{eq:approximate_kkt_split}, we obtain a multiplier sequence $\{\mu^k\} \subset \mathbb{R}_+^p$ satisfying
	\begin{equation}\label{eq:approximate_kkt_opened_xy}
		\mu_i^k G_i(x^k,y^k) = 0, i=1,\ldots,p, \quad \text{and} \quad \varepsilon_{x,y}^k = \nabla F(x^k,y^k) + \nabla G(x^k,y^k)^\top \mu^k + \rho^k\nabla_{x,y}\mathcal{G}_\gamma(x^k,y^k,z^k).
	\end{equation}
	
	Applying the gradient formula for $\mathcal{G}_\gamma$ from \eqref{Ggradient}, evaluated at $(x^k,y^k,z^k)$, we have
	\begin{equation}\label{eq:grad_xy-G}
		\begin{aligned}
			\nabla_{x,y} \mathcal{G}_\gamma(x^k,y^k,z^k)
			&= \begin{pmatrix}
				\nabla_x f(x^k,y^k) + \nabla_x g(x^k,y^k)^\top \lambda_k^* \\
				\nabla_y f(x^k,y^k) + \nabla_y g(x^k,y^k)^\top \lambda_k^* \end{pmatrix} -
			\begin{pmatrix}
				\nabla_x f(x^k,\theta_k^*) + \nabla_x g(x^k,\theta_k^*)^\top z^k \\
				(y^k - \theta_k^*) / \gamma_1
			\end{pmatrix} \\
			&= \begin{pmatrix}
				\nabla_x f(x^k,y^k) + \nabla_x g(x^k,y^k)^\top \lambda_k^* \\
				\nabla_y f(x^k,y^k) + \nabla_y g(x^k,y^k)^\top \lambda_k^* \end{pmatrix} -
			\begin{pmatrix}
				\nabla_x f(x^k,\theta_k^*) + \nabla_x g(x^k,\theta_k^*)^\top z^k \\
				\nabla_y f(x^k,\theta_k^*) + \nabla_y g(x^k,\theta_k^*)^\top z^k
			\end{pmatrix} \\
			&= \nabla g(x^k,y^k)^\top(\lambda_k^* - z^k) - \left(\nabla_{x,y}\mathcal{L}(x^k,\theta_k^*,z^k) - \nabla_{x,y}\mathcal{L}(x^k,y^k,z^k)\right),
		\end{aligned}
	\end{equation}
	where the second equality follows from the first-order stationarity condition for $\theta_k^*$ given in \eqref{eq:stationary-condition-theta}.
	
	Because $x^k \in X$, $(x^k, y^k) \to (\bar{x}, \bar{y})$, $\theta_k^* \to \bar{y}$, and the sequence $\{z^k\}$ is bounded, all points under consideration reside within a bounded neighborhood. Moreover, by Assumption~\ref{asup_stationarity_lower}, the second derivatives of $f$ and $g_i$ ($i=1,\ldots,q$) are continuous around $(x^k, y^k)$. Applying Taylor's theorem at $y^k$ to the mapping $y \mapsto \nabla_{x,y}\mathcal{L}(x^k, y, z^k)$ yields a residual vector $r^k \in \mathbb{R}^{n+m}$ satisfying \(r^k=0\) whenever \(\theta_k^*=y^k\), and \(\|r^k\|/|\theta_k^*-y^k\|\to0\) otherwise, as $k \to \infty$ and
	\[
	\nabla_{x,y}\mathcal{L}(x^k,\theta_k^*,z^k) - \nabla_{x,y}\mathcal{L}(x^k,y^k,z^k) = \nabla_y\nabla_{x,y}\mathcal{L}(x^k,y^k,z^k)(\theta_k^* - y^k) + r^k.
	\]
	Substituting this Taylor expansion into \eqref{eq:grad_xy-G} produces
	\[
	\nabla_{x,y}\mathcal{G}_\gamma(x^k,y^k,z^k) = \nabla g(x^k,y^k)^\top(\lambda_k^* - z^k) - \nabla_y\nabla_{x,y}\mathcal{L}(x^k,y^k, z^k)(\theta_k^* - y^k) - r^k.
	\]
	Multiplying this relation by $\rho^k$ and utilizing the definitions
	\[
	\omega^k := -\rho^k(\theta_k^* - y^k), \quad \eta^k := \rho^k(\lambda_k^* - z^k),
	\]
	we obtain
	\begin{equation}\label{lem2_eq:old-gap-gradient-combined}
		\rho^k\nabla_{x,y}\mathcal{G}_\gamma(x^k,y^k,z^k) = \nabla_y\nabla_{x,y}\mathcal{L}(x^k,y^k,z^k)\omega^k + \nabla g(x^k,y^k)^\top \eta^k - \rho^k r^k.
	\end{equation}
	Substituting \eqref{lem2_eq:old-gap-gradient-combined} into \eqref{eq:approximate_kkt_opened_xy} yields the $(x,y)$-component of the multiplier representation:
	\[
	\varepsilon_{x,y}^k + \rho^k r^k = \nabla F(x^k,y^k) + \nabla G(x^k,y^k)^\top \mu^k + \nabla_y\nabla_{x,y}\mathcal{L}(x^k,y^k,z^k)\omega^k + \nabla g(x^k,y^k)^\top \eta^k.
	\]
	
	Next, we derive the componentwise relations for the $z$-variables. From the gradient formula for $\mathcal{G}_\gamma$ in \eqref{Ggradient},
	\[
	\nabla_z\mathcal{G}_\gamma(x^k,y^k,z^k) = \frac{\lambda_k^* - z^k}{\gamma_2} - g(x^k,\theta_k^*).
	\]
	Combining this identity with the $z$-component of \eqref{approximate_kkt}, namely
	\[
	\varepsilon_z^k \in \rho^k\nabla_z\mathcal{G}_\gamma(x^k,y^k,z^k) + \mathcal{N}_{\mathbb{R}_+^q}(z^k),
	\]
	and applying the definition of the normal cone to the non-negative orthant $\mathbb{R}_+^q$, we deduce the following alternatives for each index $i$: if $z_i^k > 0$, then $\varepsilon_{z,i}^k - \rho^k\nabla_{z_i}\mathcal{G}_\gamma(x^k,y^k,z^k) = 0$; if $z_i^k = 0$, then $\varepsilon_{z,i}^k - \rho^k\nabla_{z_i}\mathcal{G}_\gamma(x^k,y^k,z^k) \le 0$. Rearranging these terms algebraically yields
	\begin{equation*}
		g_i(x^k,\theta_k^*) = \frac{\lambda_{k,i}^* - z_i^k}{\gamma_2} - \frac{\varepsilon_{z,i}^k}{\rho^k}, \ \text{if } z_i^k > 0, \quad g_i(x^k,\theta_k^*) \le \frac{\lambda_{k,i}^* - z_i^k}{\gamma_2} - \frac{\varepsilon_{z,i}^k}{\rho^k}, \ \text{if } z_i^k = 0.
	\end{equation*}
	These conditions match the claims in the lemma exactly.
	This completes the proof.
\end{proof}

With the multiplier sequence constructed above, we now compare the limiting approximate KKT system with the stationarity system of the MPCC reformulation \eqref{eq:kkt-reform} under the additional MPCC-MFCQ assumption.

\begin{theorem}\label{thm:akkt_to_stationarity}
	Suppose that $(\bar{x},\bar{y})$ satisfies the approximate KKT condition for $\mathrm{(GP)}$. Let\linebreak[4] $\{(x^k, y^k, z^k)\} \subset \mathbb{R}^{n+m+q}$, $\{ \varepsilon^k \} \subset \mathbb{R}^{n+m+q}$, and $\{\rho^k\} \subset \mathbb{R}_{++}$ satisfy \eqref{approximate_kkt}. 
	Assume that MFCQ holds at $\bar{y}$ for the lower-level constraint system $\{ y \mid g(\bar{x}, y) \le 0 \}$, and that MFCQ holds at $(\bar{x}, \bar{y})$ for the upper-level inequality system $\{ (x,y) \mid G(x,y) \le 0 \}$. 
	Let $\bar{z} \in \mathbb{R}_+^q$ be any accumulation point of the sequence $\{z^k\}$.
	Then the following stationarity conditions hold for $\bar{w} := (\bar{x},\bar{y},\bar{z})$:
	\begin{enumerate}
		\item If $\{\rho^k\}$ is bounded, then $\bar{w}$ is S-stationary for \eqref{eq:kkt-reform}.
		\item If $\rho^k\to+\infty$, and we further assume that MPCC-MFCQ holds at $\bar{w}$ for the MPCC reformulation \eqref{eq:kkt-reform}, then $\bar{w}$ is C-stationary for \eqref{eq:kkt-reform}.
	\end{enumerate}
\end{theorem}

\begin{proof}
	Let $\{(x^k, y^k, z^k)\} \subset \mathbb{R}^{n+m+q}$, $\{ \varepsilon^k \} \subset \mathbb{R}^{n+m+q}$, and $\{\rho^k\} \subset \mathbb{R}_{++}$ satisfy \eqref{approximate_kkt}.
	As shown in Lemma~\ref{lem:approximate_kkt_normal_cone_free}, the sequence $\{z^k\}$ is bounded, and conditions \eqref{eq:old-gap-gradient-combined} and \eqref{eq:approximate_kkt_opened_z} hold with $\{\mu^k\}$ satisfying $\mu_i^k G_i(x^k,y^k) = 0$ for $ i=1,\ldots,p$.
	Let $\bar{z} \in \mathbb{R}_+^q$ be any accumulation point of the sequence $\{z^k\}$. Then $\bar{z}_ig_i(\bar{x}, \bar{y}) =0$ for $i = 1, \ldots , q$, and $\bar{w} := (\bar{x}, \bar{y}, \bar{z})$ is a feasible point for both $\mathrm{(GP)}$ in \eqref{reformualtion_GP} and the MPCC reformulation in \eqref{eq:kkt-reform}.
	In the remainder of the proof, by passing to a subsequence if necessary, we assume without loss of generality that $(x^k, y^k, z^k) \to \bar{w} := (\bar{x}, \bar{y}, \bar{z})$. Recall that $\omega^k := -\rho^k(\theta_k^* - y^k)$, $\eta^k := \rho^k(\lambda_k^* - z^k)$.
	
	Because $\lim_{k\to\infty}\mathcal{G}_\gamma(x^k,y^k,z^k) = 0$, it follows from Lemma~\ref{lem:vanishing_gap_residuals} that $\theta_k^* - y^k \to 0$ and $\lambda_k^* - z^k \to 0$, where we recall that 
	$
	\theta_k^*:=\theta^*(x^k,y^k,z^k)$, $
	\lambda_k^*:=\lambda^*(x^k,y^k,z^k) = \operatorname{Proj}_{\mathbb R_+^q}\left(z^k+\gamma_2g(x^k,y^k)\right).
	$
	
	\textbf{Case 1: $\{\rho^k\}$ is bounded.}
	
	Suppose that the sequence $\{\rho^k\}$ is bounded. Passing to a subsequence if necessary, we may assume that $\rho^k\to\bar{\rho}\ge0$. 
	
	We first show that the sequence $\{\mu^k\} \subset \mathbb{R}_+^p$ in condition \eqref{eq:old-gap-gradient-combined} is bounded. 
	Suppose, for the sake of contradiction, that $\{\mu^k\}$ is unbounded, meaning $\|\mu^k\|\to+\infty$. Passing to a subsequence, we may assume that
	\[
	\frac{\mu^k}{\|\mu^k\|}\to \hat{\mu}\neq0.
	\]
	For any \(i\notin I_G(\bar{x},\bar{y})\), we have \(G_i(\bar{x},\bar{y})<0\). Hence, by the continuity of \(G_i\), \(G_i(x^k,y^k)<0\) for all sufficiently large \(k\). The complementarity relation \(\mu_i^kG_i(x^k,y^k)=0\) then yields \(\mu_i^k=0\) eventually, and therefore \(\hat{\mu}_i=0\).
	
	Next, we consider the limits of the vectors $\omega^k$, $\eta^k$, and $\rho^kr^k$ in \eqref{eq:old-gap-gradient-combined}. Since \(\{\rho^k\}\) is bounded, while \(\theta_k^*-y^k\to0\) and \(\lambda_k^*-z^k\to0\), it follows that $
	\omega^k:=-\rho^k(\theta_k^*-y^k)\to0$, $
	\eta^k:=\rho^k(\lambda_k^*-z^k)\to0$.
	Moreover, \(\|r^k\|/\|\theta_k^*-y^k\|\to0\) or \(r^k=0\) imply \(r^k\to0\), and the boundedness of \(\{\rho^k\}\) gives \(\rho^kr^k\to0\). Dividing \eqref{eq:old-gap-gradient-combined} by \(\|\mu^k\|\), using \(\|\mu^k\|\to+\infty\), and passing to the limit, we obtain by the continuity of the derivatives
	\[
	0= \sum_{i\in I_G(\bar{x},\bar{y})} \hat{\mu}_i\nabla G_i(\bar{x},\bar{y}),
	\]
	where \(\hat{\mu}_i\ge0\) for \(i\in I_G(\bar{x},\bar{y})\) and \(\hat{\mu}\neq0\). This contradicts MFCQ at \((\bar{x},\bar{y})\) for the upper-level inequality system \(\{(x,y)\mid G(x,y)\le0\}\). Therefore, \(\{\mu^k\}\) is bounded.
	
	Since \(\{\mu^k\}\) is bounded, we may pass to a further subsequence, if necessary, and assume that \(\mu^k\to\bar{\mu}\in\mathbb R_+^p\). Passing the complementarity relation \(\mu_i^kG_i(x^k,y^k)=0\) to the limit yields
	\[
	\bar{\mu}_iG_i(\bar{x},\bar{y})=0,\qquad i=1,\ldots,p.
	\]
	Using again \(\omega^k\to0\), \(\eta^k\to0\), and \(\rho^kr^k\to0\), and then passing to the limit in \eqref{eq:old-gap-gradient-combined}, we obtain
	\[
	0= \nabla F(\bar{x},\bar{y}) + \nabla G(\bar{x},\bar{y})^\top\bar{\mu}.
	\]
	Consequently, the MPCC stationarity condition \eqref{eq:mpcc-stationarity} holds with $\bar{\omega}:=0$, $\bar{\eta}:=0$, and $\nabla_y g_i(\bar x,\bar y)^\top\bar{\omega}:=0$ for any $i=1,\cdots,q$. For any biactive index $i\in\mathcal I_{00}(\bar{w})$, we have $\bar{\eta}_i=0\ge0$ and $\nabla_y g_i(\bar x,\bar y)^\top\bar{\omega}=0\ge0$. Therefore, $\bar{w}$ is S-stationary for \eqref{eq:kkt-reform}. 
	
	\textbf{Case 2: $\{\rho^k\}$ is unbounded.}
	
	Next, we consider the case where $\rho^k \to \infty$. We first establish the following result regarding the multipliers.
	
	\textit{Claim 1: For all sufficiently large $k$, $\mu_i^k=0$ for $i \notin I_G(\bar{x},\bar{y})$ and $z_i^k=\eta_i^k=0$ for $i \notin I_g(\bar{x},\bar{y})$. Moreover, $\nu_i^k:=\rho^k\bigl(g_i(x^k,y^k)-g_i(x^k,\theta_k^*)\bigr) \to 0$ for $i \in \mathcal I_{+0}(\bar{w})$ as $k \to \infty$.}
	
	If $i\notin I_G(\bar{x},\bar{y})$, then $G_i(\bar{x},\bar{y})<0$. By the continuity of $G_i$ and the convergence $(x^k,y^k)\to(\bar{x},\bar{y})$, we have $G_i(x^k,y^k)<0$ for all sufficiently large $k$. Hence, the complementarity condition $\mu_i^kG_i(x^k,y^k)=0$ implies $\mu_i^k=0$ for all sufficiently large $k$.
	
	If $i\notin I_g(\bar{x},\bar{y})$, then $g_i(\bar{x},\bar{y})<0$. Since $\bar{z}_ig_i(\bar{x}, \bar{y}) =0$ for $i = 1, \ldots , q$, we must have $\bar{z}_i = 0$ for $i\notin I_g(\bar{x},\bar{y})$. Because $(x^k,y^k)\to(\bar{x},\bar{y})$ and $(x^k,\theta_k^*)\to(\bar{x},\bar{y})$, there exists $\delta_g>0$ such that for all $i\notin I_g(\bar{x},\bar{y})$ and all sufficiently large $k$, 
	\[
	g_i(x^k,y^k)\le-\delta_g, \quad \text{and} \quad g_i(x^k,\theta_k^*)\le-\delta_g.
	\]
	Fixing an index $i\notin I_g(\bar{x},\bar{y})$, since $z_i^k\to \bar{z}_i = 0$, we have for all sufficiently large $k$:
	\begin{equation}\label{eq:z+g-bound}
		z_i^k+\gamma_2 g_i(x^k,y^k) \le -\frac{\gamma_2\delta_g}{2}<0, \quad \text{and} \quad z_i^k+\gamma_2 g_i(x^k,\theta_k^*) \le -\frac{\gamma_2\delta_g}{2}<0,
	\end{equation}
	and therefore, $\lambda_{k,i}^*=\operatorname{Proj}_{\mathbb R_+}\bigl(z_i^k+\gamma_2 g_i(x^k,y^k)\bigr)=0$ for all sufficiently large $k$. 
	We next show that $z_i^k=0$ for all sufficiently large $k$. If, to the contrary, $z_i^{k}>0$ for some large $k$, the first line of \eqref{eq:approximate_kkt_opened_z} gives
	\[
	\frac{z_i^k}{\gamma_2} + g_i(x^k,\theta_k^*) = -\frac{\varepsilon_{z,i}^k}{\rho^k} \to 0,
	\]
	which contradicts \eqref{eq:z+g-bound}. Thus, $z_i^k=0$ and $\eta_i^k = \rho^k(\lambda_{k,i}^* - z_i^k)=0$ for all $i\notin I_g(\bar{x},\bar{y})$ and all sufficiently large $k$.
	
	Finally, if $i\in\mathcal I_{+0}(\bar{w})$, then $\bar{z}_i>0$ and $g_i(\bar{x},\bar{y})=0$. Hence, fixing any $i\in\mathcal I_{+0}(\bar{w})$, we observe $z_i^k>0$ for all sufficiently large $k$ and $g_i(x^k,y^k)\to0$. Then,
	\[
	\lambda_{k,i}^*= \operatorname{Proj}_{\mathbb R_+}\bigl(z_i^k+\gamma_2 g_i(x^k,y^k)\bigr) = z_i^k+\gamma_2g_i(x^k,y^k).
	\]
	Combining this with the first line of \eqref{eq:approximate_kkt_opened_z} gives
	\[
	\nu_i^k = \rho^k\bigl(g_i(x^k,y^k)-g_i(x^k,\theta_k^*)\bigr) = \frac{\rho^k}{\gamma_2} \left(\lambda_{k,i}^* - z_i^k\right) - \rho^k g_i(x^k,\theta_k^*) = \varepsilon_{z,i}^k.
	\]
	Thus, $\nu_i^k=\varepsilon_{z,i}^k\to0$ for every $i\in\mathcal I_{+0}(\bar{w})$.
	
	\textit{Claim 2: The multiplier sequences $\{\mu^k\}, \{\omega^k\}, \{\eta^k\}, \{\nu^k\}$ are bounded.}
	
	We now establish the boundedness of $\{\mu^k\}, \{\omega^k\}, \{\eta^k\},$ and $\{\nu^k\}$. We demonstrated above that $\nu_i^k \to 0$ for $i \in \mathcal I_{+0}(\bar{w})$ as $k \to \infty$, ensuring that $(\nu_i^k)_{i \in \mathcal I_{+0}(\bar{w})}$ is bounded. For the remaining components, suppose, to the contrary, that
	\[
	\hat{M}_k := \max\left\{ \|\mu^k\|, \|\omega^k\|, \|\eta^k\|, \|(\nu_i^k)_{i\in I_0(\bar{z})}\| \right\} \to+\infty,
	\]
	where $I_0(\bar{z}):=\{i\mid \bar{z}_i=0\}=\mathcal I_{0+}(\bar{w})\cup\mathcal I_{00}(\bar{w})$. Passing to a subsequence if necessary, we assume that
	\[
	\frac{(\mu^k,\omega^k,\eta^k,\nu^k)}{\hat{M}_k} \to (\hat{\mu},\hat{\omega},\hat{\eta},\hat{\nu})\neq0.
	\]
	From Claim 1, we have $\hat{\mu}_i=0$ for $i\notin I_G(\bar{x},\bar{y})$, $\hat{\eta}_i=0$ for $i\notin I_g(\bar{x},\bar{y})$, and $\hat{\nu}_i=0$ for $i\in\mathcal I_{+0}(\bar{w})$.
	Moreover, since \(\mu^k\ge0\), we have \(\hat{\mu}_i\ge0\) for all \(i\in I_G(\bar{x},\bar{y})\).
	
	Next, because $\omega^k := -\rho^k(\theta_k^* - y^k) $ and $\|r^k\| / \|\theta_k^* - y^k\| \to 0$ or \(r^k=0\), we have 	
	\[
	\frac{\rho^k\|r^k\|}{\hat{M}_k} = \frac{\|r^k\|}{\|\theta_k^*-y^k\|} \cdot \frac{\|\omega^k\|}{\hat{M}_k} \to0.
	\]
	Dividing \eqref{eq:old-gap-gradient-combined} by $\hat{M}_k$, passing to the limit as $k \to \infty$, and using the continuity of the derivatives along with $\varepsilon_{x,y}^k \to 0$, we obtain
	\begin{equation}\label{limit_xy_stationarity}
		0= \nabla_y\nabla_{x,y}\mathcal{L}(\bar{x},\bar{y},\bar{z})\hat{\omega} + \nabla g(\bar{x},\bar{y})^\top\hat{\eta} + \nabla G(\bar{x},\bar{y})^\top\hat{\mu}.
	\end{equation}
	Applying Taylor's theorem at $y^k$ to the vector mapping $y \mapsto g(x^k,y)$ yields a residual vector $r_{g}^k \in \mathbb{R}^{q}$ satisfying $\|r_{g}^k\| / \|\theta_k^* - y^k\| \to 0$ as $k \to \infty$, where the $i$-th component satisfies
	\[
	g_i(x^k,y^k)-g_i(x^k,\theta_k^*) = -\nabla_y g_i(x^k,y^k)^\top(\theta_k^*-y^k) + r_{g,i}^k, \quad i = 1, \ldots ,q.
	\]
	Multiplying by $\rho^k$ and using $\omega^k=-\rho^k(\theta_k^*-y^k)$ and $\nu_i^k:=\rho^k\bigl(g_i(x^k,y^k)-g_i(x^k,\theta_k^*)\bigr)$, we obtain
	\begin{equation}\label{g}
		\nu_i^k = \nabla_y g_i(x^k,y^k)^\top\omega^k + \rho^kr_{g,i}^k, \quad i = 1, \ldots ,q.
	\end{equation}
	Moreover, since 
	\[
	\frac{|\rho^kr_{g,i}^k|}{\hat{M}_k} = \frac{|r_{g,i}^k|}{\|\theta_k^*-y^k\|} \cdot \frac{\|\omega^k\|}{\hat{M}_k} \to0,
	\]
	dividing \eqref{g} by $\hat{M}_k$ and noting that $\nu_i^k \to 0$ for $i \in \mathcal I_{+0}(\bar{w})$, we obtain
	\begin{equation} \label{limit_Izero}
		0=\nabla_y g_i(\bar{x},\bar{y})^\top\hat{\omega}, \quad i\in\mathcal I_{+0}(\bar{w}), \quad \text{and} \quad
		0=\nabla_y g_i(\bar{x},\bar{y})^\top\hat{\omega}-\hat{\nu}_i, \quad i\in I_0(\bar{z}).
	\end{equation}
	Equations \eqref{limit_xy_stationarity} and \eqref{limit_Izero} exactly match systems \eqref{eq:mpcc_mfcq_xy_component_system} and \eqref{eq:mpcc_mfcq_y_gradient_component_system} in Lemma~\ref{lem:mpcc_mfcq_no_abnormal_multiplier}. Combined with the support relations from Claim~1 and the assumption that MPCC-MFCQ holds at \(\bar{w}\), Lemma~\ref{lem:mpcc_mfcq_no_abnormal_multiplier} excludes any nonzero multiplier tuple satisfying these component relations. This contradicts the normalization \((\hat{\mu},\hat{\omega},\hat{\eta},\hat{\nu})\neq0\). Therefore, $\hat{M}_k$ is bounded, and consequently, the sequences $\{\mu^k\},\ \{\omega^k\},\ \{\eta^k\}, \text{ and } \{\nu^k\}$ are bounded.
	
	We now prove that $\bar{w}$ is C-stationary for \eqref{eq:kkt-reform}. 
	
	Because the sequences $\{\mu^k\}, \{\omega^k\}, \{\eta^k\},$ and $\{\nu^k\}$ are bounded, we may pass to a further subsequence, if necessary, and assume that 
	\[
	(\mu^k,\omega^k,\eta^k,\nu^k) \to (\bar{\mu},\bar{\omega},\bar{\eta},\bar{\nu}).
	\]
	It follows from Claim 1 that $\bar{\mu}_i=0$ for $i\notin I_G(\bar{x},\bar{y})$, $\bar{\eta}_i=0$ for $i\notin I_g(\bar{x},\bar{y})$, and $\bar{\nu}_i=0$ for $i\in\mathcal I_{+0}(\bar{w})$.
	Since $\{\omega^k\}$ is bounded and $\|r^k\| / \|\theta_k^* - y^k\| \to 0$ or \(r^k=0\), we have 	
	\[
	{\rho^k\|r^k\|} = \frac{\|r^k\|}{\|\theta_k^*-y^k\|} \cdot {\|\omega^k\|} \to0.
	\] 
	Similarly, $\rho^kr_{g,i}^k\to0$.
	
	Taking $k \to \infty$ in \eqref{eq:old-gap-gradient-combined} and applying the continuity of the derivatives, we obtain
	\begin{equation}\label{eq:C_limit_stationarity_relation}
		0= \nabla F(\bar{x},\bar{y}) + \nabla_y\nabla_{x,y}\mathcal{L}(\bar{x},\bar{y},\bar{z})\bar{\omega} + \nabla g(\bar{x},\bar{y})^\top\bar{\eta} + \nabla G(\bar{x},\bar{y})^\top\bar{\mu},
	\end{equation}
	where
	\begin{equation}\label{eq:C_limit_support_relations}
		\bar{\mu}_i=0 \quad \text{for } i\notin I_G(\bar{x},\bar{y}), \qquad \bar{\eta}_i=0 \quad \text{for } i\notin I_g(\bar{x},\bar{y}).
	\end{equation}
	Next, taking $k \to \infty$ in \eqref{g} and applying the continuity of the derivatives yields
	\[
	\bar{\nu}_i = \nabla_y g_i(\bar{x},\bar{y})^\top\bar{\omega}, \quad i = 1, \ldots ,q.
	\]
	For \(i\in\mathcal I_{+0}(\bar{w})\), we know $\bar{\nu}_i=0$, thus
	\begin{equation}\label{eq:C_limit_I0plus_orthogonality}
		\nabla_y g_i(\bar{x},\bar{y})^\top\bar{\omega}=0, \quad i\in\mathcal I_{+0}(\bar{w}).
	\end{equation}
	Consider $i\in\mathcal I_{00}(\bar{w})$. If $\bar{\eta}_i>0$, then for all sufficiently large $k$, we have $\eta_i^k = \rho^k(\lambda_{k,i}^*-z_i^k)>0$, which implies $\lambda_{k,i}^* > 0$. The definition of $\lambda_{k,i}^*$ then ensures that $\lambda_{k,i}^*= z_i^k + \gamma_2 g_i(x^k,y^k)$.
	Therefore, from \eqref{eq:approximate_kkt_opened_z}, we have
	\[
	\nu_i^k = \rho^k\bigl(g_i(x^k,y^k)-g_i(x^k,\theta_k^*)\bigr) = \frac{\rho^k}{\gamma_2}(\lambda_{k,i}^*- z_i^k) - \rho^kg_i(x^k,\theta_k^*) \ge \varepsilon_{z,i}^k.
	\]
	Passing to the limit as $k \to \infty$ in the above inequality and noting that $\varepsilon_{z,i}^k \to 0$ yields $\bar{\nu}_i\ge0$, and hence $\bar{\eta}_i\bar{\nu}_i\ge0$. 
	
	If $\bar{\eta}_i<0$, then for all sufficiently large $k$, we have $\eta_i^k = \rho^k(\lambda_{k,i}^*-z_i^k)<0$. Since $\lambda_{k,i}^*\ge0$, this implies $z_i^k>0$. The definition of $\lambda_{k,i}^*$ dictates that $\lambda_{k,i}^* = \operatorname{Proj}_{\mathbb R_+} \left(z_i^k+\gamma_2g_i(x^k,y^k)\right) \ge z_i^k+\gamma_2g_i(x^k,y^k)$. Using the first line of \eqref{eq:approximate_kkt_opened_z}, we obtain
	\[
	\nu_i^k = \rho^k\bigl(g_i(x^k,y^k)-g_i(x^k,\theta_k^*)\bigr) \le \frac{\rho^k}{\gamma_2}(\lambda_{k,i}^*- z_i^k) - \rho^kg_i(x^k,\theta_k^*) = \varepsilon_{z,i}^k.
	\]
	Passing to the limit as $k \to \infty$ in the inequality and noting that $\varepsilon_{z,i}^k \to 0$ yields $\bar{\nu}_i\le0$, and hence $\bar{\eta}_i\bar{\nu}_i\ge0$. 
	
	If $\bar{\eta}_i=0$, then clearly $\bar{\eta}_i\bar{\nu}_i = 0$. Hence, we conclude that
	\begin{equation}\label{eq:C_limit_biactive_sign}
		\bar{\eta}_i\bar{\nu}_i\ge0, \quad \forall i\in\mathcal I_{00}(\bar{w}).
	\end{equation}
	Combining \eqref{eq:C_limit_stationarity_relation}, \eqref{eq:C_limit_support_relations}, \eqref{eq:C_limit_I0plus_orthogonality}, and \eqref{eq:C_limit_biactive_sign}, we conclude that $\bar{w}$ is C-stationary for \eqref{eq:kkt-reform}. 
	This completes the proof.
\end{proof}

\begin{remark}
	Theorem~\ref{thm:akkt_to_stationarity} implies that under the lower-level and upper-level MFCQ assumptions, together with MPCC-MFCQ at the limiting MPCC point, the case \(\rho^k\to+\infty\) yields C-stationarity for the MPCC reformulation \eqref{eq:kkt-reform}. Thus, even though the regularized gap-function formulation avoids writing the full MPCC system explicitly, its approximate KKT limits still recover C-stationarity for the MPCC reformulation of the original bilevel problem. 
\end{remark}

Now, a natural question raised by the preceding analysis is whether, as \(\rho^k\to+\infty\), an accumulation point \(\bar{w}\) necessarily satisfies a stronger stationarity condition, such as M-stationarity, for the MPCC reformulation \eqref{eq:kkt-reform}. The difference between C- and M-stationarity lies only in the sign condition on $\mathcal I_{00}(\bar{w})$. The following example demonstrates that this is not generally the case.

\begin{example}
	Consider the following bilevel optimization problem:
	\[
	\begin{aligned}
		\min_{x\in\mathbb{R}, y\in\mathbb{R}} \  & F(x, y)=x-2y+\frac12x^2+\frac12y^2 \\
		\text{s.t.}\ \quad & y \in \arg\min_{y'} \left\{ f(x, y')=\frac12y'^2-xy' \;\middle|\; g(x, y')=-y' \le 0 \right\}.
	\end{aligned}
	\]
	Set \(\gamma_1=\gamma_2=1\) in the regularized gap-function reformulation \eqref{reformualtion_GP}, and consider the penalty subproblem:
	\[
	\min_{x\in\mathbb{R},y\in\mathbb{R},z\in\mathbb{R}_+} \quad \psi_\rho(x,y,z) := F(x,y)+\rho \mathcal{G}_1(x,y,z).
	\]
	For any \((x,y,z)\), the auxiliary points defined in \eqref{eq:deftheta} are given by $\theta^*(x,y,z)=(x+y+z)/2$, $\lambda^*(x,y,z)=\operatorname{Proj}_{\mathbb R_+}(z+g(x,y))=[z-y]_+$. Consequently,
	\[
	\begin{aligned}
		\mathcal G_1(x,y,z)
		=\frac14(x+y+z)^2-xy-\frac12z^2+\frac12[z-y]_+^2.
	\end{aligned}
	\]
	For every \(\rho>0\), the point
	\[
	w_\rho := (x_\rho,y_\rho,z_\rho) :=
	\left(-\frac{1}{\rho+1},\frac{2}{\rho+1},\frac{1}{\rho+1}\right)
	\]
	is a stationary point of the penalty subproblem. Indeed, since \(z_\rho-y_\rho<0\), we have $\theta^*_\rho = \theta^*(x_\rho,y_\rho,z_\rho)=z_\rho$, $\lambda^*_\rho = \lambda^*(x_\rho,y_\rho,z_\rho)=0$, while the gradient formula \eqref{Ggradient} yields \(\nabla\psi_\rho(w_\rho)=0\). 
	
	We have \(w_\rho\to\bar{w}:=(0,0,0)\) as \(\rho\to+\infty\). The multiplier sequences are given by
	\[
	\eta_\rho := \rho\left(\lambda^*_\rho-z_\rho\right)
	=-\frac{\rho}{\rho+1}\to -1, \qquad 
	\nu_\rho := \rho\left(g(x_\rho,y_\rho)-g(x_\rho,\theta^*_\rho)\right)
	=-\frac{\rho}{\rho+1}\to -1.
	\]
	At the limit point \(\bar{w}\), the MPCC stationarity system from Definition~\ref{def:stationarity} uniquely determines the multipliers \(\bar{\omega}=1\), \(\bar{\eta}=-1\), and \(\bar{\nu}=-1\). Because \(\bar{\eta}\bar{\nu}=1\ge0\), the limit point \(\bar{w}\) satisfies the C-stationarity sign condition. However, the pair \((\bar{\eta},\bar{\nu})=(-1,-1)\) satisfies neither \(\bar{\eta}>0, \bar{\nu}>0\) nor \(\bar{\eta}\bar{\nu}=0\). Consequently, the M-stationarity condition in Definition~\ref{def:stationarity} fails.
\end{example}

Motivated by this limitation, we introduce a two-parameter slack reformulation. In the subsequent section, we establish that the approximate KKT condition associated with this new reformulation implies M-stationarity.

\section{M-Stationarity Enhancement}\label{sec4}

In the preceding section, we established that accumulation points generated by stationary points of the standard penalty approximation \eqref{eq:penalized} for the regularized gap function-based reformulation (GP) in \eqref{reformualtion_GP} are guaranteed only to be C-stationary points of the MPCC reformulation \eqref{eq:kkt-reform}. This limitation primarily arises from insufficiently strict control over the multiplier variable $z$.

To overcome this limitation, we introduce an equivalent reformulation of (GP) in \eqref{reformualtion_GP}. By introducing a slack variable \(s\in\mathbb R^q_+\), we convert the inequality constraint \(g(x,y)\le0\) into the equation $g(x,y)+s=0$. We then explicitly impose the complementarity relation between this slack variable and the multiplier \(z\) by requiring \((z,s)\in\mathcal C\), where \(\mathcal C\) is the complementarity set defined in \eqref{def_C}. This yields the following slack reformulation:
\begin{equation}\label{reformulation_slack}
	\text{(GP)}_{\mathcal C} \quad	\begin{aligned}
		\min_{\substack{(x,y)\in\mathbb R^{n+m}, (z,s)\in\mathbb R^{q+q}}} &F(x,y)\\
		\mathrm{s.t.}\ \quad \quad \quad &(x,y)\in\Omega,\qquad
		\mathcal G_\gamma(x,y,z)\le0,\\
		&g(x,y)+s=0,\qquad
		(z,s)\in\mathcal C.
	\end{aligned}
\end{equation}

For this formulation, we consider the two-parameter penalty subproblem
\begin{equation}\label{eq:penalized-slack}
	\min_{(x,y) \in \Omega, (z,s) \in \mathcal{C} } \quad F(x, y) + \rho_1 \mathcal G_\gamma(x,y,z) + \rho_2 \|g(x,y)+s\|^2,
\end{equation}
where \(\rho_1\) and $\rho_2$ are positive penalty parameters. Observe that in this subproblem, the complementarity condition $(z,s) \in \mathcal{C} $ is retained as a hard constraint rather than being penalized. This exact complementarity condition is used below to recover M-stationarity for the MPCC reformulation \eqref{eq:kkt-reform}.

To analyze the stationarity properties of the accumulation points generated by \eqref{eq:penalized-slack} as the penalty parameters $\rho_1, \rho_2 \to \infty$, we introduce the following approximate KKT condition for $\text{(GP)}_{\mathcal C}$.

\begin{definition}\label{def:two_parameter_approximate_kkt}
	We say that a point $(\bar{x}, \bar{y})$ satisfies the approximate KKT condition for $\text{(GP)}_{\mathcal C}$ in \eqref{reformulation_slack} if there exist sequences $\{(x^k, y^k, z^k,s^k)\} \subset \mathbb{R}^{n+m+q+q}$, $\{ \varepsilon^k \} \subset \mathbb{R}^{n+m+q+q}$, and $\{(\rho_1^k, \rho_2^k)\} \subset \mathbb{R}_{++}^2$ such that
	\begin{equation}\label{approximate_kkt-s}
		\begin{aligned}
			&\lim_{k\to\infty} (x^k, y^k) = (\bar{x}, \bar{y}), \quad \lim_{k\to\infty} \varepsilon^k = 0, \quad (x^k,y^k) \in \Omega, \quad (z^k, s^k) \in \mathcal{C}, \\
			&\lim_{k\to\infty} \mathcal{G}_\gamma(x^k,y^k,z^k) = 0, \quad \lim_{k\to\infty} \bigl( g(x^k, y^k) + s^k \bigr) = 0, \\
			&\varepsilon^k \in \nabla F(x^k,y^k) + \rho_1^k\nabla \mathcal{G}_\gamma(x^k,y^k,z^k) + \rho_2^k\nabla \left( \|g(x^k, y^k) + s^k\|^2\right) + \mathcal{N}_{\Omega }(x^k,y^k) \times \mathcal N_{\mathcal C}(z^k,s^k),
		\end{aligned}
	\end{equation}
	where, for notational simplicity, the gradients \(\nabla F(x^k,y^k)\), \(\nabla\mathcal G_\gamma(x^k,y^k,z^k)\), and $\nabla \left( \|g(x^k, y^k) + s^k\|^2\right)$ are embedded into the \((x,y,z,s)\)-space by assigning zero entries to the components of absent variables.
\end{definition}

We use the shorthand \(u = (x,y,z,s)\). As in the preceding section, before presenting the approximate KKT condition, we recall the index sets used below. For \(z,s\in\mathbb R_+^q\), define
\[
I_+(z):=\{i\mid z_i>0\},\quad
I_0(z):=\{i\mid z_i=0\},\quad
I_+(s):=\{i\mid s_i>0\},\quad
I_0(s):=\{i\mid s_i=0\}.
\]
For a limit point \(\bar u=(\bar x,\bar y,\bar z,\bar s)\), we use the same notation as in Section \ref{sec3}: \(I_G(\bar x,\bar y)\), \(I_g(\bar x,\bar y)\), and \(I_0(\bar z)\) denote the active index sets of the constraints \(G(x,y)\le0\), \(g(x,y)\le0\), and \(z\ge0\), respectively. If \(\bar u\) is feasible for the slack reformulation, then \(g(\bar x,\bar y)+\bar s=0\). Hence $I_g(\bar x,\bar y)=I_0(\bar s)$. Moreover, for \(\bar w=(\bar x,\bar y,\bar z)\), the MPCC complementarity index sets satisfy
\[
I_g(\bar x,\bar y)=\mathcal I_{+0}(\bar w)\cup\mathcal I_{00}(\bar w),
\quad
I_0(\bar z)=\mathcal I_{0+}(\bar w)\cup\mathcal I_{00}(\bar w).
\]

As in Lemma~\ref{lem:approximate_kkt_normal_cone_free}, the next lemma relates the approximate KKT condition for \(\mathrm{(GP)}_{\mathcal C}\) in \eqref{reformulation_slack} to the stationarity conditions of the MPCC reformulation \eqref{eq:kkt-reform}.

\begin{lemma}\label{lem:approximate_kkt_slack_normal_cone_free}
	Suppose that $(\bar{x}, \bar{y})$ satisfies the approximate KKT condition for $\mathrm{(GP)}_{\mathcal C}$. Let\linebreak[4] $\{(x^k, y^k, z^k, s^k)\} \subset \mathbb{R}^{n+m+q+q}$, $\{ \varepsilon^k \} \subset \mathbb{R}^{n+m+q+q}$, and $\{(\rho_1^k, \rho_2^k)\} \subset \mathbb{R}_{++}^2$ satisfy \eqref{approximate_kkt-s}. If the MFCQ holds at $\bar{y}$ for the lower-level constraint system $\{ y \mid g(\bar{x}, y) \le 0 \}$, then $\{z^k\}$ and $\{s^k\}$ are bounded.
	
	Furthermore, if $(\bar{z}, \bar{s}) \in \mathbb{R}_+^q \times \mathbb{R}_+^q$ is any accumulation point of the sequence $\{(z^k, s^k)\}$, then $\bar{u} := (\bar{x}, \bar{y}, \bar{z}, \bar{s})$ is a feasible point for $\mathrm{(GP)}_{\mathcal C}$ in \eqref{reformulation_slack}, and $\bar{w} := (\bar{x}, \bar{y}, \bar{z})$ is a feasible point for the MPCC reformulation in \eqref{eq:kkt-reform}.
	
	Assume, in addition, that the MFCQ holds at $(\bar{x}, \bar{y})$ for the upper-level inequality system $\{ (x,y) \mid G(x,y) \le 0 \}$. By partitioning the error sequence $\varepsilon^k = (\varepsilon_{x,y}^k, \varepsilon_{z}^k, \varepsilon_s^k)$ with $\varepsilon_{x,y}^k \in \mathbb{R}^{n+m}$ and $\varepsilon_z^k, \varepsilon_s^k \in \mathbb{R}^q$, there exists a multiplier sequence $\{\mu^k\} \subset \mathbb{R}_+^p$ satisfying $\mu_i^k G_i(x^k, y^k) = 0$ for $i = 1, \ldots, p$, such that
	\begin{equation}\label{eq:slack-gap-gradient-combined}
		\varepsilon_{x,y}^k + \rho_1^k r^k = \nabla F(x^k, y^k) + \nabla G(x^k, y^k)^\top \mu^k + \nabla_y\nabla_{x,y}\mathcal{L}(x^k, y^k, z^k)\omega^k + \nabla g(x^k, y^k)^\top \eta^k,
	\end{equation}
	where 
	\[
	\omega^k := -\rho_1^k(\theta_k^* - y^k), \quad
	\eta^k := \rho_1^k(\lambda_k^* - z^k) + \zeta^k, \quad
	\zeta^k := 2\rho_2^k\bigl(g(x^k, y^k) + s^k\bigr),
	\]
	with $\theta_k^* := \theta^*(x^k, y^k, z^k)$ and $\lambda_k^* := \lambda^*(x^k, y^k, z^k)$. Here, $r^k \in \mathbb{R}^{n+m}$ is a residual vector satisfying \(r^k=0\) whenever \(\theta_k^*=y^k\), and \(\|r^k\|/|\theta_k^*-y^k\|\to0\) otherwise, as $k \to \infty$. Moreover, the following componentwise conditions hold:
	\begin{equation}\label{eq:approximate_kkt_slack_zs}
		\begin{aligned}
			&\text{for } i\in I_+(z^k):\
			g_i(x^k,\theta_k^*) =
			\dfrac{\lambda_{k,i}^* - z_i^k}{\gamma_2}
			-
			\dfrac{\varepsilon_{z,i}^k}{\rho_1^k},
			\qquad
			\text{for } i\in I_+(s^k):\
			\zeta_i^k = \varepsilon_{s,i}^k,\\[0.6em]
			&\text{for } i\in I_0(z^k)\cap I_0(s^k):\ \text{one of }
			\left\{
			\begin{aligned}
				&g_i(x^k,\theta_k^*) \le
				\dfrac{\lambda_{k,i}^* - z_i^k}{\gamma_2}
				-
				\dfrac{\varepsilon_{z,i}^k}{\rho_1^k},
				\quad
				\zeta_i^k \ge \varepsilon_{s,i}^k,
				\\[0.2em]
				&g_i(x^k,\theta_k^*) =
				\dfrac{\lambda_{k,i}^* - z_i^k}{\gamma_2}
				-
				\dfrac{\varepsilon_{z,i}^k}{\rho_1^k},
				\\[0.2em]
				&\zeta_i^k = \varepsilon_{s,i}^k
			\end{aligned}
			\right.
			\text{ holds,}
		\end{aligned}
	\end{equation}
	where $I_+(z^k) := \{i \mid z_i^k > 0\}$, $I_0(z^k) := \{i \mid z_i^k = 0\}$, $I_+(s^k) := \{i \mid s_i^k > 0\}$, and $I_0(s^k) := \{i \mid s_i^k = 0\}$.
\end{lemma}

\begin{proof}
	The boundedness of $\{z^k\}$ follows from an argument analogous to that used in the proof of Lemma~\ref{lem:approximate_kkt_normal_cone_free}. Indeed, the present approximate KKT condition still implies that $(x^k, y^k) \to (\bar{x}, \bar{y})$, $z^k \ge 0$, and $\mathcal{G}_\gamma(x^k, y^k, z^k) \to 0$. 
	Because $\lim_{k\to\infty} \mathcal{G}_\gamma(x^k, y^k, z^k) = 0$, Lemma~\ref{lem:vanishing_gap_residuals} ensures that $\theta_k^* - y^k \to 0$ and $\lambda_k^* - z^k \to 0$, where we recall that $
	\theta_k^* := \theta^*(x^k, y^k, z^k)$, $ 
	\lambda_k^* := \lambda^*(x^k, y^k, z^k) = \operatorname{Proj}_{\mathbb{R}_+^q}\left(z^k + \gamma_2 g(x^k, y^k)\right)$.
	If $\{z^k\}$ were unbounded, normalizing by setting $\hat{z}^k = z^k / \|z^k\|$ and passing to a necessary subsequence would yield a nonzero limit vector $\hat{z} \ge 0$, supported on $I_g(\bar{x}, \bar{y})$, such that
	\[
	\sum_{i \in I_g(\bar{x}, \bar{y})} \hat{z}_i \nabla_y g_i(\bar{x}, \bar{y}) = 0.
	\]
	This contradicts the lower-level MFCQ for the system $g(\bar{x}, \cdot) \le 0$ at $\bar{y}$. Thus, $\{z^k\}$ is bounded.
	
	The first difference from Lemma~\ref{lem:approximate_kkt_normal_cone_free} is the presence of the slack variable $s^k$. Since $g(x^k, y^k) + s^k \to 0$ and $(x^k, y^k) \to (\bar{x}, \bar{y})$, the continuity of $g$ yields $s^k \to -g(\bar{x}, \bar{y})$, implying that $\{s^k\}$ is also bounded. Let $(\bar{z}, \bar{s})$ be an accumulation point of $\{(z^k, s^k)\}$. Passing to a subsequence if necessary, we may assume that $(z^k, s^k) \to (\bar{z}, \bar{s})$. The closedness of $\Omega$ and $\mathcal{C}$, combined with $g(x^k, y^k) + s^k \to 0$, establishes that
	\[
	(\bar{x}, \bar{y}) \in \Omega, \qquad
	(\bar{z}, \bar{s}) \in \mathcal{C}, \qquad
	g(\bar{x}, \bar{y}) + \bar{s} = 0,
	\]
	which implies that $\bar{z}_i g_i(\bar{x}, \bar{y}) = 0$ for $i = 1, \ldots, q$. Moreover, the continuity of $\mathcal{G}_\gamma$ established in Lemma~\ref{smooth_G} ensures that $\mathcal{G}_\gamma(\bar{x}, \bar{y}, \bar{z}) = 0$. Thus, $\bar{u} = (\bar{x}, \bar{y}, \bar{z}, \bar{s})$ is feasible for $\mathrm{(GP)}_{\mathcal C}$ in \eqref{reformulation_slack}.  This immediately dictates that $\bar{y} \in S(\bar{x})$ and $\bar{z} \in \mathcal{M}(\bar{x}, \bar{y})$. Consequently, 
	\[
	\nabla_y \mathcal{L}(\bar{x}, \bar{y}, \bar{z}) = 0, \qquad
	g(\bar{x}, \bar{y}) \le 0, \qquad
	\bar{z} \ge 0, \qquad
	\bar{z}^\top g(\bar{x}, \bar{y}) = 0.
	\]
	Hence, $\bar{w} = (\bar{x}, \bar{y}, \bar{z})$ is a feasible point for the MPCC reformulation \eqref{eq:kkt-reform}.
	
	It remains to analyze the multiplier representation. Define $\zeta^k := 2\rho_2^k\bigl(g(x^k, y^k) + s^k\bigr)$. The $(x, y)$-component of \eqref{approximate_kkt-s} contains, alongside the gap-function term, the equality-penalty contribution $\nabla g(x^k, y^k)^\top \zeta^k$. Applying the normal-cone representation of $\mathcal{N}_\Omega(x^k, y^k)$ under the upper-level MFCQ, we deduce the existence of a multiplier $\mu^k \in \mathbb{R}_+^p$ satisfying $\mu_i^k G_i(x^k, y^k) = 0$ such that
	\[
	\varepsilon_{x,y}^k
	=
	\nabla F(x^k, y^k)
	+
	\nabla G(x^k, y^k)^\top \mu^k
	+
	\rho_1^k \nabla_{x,y} \mathcal{G}_\gamma(x^k, y^k, z^k)
	+
	\nabla g(x^k, y^k)^\top \zeta^k.
	\]
	We evaluate the gradient formula for $\mathcal{G}_\gamma$ from \eqref{Ggradient} at $(x^k, y^k, z^k)$ and apply Taylor's theorem at $y^k$ to the mapping $y \mapsto \nabla_{x,y} \mathcal{L}(x^k, y, z^k)$. Proceeding analogously to the derivation in Lemma~\ref{lem:approximate_kkt_normal_cone_free}, but substituting $\rho^k$ with $\rho_1^k$, we obtain \eqref{eq:slack-gap-gradient-combined}. Here, $\omega^k := -\rho_1^k(\theta_k^* - y^k)$ and $\eta^k := \rho_1^k(\lambda_k^* - z^k) + \zeta^k$, while $r^k$ represents the Taylor remainder satisfying \(r^k=0\) whenever \(\theta_k^*=y^k\), and \(\|r^k\|/|\theta_k^*-y^k\|\to0\) otherwise, as $k \to \infty$.
	
	Another difference is that the normal cone operating on the $(z, s)$-variables is the limiting normal cone $\mathcal{N}_{\mathcal{C}}$, rather than $\mathcal{N}_{\mathbb{R}_+^q}$. From the gradient formula
	\[
	\nabla_z \mathcal{G}_\gamma(x^k, y^k, z^k)
	=
	\frac{\lambda_k^* - z^k}{\gamma_2}
	-
	g(x^k, \theta_k^*),
	\]
	the $(z, s)$-component of \eqref{approximate_kkt-s} yields
	\[
	\left(
	\varepsilon_z^k - \rho_1^k \left( \frac{\lambda_k^* - z^k}{\gamma_2} - g(x^k, \theta_k^*) \right),
	\varepsilon_s^k - \zeta^k
	\right)
	\in
	\mathcal{N}_{\mathcal{C}}(z^k, s^k).
	\]
	For each scalar complementarity set $\mathcal{C}_i = \{(z_i, s_i) \in \mathbb{R}_+^2 \mid s_i z_i = 0\}$, by Lemma \ref{lem:normal_cone_complementarity}, the limiting normal cone is explicitly given by
	\[
	\mathcal{N}_{\mathcal{C}_i}(z_i, s_i) =
	\begin{cases}
		\{0\} \times \mathbb{R}, & z_i > 0, \ s_i = 0,\\
		\mathbb{R} \times \{0\}, & z_i = 0, \ s_i > 0,\\
		(\mathbb{R}_- \times \mathbb{R}_-) \cup (\{0\} \times \mathbb{R}) \cup (\mathbb{R} \times \{0\}), & z_i = s_i = 0.
	\end{cases}
	\]
	Applying this componentwise formula to the preceding normal-cone inclusion gives \eqref{eq:approximate_kkt_slack_zs}. Indeed, if $z_i^k > 0$, then $s_i^k = 0$ and the $z_i$-component of the normal cone is zero. Hence the first component of the above inclusion vanishes, which gives the condition displayed for $i \in I_+(z^k)$ in the first line of \eqref{eq:approximate_kkt_slack_zs}. Similarly, if $s_i^k > 0$, then $z_i^k = 0$ and the $s_i$-component of the normal cone is zero. The second component of the inclusion therefore yields $\zeta_i^k = \varepsilon_{s,i}^k$, which is the condition for $i \in I_+(s^k)$ in the first line of \eqref{eq:approximate_kkt_slack_zs}. Finally, if $z_i^k = s_i^k = 0$, the three branches of $\mathcal{N}_{\mathcal{C}_i}(0, 0)$ give the second line of \eqref{eq:approximate_kkt_slack_zs}. This completes the proof.
\end{proof}

The preceding lemma plays the same role for \((\mathrm{GP})_{\mathcal C}\) as Lemma~\ref{lem:approximate_kkt_normal_cone_free} does for the single-parameter formulation, it rewrites the approximate KKT condition for \((\mathrm{GP})_{\mathcal C}\) as a multiplier system for the MPCC reformulation \eqref{eq:kkt-reform}. The slack variable (s) and the penalty parameter \(\rho_2\) provide additional control over the feasibility residual \(g(x,y)+s\), leading to a stronger limiting stationarity result.

\begin{theorem}\label{thm:two_parameter_M_stationarity}
	Suppose that $(\bar{x},\bar{y})$ satisfies the approximate KKT condition for $\mathrm{(GP)}_{\mathcal C}$. Let $\{(x^k,y^k,z^k,s^k)\}\subset\mathbb R^{n+m+q+q}$ and $\{(\rho_1^k,\rho_2^k)\}\subset\mathbb R_{++}^2$ satisfy \eqref{approximate_kkt-s}. Assume that the MFCQ holds at $\bar y$ for the lower-level constraint system $\{y\mid g(\bar x,y)\le0\}$, and that the MFCQ holds at $(\bar x,\bar y)$ for the upper-level inequality system $\{(x,y)\mid G(x,y)\le0\}$.  Let $\bar z\in\mathbb R_+^q$ be any accumulation point of the sequence $\{z^k\}$. Then the following stationarity conditions hold for $\bar w:=(\bar x,\bar y,\bar z)$:
	\begin{enumerate}
		\item If both $\{\rho_1^k\}$ and $\{\rho_2^k\}$ are bounded, then $\bar w$ is S-stationary for \eqref{eq:kkt-reform}.
		\item If $\rho_1^k\to+\infty$, $\rho_2^k\to+\infty$, and $\rho_2^k/\rho_1^k\to+\infty$, and we further assume that the MPCC-MFCQ holds at $\bar w$ for the MPCC reformulation \eqref{eq:kkt-reform}, then $\bar w$ is M-stationary for \eqref{eq:kkt-reform}.
	\end{enumerate}
\end{theorem}

\begin{proof}
	Let $\{(x^k, y^k, z^k, s^k)\} \subset \mathbb{R}^{n+m+q+q}$, $\{ \varepsilon^k \} \subset \mathbb{R}^{n+m+q+q}$, and $\{(\rho_1^k,\rho_2^k)\} \subset \mathbb{R}_{++}^2$ satisfy \eqref{approximate_kkt-s}.
	By Lemma~\ref{lem:approximate_kkt_slack_normal_cone_free}, the sequences $\{z^k\}$ and $\{s^k\}$ are bounded, and conditions \eqref{eq:slack-gap-gradient-combined} and \eqref{eq:approximate_kkt_slack_zs} hold with $\{\mu^k\}$ satisfying $\mu_i^k G_i(x^k,y^k)=0$ for $i=1,\ldots,p$.
	Let $(\bar{z},\bar{s}) \in \mathbb{R}_+^q\times\mathbb{R}_+^q$ be any accumulation point of the sequence $\{(z^k,s^k)\}$. It follows that $g(\bar{x},\bar{y})+\bar{s}=0$, $(\bar{z},\bar{s})\in\mathcal C$, and $\bar{z}_ig_i(\bar{x},\bar{y})=0$ for $i=1,\ldots,q$. Consequently, $\bar{u}:=(\bar{x},\bar{y},\bar{z},\bar{s})$ is a feasible point for $\mathrm{(GP)}_{\mathcal C}$ in \eqref{reformulation_slack}, and $\bar{w}:=(\bar{x},\bar{y},\bar{z})$ is a feasible point for the MPCC reformulation in \eqref{eq:kkt-reform}.
	In the remainder of the proof, by passing to a subsequence if necessary, we assume without loss of generality that $(x^k, y^k, z^k, s^k) \to \bar{u}:=(\bar{x},\bar{y},\bar{z},\bar{s})$.
	
	Since $\lim_{k\to\infty}\mathcal{G}_\gamma(x^k,y^k,z^k) = 0$, Lemma~\ref{lem:vanishing_gap_residuals} ensures that $\theta_k^* - y^k \to 0$ and $\lambda_k^* - z^k \to 0$, where
	$ \theta_k^*:=\theta^*(x^k,y^k,z^k)$, and $ 
	\lambda_k^*:=\lambda^*(x^k,y^k,z^k) = \operatorname{Proj}_{\mathbb R_+^q}\left(z^k+\gamma_2g(x^k,y^k)\right)$. 
	We also recall the notation $\omega^k:=-\rho_1^k(\theta_k^*-y^k)$, $\zeta^k:=2\rho_2^k\bigl(g(x^k,y^k)+s^k\bigr)$,  $\eta^k:=\rho_1^k(\lambda_k^*-z^k)+\zeta^k$, and  $\nu^k:=\rho_1^k(g(x^k,y^k)-g(x^k,\theta_k^*))$.
	
	\textbf{Case 1: Both $\{\rho_1^k\}$ and $\{\rho_2^k\}$ are bounded.}
	
	The argument proceeds similarly to Case 1 in the proof of Theorem~\ref{thm:akkt_to_stationarity}, with $\rho^k$ replaced by $\rho_1^k$ and \eqref{eq:old-gap-gradient-combined} replaced by \eqref{eq:slack-gap-gradient-combined}. The only additional term is the slack-penalty multiplier $\zeta^k$. Since $g(x^k,y^k)+s^k\to0$ and $\{\rho_2^k\}$ is bounded, it follows that $\zeta^k\to0$. Thus, $\omega^k\to0$, $\eta^k\to0$, and $\rho_1^kr^k\to0$.
	Repeating the MFCQ contradiction argument utilized in Case 1 of Theorem~\ref{thm:akkt_to_stationarity} establishes the boundedness of $\{\mu^k\}$. Passing to the limit in \eqref{eq:slack-gap-gradient-combined} yields
	$$ 0=\nabla F(\bar{x},\bar{y})+\nabla G(\bar{x},\bar{y})^\top\bar\mu,\qquad
	\bar\mu\ge0,\qquad
	\bar\mu^\top G(\bar{x},\bar{y})=0. $$
	Therefore, \eqref{eq:mpcc-stationarity} holds with $\bar\omega:=0$, $\bar\eta:=0$, and $\bar\nu:=0$. In particular, since $\bar\eta_i=0\ge0$ and $\bar\nu_i=0\ge0$ for all $i\in\mathcal I_{00}(\bar w)$, $\bar w$ is S-stationary for \eqref{eq:kkt-reform}.
	
	\textbf{Case 2: $\rho_1^k\to+\infty$, $\rho_2^k\to+\infty$, and $\rho_2^k/\rho_1^k\to+\infty$.}
	
	We first establish the following result regarding the multiplier sequences. 
	
	\textit{Claim 1: For all sufficiently large $k$, $\mu_i^k=0$ for $i\notin I_G(\bar x,\bar y)$ and $z_i^k=0$ for $i\in\mathcal I_{0+}(\bar w)$. Moreover, $\eta_i^k\to0$ for $i\in\mathcal I_{0+}(\bar w)$ and $\nu_i^k:=\rho_1^k\bigl(g_i(x^k,y^k)-g_i(x^k,\theta_k^*)\bigr)\to0$ for $i\in\mathcal I_{+0}(\bar w)$ as $k\to\infty$.}
	
	If $i\notin I_G(\bar x,\bar y)$, then $G_i(\bar x,\bar y)<0$. By the continuity of $G_i$ and the convergence $(x^k,y^k)\to(\bar x,\bar y)$, the complementarity relation $\mu_i^kG_i(x^k,y^k)=0$ ensures that $\mu_i^k=0$ for all sufficiently large $k$.
	
	Let $i\in\mathcal I_{0+}(\bar w)$. Then $g_i(\bar x,\bar y)<0$, $\bar z_i=0$, and $\bar s_i=-g_i(\bar x,\bar y)>0$. Consequently, $s_i^k>0$ for all sufficiently large $k$. Since $(z^k,s^k)\in\mathcal C$, this eventually necessitates $z_i^k=0$. The $s_i$-component equality for $i \in I_+(s^k)$ in \eqref{eq:approximate_kkt_slack_zs} dictates that $\zeta_i^k=\varepsilon_{s,i}^k\to0$. Furthermore, because $g_i(x^k,y^k)<0$ for all sufficiently large $k$, we obtain $\lambda_{k,i}^*= [\gamma_2g_i(x^k,y^k)]_+=0$, which in turn implies $\eta_i^k=\rho_1^k(\lambda_{k,i}^*-z_i^k)+\zeta_i^k=\zeta_i^k\to0$.
	
	Let $i\in\mathcal I_{+0}(\bar w)$. Then $\bar z_i>0$ and $g_i(\bar x,\bar y)=0$. Hence, for all sufficiently large $k$, we observe $z_i^k>0$ and $z_i^k+\gamma_2g_i(x^k,y^k)>0$, which leads to $\lambda_{k,i}^*=z_i^k+\gamma_2g_i(x^k,y^k)$.
	It follows that $g_i(x^k,y^k)=(\lambda_{k,i}^*-z_i^k)/\gamma_2$. Concurrently, since $z_i^k>0$, the $z_i$-component equality for $i\in I_+(z^k)$ in \eqref{eq:approximate_kkt_slack_zs} gives
	$ g_i(x^k,\theta_k^*)=
	(\lambda_{k,i}^*-z_i^k)/\gamma_2
	-{\varepsilon_{z,i}^k}/{\rho_1^k} $.
	Combining these two identities proves Claim~1, since
	\begin{equation}\label{eq:nuik}
		\nu_i^k
		=\rho_1^k\bigl(g_i(x^k,y^k)-g_i(x^k,\theta_k^*)\bigr)
		=\varepsilon_{z,i}^k\to0.
	\end{equation}
	
	\textit{Claim 2: The multiplier sequences $\{\mu^k\}$, $\{\omega^k\}$, $\{\eta^k\}$, and $\{\nu^k\}$ are bounded.}
	
	We now verify the boundedness of $\{\mu^k\}, \{\omega^k\}, \{\eta^k\}$, and $\{\nu^k\}$. Claim 1 established that $\nu_i^k\to0$ for $i\in\mathcal I_{+0}(\bar w)$, ensuring these components remain bounded. For the remaining components, assume for the sake of contradiction that
	$$ \hat M_k:=\max\{\|\mu^k\|,\|\omega^k\|,\|\eta^k\|,\|(\nu_i^k)_{i\in I_0(\bar z)}\|\}\to+\infty, $$
	where $I_0(\bar z):=\{i\mid \bar z_i=0\}=\mathcal I_{0+}(\bar w)\cup\mathcal I_{00}(\bar w)$. Extracting a subsequence if necessary, we may assume that
	$$ \frac{(\mu^k,\omega^k,\eta^k,\nu^k)}{\hat M_k}\to(\hat\mu,\hat\omega,\hat\eta,\hat\nu)\neq0. $$
	By Claim 1, $\hat\mu_i=0$ for $i\notin I_G(\bar x,\bar y)$, and $\hat\eta_i=0$ for $i\notin I_g(\bar x,\bar y)$ (since $i\notin I_g(\bar x,\bar y)$ is equivalent to $i\in\mathcal I_{0+}(\bar w)$), while $\hat\nu_i=0$ for $i\in\mathcal I_{+0}(\bar w)$. Additionally, $\hat\mu_i\ge0$ for $i\in I_G(\bar x,\bar y)$.
	
	The remainder of the argument mirrors Claim 2 in Theorem~\ref{thm:akkt_to_stationarity}, substituting $\rho_1^k$ for $\rho^k$ and \eqref{eq:slack-gap-gradient-combined} for \eqref{eq:old-gap-gradient-combined}. Specifically, given $\omega^k=-\rho_1^k(\theta_k^*-y^k)$ and \(r^k=0\) whenever \(\theta_k^*=y^k\), and \(\|r^k\|/|\theta_k^*-y^k\|\to0\) otherwise, we deduce $\rho_1^k\|r^k\|/\hat M_k\to0$. Dividing \eqref{eq:slack-gap-gradient-combined} by $\hat M_k$ recovers the identical limiting $(x,y)$-stationarity relation in \eqref{limit_xy_stationarity}. Similarly, Taylor expansion yields
	$$ \nu_i^k=\nabla_y g_i(x^k,y^k)^\top\omega^k+\rho_1^kr_{g,i}^k,\qquad
	\frac{\rho_1^k r_{g,i}^k}{\hat M_k}\to0, $$
	which recovers the identical limiting $g$-gradient relation as \eqref{limit_Izero}. Applying the assumption that MPCC-MFCQ holds at \(\bar{w}\) and  Lemma~\ref{lem:mpcc_mfcq_no_abnormal_multiplier} then contradicts the normalization condition\linebreak[4] $(\hat\mu,\hat\omega,\hat\eta,\hat\nu)\neq0$. Therefore, $\hat M_k$ is bounded, and consequently, the sequences $\{\mu^k\}$, $\{\omega^k\}$, $\{\eta^k\}$, and $\{\nu^k\}$ are bounded. This proves Claim 2.
	
	We now prove that $\bar w$ is M-stationary for \eqref{eq:kkt-reform}.
	
	Because the sequences $\{\mu^k\}, \{\omega^k\}, \{\eta^k\},$ and $\{\nu^k\}$ are bounded, we may extract a further subsequence, if necessary, to assume that 
	$$ (\mu^k,\omega^k,\eta^k,\nu^k)\to(\bar\mu,\bar\omega,\bar\eta,\bar\nu). $$
	Applying the same limit transition as in Theorem~\ref{thm:akkt_to_stationarity} to \eqref{eq:slack-gap-gradient-combined} and the Taylor expansion above yields the MPCC stationarity system \eqref{eq:mpcc-stationarity}; in particular,
	$$ \bar\eta_i=0\quad (i\in\mathcal I_{0+}(\bar w)),\qquad
	\bar\nu_i=0\quad (i\in\mathcal I_{+0}(\bar w)). $$
	It remains only to verify the stronger M-stationarity sign condition on the biactive set $\mathcal I_{00}(\bar w)$.
	
	Fix $i\in\mathcal I_{00}(\bar w)$. Since $(z_i^k,s_i^k)\in\{(a,b)\in\mathbb R_+^2\mid ab=0\}$, exactly one of the following three alternatives must hold for each $k$:
	$$ z_i^k>0,\ s_i^k=0;\qquad
	z_i^k=0,\ s_i^k>0;\qquad
	z_i^k=s_i^k=0. $$
	With only finitely many alternatives, at least one must hold along an infinite subsequence. Passing to this subsequence and relabeling it if necessary, we may assume a single alternative holds for all $k$. Since $(\eta_i^k,\nu_i^k)\to(\bar\eta_i,\bar\nu_i)$, this refinement does not alter the limiting multipliers.
	
	\textit{First alternative: $z_i^k>0$ and $s_i^k=0$ for all $k$.}
	We will show that $\bar\nu_i=0$. If $\lambda_{k,i}^*=z_i^k+\gamma_2g_i(x^k,y^k)$ holds along an infinite subsequence, we may pass to this subsequence (relabeling if necessary) and assume this relation holds for all $k$. Utilizing the same logic as in the derivation of \eqref{eq:nuik}, the $z_i$-component equality for $i\in I_+(z^k)$ in \eqref{eq:approximate_kkt_slack_zs} gives $\nu_i^k = \rho_1^k\bigl(g_i(x^k,y^k)-g_i(x^k,\theta_k^*)\bigr) = \varepsilon_{z,i}^k \to 0$. 
	
	Alternatively, if $\lambda_{k,i}^*=0$ along an infinite subsequence, passing to this subsequence and relabeling it allows us to assume $\lambda_{k,i}^*=0$ and $z_i^k+\gamma_2g_i(x^k,y^k)\le0$ for all $k$. In this scenario,
	$$ \eta_i^k=-\rho_1^kz_i^k+2\rho_2^kg_i(x^k,y^k). $$
	The boundedness of $\{\eta_i^k\}$, coupled with $-\rho_1^kz_i^k < 0$ and $g_i(x^k,y^k)\le-z_i^k/\gamma_2<0$, implies that $\rho_2^k|g_i(x^k,y^k)|$ is bounded. Since $\rho_1^k/\rho_2^k\to0$, it subsequently follows that $\rho_1^k|g_i(x^k,y^k)|\to0$. Consequently, $0 < \rho_1^kz_i^k \le - \gamma_2\rho_1^kg_i(x^k,y^k) \to 0$. Because the $z_i$-component equality for $i\in I_+(z^k)$ in \eqref{eq:approximate_kkt_slack_zs} gives
	$$ g_i(x^k,\theta_k^*)= -\frac{z_i^k}{\gamma_2}-\frac{\varepsilon_{z,i}^k}{\rho_1^k}, $$
	we obtain
	$$ \nu_i^k=\rho_1^k\bigl(g_i(x^k,y^k)-g_i(x^k,\theta_k^*)\bigr) = \rho_1^k\left(g_i(x^k,y^k)+\frac{z_i^k}{\gamma_2}\right)+\varepsilon_{z,i}^k\to0. $$
	Thus, $\bar\nu_i=0$ under the first alternative.
	
	\textit{Second alternative: $z_i^k=0$ and $s_i^k>0$ for all $k$.}
	We will show that $\bar\eta_i=0$. The $s_i$-component equality for $i \in I_+(s^k)$ in \eqref{eq:approximate_kkt_slack_zs} provides $\zeta_i^k=\varepsilon_{s,i}^k\to0$. Furthermore,
	$$ \eta_i^k=\rho_1^k\lambda_{k,i}^*+\zeta_i^k
	=\gamma_2\rho_1^k[g_i(x^k,y^k)]_++\zeta_i^k. $$
	If $g_i(x^k,y^k)<0$, then $\lambda_{k,i}^*=0$ and $\eta_i^k=\zeta_i^k\to0$. If $g_i(x^k,y^k)\ge0$, then $\zeta_i^k=2\rho_2^k(g_i(x^k,y^k)+s_i^k)\ge2\rho_2^kg_i(x^k,y^k)$, and hence, leveraging $\rho_1^k/\rho_2^k\to0$, we have
	$$ 0\le \rho_1^k[g_i(x^k,y^k)]_+
	\le \frac{\rho_1^k}{2\rho_2^k}\zeta_i^k\to0. $$
	Thus $\eta_i^k\to0$, meaning $\bar\eta_i=0$ under the second alternative.
	
	\textit{Third alternative: $z_i^k=s_i^k=0$ for all $k$.}
	In this scenario, $\lambda_{k,i}^*=\gamma_2[g_i(x^k,y^k)]_+$ and $\zeta_i^k=2\rho_2^kg_i(x^k,y^k)$. The limiting-normal inclusion in \eqref{eq:approximate_kkt_slack_zs} separates into three branches. By passing, if necessary, to a further infinite subsequence and relabeling it, we may assume a single branch holds for all $k$.
	
	On the regular-normal branch,
	$$ g_i(x^k,\theta_k^*)\le [g_i(x^k,y^k)]_+-\frac{\varepsilon_{z,i}^k}{\rho_1^k},
	\qquad
	\zeta_i^k\ge\varepsilon_{s,i}^k. $$
	The second inequality ensures $\bar\eta_i\ge0$, because $\eta_i^k=\gamma_2\rho_1^k[g_i(x^k,y^k)]_++\zeta_i^k\ge\varepsilon_{s,i}^k \to 0$, while the first yields
	$$ \nu_i^k
	=\rho_1^k\bigl(g_i(x^k,y^k)-g_i(x^k,\theta_k^*)\bigr)
	\ge \rho_1^k\bigl(g_i(x^k,y^k)-[g_i(x^k,y^k)]_+\bigr)+\varepsilon_{z,i}^k. $$
	If $g_i(x^k,y^k)\ge0$, the right-hand side simplifies to $\varepsilon_{z,i}^k$, yielding $\nu_i^k \ge \varepsilon_{z,i}^k \to 0$. If $g_i(x^k,y^k)<0$, the boundedness of $\eta_i^k=\zeta_i^k=2\rho_2^kg_i(x^k,y^k)$ combined with $\rho_1^k/\rho_2^k\to0$ implies $\rho_1^kg_i(x^k,y^k)\to0$, thereby giving $\nu_i^k \ge \rho_1^kg_i(x^k,y^k) + \varepsilon_{z,i}^k \to 0$. Consequently, both $\bar\eta_i\ge0$ and $\bar\nu_i\ge0$ hold on this branch.
	
	On the $z_i$-component equality branch,
	$$ g_i(x^k,\theta_k^*)=[g_i(x^k,y^k)]_+-\frac{\varepsilon_{z,i}^k}{\rho_1^k}. $$
	Thus,
	$\nu_i^k
	=\rho_1^k\bigl(g_i(x^k,y^k)-[g_i(x^k,y^k)]_+\bigr)+\varepsilon_{z,i}^k. $
	If $g_i(x^k,y^k)\ge0$, then $\nu_i^k=\varepsilon_{z,i}^k\to0$. If $g_i(x^k,y^k)<0$, applying the same boundedness argument as above yields $\rho_1^kg_i(x^k,y^k)\to0$, meaning $\nu_i^k=\rho_1^kg_i(x^k,y^k) + \varepsilon_{z,i}^k\to0$. Hence, $\bar\nu_i=0$ on this branch.
	
	On the $s_i$-component equality branch, we have $\zeta_i^k=\varepsilon_{s,i}^k\to0$. If $g_i(x^k,y^k)<0$, then $[g_i(x^k,y^k)]_+=0$. If $g_i(x^k,y^k)\ge0$, then since $s_i^k=0$, the relation $\zeta_i^k=2\rho_2^k(g_i(x^k,y^k)+s_i^k)$ provides $\zeta_i^k=2\rho_2^kg_i(x^k,y^k)$. Utilizing $\rho_1^k/\rho_2^k\to0$, this results in
	$$ 0\le \rho_1^k[g_i(x^k,y^k)]_+
	=\frac{\rho_1^k}{2\rho_2^k}\zeta_i^k\to0. $$
	Thus, regardless of the sign of $g_i(x^k,y^k)$, we find $\rho_1^k[g_i(x^k,y^k)]_+\to0$. Because $z_i^k=0$, we have $\lambda_{k,i}^*=\gamma_2[g_i(x^k,y^k)]_+$, giving
	$$ \eta_i^k=\rho_1^k\lambda_{k,i}^*+\zeta_i^k
	=\gamma_2\rho_1^k[g_i(x^k,y^k)]_++\zeta_i^k\to0. $$
	Consequently, $\bar\eta_i=0$ on this branch.
	
	Therefore, for every $i\in\mathcal I_{00}(\bar w)$, either $\bar\eta_i=0$, or $\bar\nu_i=0$, or both $\bar\eta_i\ge0$ and $\bar\nu_i\ge0$ hold. This implies
	$$ (\bar\eta_i>0\ \text{and}\ \bar\nu_i>0)
	\quad\text{or}\quad
	\bar\eta_i\bar\nu_i=0,
	\qquad i\in\mathcal I_{00}(\bar w). $$
	Thus, $\bar w$ is M-stationary for \eqref{eq:kkt-reform}. This completes the proof.
\end{proof}

\section{Algorithm and Convergence Analysis}\label{sec5}
Using the slack-based two-parameter formulation, we develop an inexact regularized gap-function bilevel slack-penalty method (iG-BSPM) for the slack reformulation \(\mathrm{(GP)}_{\mathcal C}\) in \eqref{reformulation_slack}. At outer iteration \(k\), the method approximately solves the following penalized subproblem:
\begin{equation}\label{penprob2}
	\min_{(x,y,z,s)\in\Omega\times \mathcal C}
	\psi_{\sigma_k}(x,y,z,s)
	:=
	\frac{1}{\sigma_{k,1}}F(x,y)+\mathcal G_\gamma(x,y,z)+\frac{\sigma_{k,2}}{2}\|g(x,y)+s\|^2,
\end{equation}
where $\mathcal C := \big\{ (z,s) \in \mathbb R_+^q\times\mathbb R_+^q \bigm| z_is_i=0,\ i=1,\ldots,q \big\}$ denotes the complementarity set, and $\sigma_k:=(\sigma_{k,1},\sigma_{k,2})$ is the penalty parameter at the outer iteration $k$.

For fixed \(\sigma_k\), the inner loop alternately updates \((x,y)\) and \((z,s)\), where \(z_i,s_i\le M_{z,k}\) are imposed for the convergence analysis for $i=1,\ldots,q$. The outer loop updates the penalty and bound parameters and applies a feasibility correction when needed. Throughout this section, we assume that \(\Omega\) is bounded and the gradient \(\nabla F\) is Lipschitz continuous on \(\Omega\).

\begin{assumption}\label{asup:bound}
	The set $\Omega= \big\{ (x,y) \in \mathbb R^n\times\mathbb R^m \bigm| G(x,y) \le 0 \big\}$ is nonempty, closed, convex, and bounded. 
\end{assumption}

\begin{assumption}\label{asup:F_lis_cts}
	The gradient \(\nabla F\) is \(L_F\)-Lipschitz continuous on \(\Omega\) for some \(L_F>0\).
\end{assumption}

\subsection{Algorithm Design}

The gradient of $\mathcal G_\gamma$ contains $\theta^*(x,y,z)$ from \eqref{eq:deftheta}.  Solving this subproblem exactly at each iteration may be expensive. We therefore compute $\theta^*(x,y,z)$ inexactly and update this approximation during the inner loop. The indices $k$ and $\ell$ denote the outer and inner iterations. Given $(x^k, y^k, z^k, s^k) \in \Omega\times\mathcal{C}$, the feasibility-corrected iterate $(x^k, \tilde{y}^k, \tilde{z}^k, \tilde{s}^k) \in \Omega\times\mathcal{C}$, and $\tilde{\theta}^k \in \mathbb R^m$, we initialize the inner loop as $(x^{k,0}, y^{k,0}, z^{k,0}, s^{k,0}, \theta^{k,0}) := (x^k, \tilde y^k, \tilde z^k, \tilde s^k, \tilde\theta^k)$. To keep the inner iterates bounded, $(z,s)$ is restricted to
\[
\mathcal C_{k} := \big\{ (z,s) \in \mathbb R_+^q \times \mathbb R_+^q \bigm| z_i s_i = 0,\ z_i \le M_{z,k},\ s_i \le M_{z,k},\ i=1,\ldots,q \big\}.
\]
The projection onto $\mathcal C_{k}$ admits a closed-form formula, and $M_{z,k}$ is fixed during the $k$-th inner loop.

\paragraph*{1. The Inner Loop: Alternating Projected Gradient Steps}

For a fixed $k$, we omit the superscript $k$ when there is no ambiguity. Thus, $(x^{k,\ell}, y^{k,\ell}, z^{k,\ell}, s^{k,\ell})$ is written as $(x^{\ell}, y^{\ell}, z^{\ell}, s^{\ell})$, and similarly for $\theta^\ell$, $\theta^{\ell+1/2}$, $d_{xy}^\ell$, $d_{zs}^\ell$, $\alpha_\ell$, and $\beta_\ell$.

\paragraph*{Update $(x,y)$ (Fixing $(z,s)$)}
We first update $(x,y)$ with $(z,s)$ fixed. The search direction is obtained from $\nabla_{x,y}\psi_{\sigma_k}$ at $(x^\ell,y^\ell,z^\ell,s^\ell)$, with $\theta^*(x^\ell,y^\ell,z^\ell)$ in \eqref{Ggradient} replaced by $\theta^\ell \in \mathbb R^m$. This gives
\begin{equation}\label{dxy_new}
	d_{xy}^{\ell} :=
	\frac{1}{\sigma_{k,1}}\nabla F(x^\ell, y^\ell)
	+
	\sigma_{k,2}(\nabla g(x^\ell,y^\ell))^\top
	\bigl(g(x^\ell,y^\ell)+s^\ell\bigr)
	+
	\begin{pmatrix}
		\nabla_x \mathcal L(x^\ell, y^\ell, \lambda^\ell) - \nabla_x \mathcal L(x^\ell, \theta^\ell, z^\ell) \\[4pt]
		\nabla_y \mathcal L(x^\ell, y^\ell, \lambda^\ell) - (y^{\ell}- \theta^{\ell}) / \gamma_1
	\end{pmatrix},
\end{equation}
where $\mathcal L(x,y,z):=f(x,y)+z^\top g(x,y)$ is the lower-level Lagrangian, and $\lambda^\ell$ is defined by
\[
\lambda^\ell
:=
\operatorname{Proj}_{\mathbb R_+^q}
\bigl( z^\ell + \gamma_2 g(x^\ell,y^\ell) \bigr).
\]
Given $d_{xy}^{\ell}$, the variables $(x,y)$ are updated via the projected gradient step:
\begin{equation}\label{update_xy_slack_text}
	(x^{\ell+1},y^{\ell+1})
	=
	\operatorname{Proj}_{\Omega}
	\bigl(
	(x^\ell,y^\ell) - \alpha_\ell d_{xy}^\ell
	\bigr),
\end{equation}
where $\alpha_\ell > 0$ is determined by the line search below. To avoid evaluating the exact Moreau envelope \(M_{\mathcal L}^{\gamma_1}(x,y,z^\ell)\) and the exact gap value \(\mathcal G_\gamma(x,y,z^\ell)\), or estimating their Lipschitz moduli, we employ a computable line search based on an adaptive descent condition. For any \(\theta\in\mathbb R^m\), define
\[
	\hat{\psi}_{\sigma_k}(x,y,z,s; \theta)
	:=
	\frac{1}{\sigma_{k,1}}F(x,y) + \hat{\mathcal G}_\gamma(x,y,z;\theta) + \frac{\sigma_{k,2}}{2}\|g(x,y)+s\|^2,
\]
where
\[
	\hat{\mathcal G}_\gamma(x,y,z;\theta)
	:=
	f(x,y)
	+ \mathcal P_{\gamma_2}(x,y,z)
	- f(x,\theta) - z^\top g(x,\theta)
	- \frac{1}{2\gamma_1}\|\theta-y\|^2.
\]
Then $d_{xy}^{\ell}=\nabla_{x,y}\hat{\psi}_{\sigma_k}(x^\ell,y^\ell,z^\ell,s^\ell;\theta^\ell)$. Let $0<\widehat L_{k,\ell}^{xy,0}\le \widehat{L}$, $c_\alpha > 0$, and $\kappa_L > 1$ be given. For $j = 0, 1, \ldots$, set $\widehat L_{k,\ell}^{xy,j} := \widehat L_{k,\ell}^{xy,0}\kappa_L^j$, $\alpha_\ell^j := 1/(\widehat L_{k,\ell}^{xy,j}+c_\alpha)$, and \((x^{\ell+1,j},y^{\ell+1,j}):=\operatorname{Proj}_{\Omega}\bigl((x^\ell,y^\ell)-\alpha_\ell^j d_{xy}^\ell\bigr)\). We accept the first index $j = j^{xy}_\ell$ for which the following descent condition is satisfied:
\begin{equation}\label{linesearch_x}
	\begin{aligned}
		&\quad \hat{\psi}_{\sigma_k}(x^{\ell+1,j},y^{\ell+1,j},z^\ell,s^\ell;\theta^\ell)
		+
		\frac{\gamma_1}{2} R(\theta^\ell;x^{\ell+1,j},y^{\ell+1,j},z^\ell)^2 \\
		&\le
		\hat{\psi}_{\sigma_k}(x^\ell,y^\ell,z^\ell,s^\ell;\theta^\ell)
		- c_\alpha \|(x^{\ell+1,j},y^{\ell+1,j})-(x^\ell,y^\ell)\|^2
		+
		\gamma_1 (\zeta^\theta_{k,\ell})^2,
	\end{aligned}
\end{equation}
where $\zeta^\theta_{k,\ell}$ controls the solution's inexactness. The stepsize in \eqref{update_xy_slack_text} is then set to \(\alpha_\ell:=\alpha_\ell^{j_\ell^{xy}}\). Here \(R\) denotes the first-order stationarity residual of the Moreau-envelope subproblem \eqref{def_moreau_envelope}, given by
\[
	R(\theta; x, y, z)
	:=
	\Big\|
	\nabla_y\mathcal L(x,\theta,z)
	+
	\frac{1}{\gamma_1}(\theta-y)
	\Big\|.
\]
Due to the strong convexity of \eqref{def_moreau_envelope} with respect to $\theta$, $R(\theta; x, y, z) = 0$ holds if and only if $\theta = \theta^*(x, y, z)$. 

\paragraph*{Update $(z,s)$ (Fixing $(x,y)$)}
After obtaining $(x^{\ell+1},y^{\ell+1})$, we update $(z,s)$ over the bounded complementarity set $\mathcal C_{k}$. For the prescribed tolerance \(\zeta^\theta_{k,\ell}\), we first compute \(\theta^{\ell+1/2}\) satisfying
\begin{equation}\label{theta_slack_half}
	R(\theta^{\ell+1/2}; x^{\ell+1}, y^{\ell+1}, z^\ell) \le \zeta^\theta_{k,\ell}.
\end{equation}

Replacing \(\theta^*(x^{\ell+1},y^{\ell+1},z^\ell)\) in \eqref{Ggradient} with \(\theta^{\ell+1/2}\), we define the computable direction
\begin{equation}\label{d_zs}
	d_{zs}^\ell
	:=
	\left(
	\nabla_z\mathcal P_{\gamma_2}(x^{\ell+1},y^{\ell+1},z^\ell) - g(x^{\ell+1},\theta^{\ell+1/2}),
	\;
	\sigma_{k,2}\bigl(g(x^{\ell+1},y^{\ell+1})+s^\ell\bigr)
	\right).
\end{equation}
The variables \((z,s)\) are updated by the projected-gradient step
\begin{equation}\label{update_zs_full}
	(z^{\ell+1},s^{\ell+1})
	=
	\operatorname{Proj}_{\mathcal C_k}
	\bigl(
	(z^\ell,s^\ell)-\beta_\ell d_{zs}^\ell
	\bigr),
\end{equation}
where $\beta_\ell > 0$ is determined by backtracking. Since $\mathcal C_k$ is nonconvex, its Euclidean projection may be set-valued. Throughout, $\operatorname{Proj}_{\mathcal C_k}$ denotes the Euclidean projection specified below; when the two candidate projections have the same distance, the first branch is selected. For any \((z,s)\in\mathbb R^q\times\mathbb R^q\), let $[t]_{[0,M_{z,k}]}:=\min\{M_{z,k},[t]_+\}$. Then for any $i=1,\ldots,q$, the \(i\)-th component pair of \(\operatorname{Proj}_{\mathcal C_k}(z,s)\) is given by
\[
\bigl(\operatorname{Proj}_{\mathcal C_k}(z,s)\bigr)_i
=
\begin{cases}
	\bigl([z_i]_{[0,M_{z,k}]},0\bigr),
	&
	\bigl(z_i-[z_i]_{[0,M_{z,k}]}\bigr)^2+s_i^2
	\le
	z_i^2+\bigl(s_i-[s_i]_{[0,M_{z,k}]}\bigr)^2,
	\\
	\bigl(0,[s_i]_{[0,M_{z,k}]}\bigr),
	&
	\text{otherwise},
\end{cases}.
\]
Let $0<\widehat L_{k,\ell}^{zs,0}\le \widehat{L}$, $c_\beta > 0$, and $\kappa_L > 1$ be given. For $j = 0, 1, \ldots$, set $\widehat L_{k,\ell}^{zs,j} := \widehat L_{k,\ell}^{zs,0}\kappa_L^j$, $\beta_\ell^j := 1/(\widehat L_{k,\ell}^{zs,j}+c_\beta)$, and \((z^{\ell+1,j},s^{\ell+1,j}):=\operatorname{Proj}_{\mathcal C_k}\bigl((z^\ell,s^\ell)-\beta_\ell^j d_{zs}^\ell\bigr)\). We accept the first index $j = j^{zs}_\ell$ for which the following condition holds:
\begin{equation}\label{linesearch_zs}
	\begin{aligned}
		&\quad \hat{\psi}_{\sigma_k}(x^{\ell+1},y^{\ell+1},z^{\ell+1,j},s^{\ell+1,j};\theta^{\ell+1/2})
		+
		\frac{\gamma_1}{2} R(\theta^{\ell+1/2};x^{\ell+1},y^{\ell+1},z^{\ell+1,j})^2 \\
		&\le
		\hat{\psi}_{\sigma_k}(x^{\ell+1},y^{\ell+1},z^\ell,s^\ell;\theta^{\ell+1/2})
		- c_\beta \|(z^{\ell+1,j},s^{\ell+1,j})-(z^\ell,s^\ell)\|^2
		+
		\gamma_1 (\zeta^\theta_{k,\ell})^2.
	\end{aligned}
\end{equation}
The stepsize in \eqref{update_zs_full} is then set to \(\beta_\ell:=\beta_\ell^{j_\ell^{zs}}\). We then choose an approximation $\theta^{\ell+1} \approx \theta^*(x^{\ell+1},y^{\ell+1},z^{\ell+1})$ satisfying the inner-iteration inexactness condition
\begin{equation}\label{theta_slack}
	R(\theta^{\ell+1}; x^{\ell+1}, y^{\ell+1}, z^{\ell+1}) \le \zeta^\theta_{k,\ell+1}.
\end{equation} 

\paragraph*{Inner Loop Termination}
The inner loop terminates at iteration $\ell$ once an approximate stationary point of the penalized subproblem \eqref{penprob2} is attained, certified by the following conditions:
\begin{equation}\label{inner_stop_slack}
	\begin{aligned}
		\|(x^{k,\ell+1},y^{k,\ell+1},z^{k,\ell+1},s^{k,\ell+1}) - (x^{k,\ell},y^{k,\ell},z^{k,\ell},s^{k,\ell})\| &\le \frac{\tau_k}{\sigma_{k,1}\sigma_{k,2}M_{z,k}},\\
		\sigma_{k,1} M_{z,k} \max\{\zeta^\theta_{k,\ell},\zeta^\theta_{k,\ell+1}\} &\le \tau_k,
	\end{aligned}
\end{equation}
where $\tau_k > 0$ is a prescribed tolerance parameter controlling the accuracy of the approximate stationarity and satisfies \(\tau_k\to0\) as \(k\to\infty\). Upon termination, we update the outer iterates by setting \((x^{k+1},y^{k+1},z^{k+1},s^{k+1}):=(x^{k,\ell+1},y^{k,\ell+1},z^{k,\ell+1},s^{k,\ell+1})\), \(\theta^{k+1}:=\theta^{k,\ell+1}\), and \(\zeta^\theta_{k+1,0}:=\zeta^\theta_{k,\ell+1}\).

\paragraph*{2. The Outer Loop: Penalty Update and Feasibility Correction}

\paragraph*{Penalty Parameter Update}
To avoid solving \eqref{def_moreau_envelope} when updating \(\sigma_k\), we estimate the constraint violation using \(\hat{\mathcal G}_\gamma(x,y,z;\theta)\), which replaces \(\theta^*(x,y,z)\) in \eqref{vg_eq2} with the available point \(\theta\):
\begin{equation}\label{def_t_slack}
	t^{k+1} :=
	\hat{\mathcal G}_\gamma(x^{k+1},y^{k+1},z^{k+1};\theta^{k+1}) + \frac{\sigma_{k,2}}{2} \|g(x^{k+1},y^{k+1})+s^{k+1}\|^2.
\end{equation}
The penalty residual test is subsequently evaluated as:
\begin{equation}\label{penalty_residual_test_slack}
	t^{k+1} \le c_\rho\tau_k,
\end{equation}
where \(c_\rho>0\). Let \(\varrho_\rho>0\). Based on this test, the penalty parameters are updated according to
\begin{equation}\label{update_rho_slack}
	\sigma_{k+1}
	:=
	\begin{cases}
		\sigma_k,
		&
		\text{if \eqref{penalty_residual_test_slack} holds},
		\\[3pt]
		(\sigma_{k,1}+\varrho_\rho,\sigma_{k,2}+\varrho_\rho),
		&
		\text{otherwise}.
	\end{cases}
\end{equation}

\paragraph*{Feasibility Correction Procedure}
When the penalty residual test \eqref{penalty_residual_test_slack} fails and \(\sigma_k\) is increased, a feasibility correction is applied to avoid convergence to an infeasible stationary point. 

Given a prescribed constant $\tilde c > 0$, we first check whether the current iterate satisfies
\begin{equation}\label{correction_not_needed_slack}
	\hat{\mathcal G}_\gamma(x^{k+1},y^{k+1},z^{k+1};\theta^{k+1}) + \frac{\sigma_{k,2}}{2}\|g(x^{k+1},y^{k+1})+s^{k+1}\|^2
	\le \frac{\tilde c}{\sigma_{k,1}}.
\end{equation}
If \eqref{correction_not_needed_slack} holds, no correction is performed. Otherwise, with \(x=x^{k+1}\) fixed, we approximately solve the lower-level problem \eqref{LL} subject additionally to the upper-level feasibility constraint \(G(x^{k+1},y)\le 0\):
\begin{equation}\label{inexacty_yz_subproblem}
	\min_{y \in \mathbb R^m} \quad f(x^{k+1},y) \quad 
	\mathrm{s.t.} \quad
	g(x^{k+1},y) \le 0, \quad G(x^{k+1},y) \le 0.
\end{equation}

Specifically, we compute an approximate pair $(\tilde y^{k+1}, \tilde z^{k+1})$, along with an approximation $\tilde{\theta}^{k+1} \approx \theta^*(x^{k+1}, \tilde y^{k+1}, \tilde z^{k+1})$, such that $(x^{k+1}, \tilde y^{k+1},\tilde z^{k+1}) \in \Omega\times \mathbb R_+^q$, $R(\tilde\theta^{k+1}; x^{k+1},\tilde y^{k+1},\tilde z^{k+1}) \le \zeta^\theta_{k+1,0}$, and the following inexactness criterion holds:
\begin{equation}\label{inexacty_yz_slack}
	\begin{aligned}
		&\quad \hat{\mathcal G}_\gamma(x^{k+1},\tilde y^{k+1},\tilde z^{k+1};\tilde \theta^{k+1}) 
		+ \frac{\gamma_1}{2}R(\tilde \theta^{k+1};x^{k+1},\tilde y^{k+1},\tilde z^{k+1})^2 \\
		&+\frac{\sigma_{k,2}}{2} \left(
		\|g_{I_+(\tilde z^{k+1})}
		(x^{k+1},\tilde y^{k+1})\|^2
		+
		\|[g_{I_0(\tilde z^{k+1})}
		(x^{k+1},\tilde y^{k+1})]_+\|^2
		\right)
		\le
		\frac{\tilde c}{\sigma_{k,1}},
	\end{aligned}
\end{equation}
where $I_+(z):=\{i\mid z_i>0\}$, $I_0(z):=\{i\mid z_i=0\}$.

To ensure that the correction procedure is well defined, we require that, for every \(x\in X\), there exists an upper-level feasible lower-level solution admitting a lower-level KKT multiplier.

\begin{assumption}\label{asup4}
	For every $x \in X$, there exist $y_x \in S(x) \cap \{y \in \mathbb{R}^m \mid (x,y) \in \Omega\}$ and $z_x\in\mathcal M(x,y_x)$.
\end{assumption}

To see this, fix \(x=x^{k+1}\) and choose \((y_x,z_x)\) as in Assumption~\ref{asup4}. Since \(z_x\in\mathcal M(x,y_x)\), optimality and complementarity give \(\theta^*(x,y_x,z_x)=y_x\), \(R(y_x;x,y_x,z_x)=0\), \(g_{I_+(z_x)}(x,y_x)=0\), and \([g_{I_0(z_x)}(x,y_x)]_+=0\), while Lemma~\ref{lem:gap_nonnegative_exact} gives \(\hat{\mathcal G}_\gamma(x,y_x,z_x;y_x)=\mathcal G_\gamma(x,y_x,z_x)=0\). Thus, the left-hand side of \eqref{inexacty_yz_slack} vanishes at \((\tilde y^{k+1},\tilde z^{k+1},\tilde\theta^{k+1})=(y_x,z_x,y_x)\), proving that the correction procedure is well defined.

 We then construct \(\tilde s^{k+1}\) component-wise as follows:
\[
\tilde{s}_i^{k+1} := 	
\begin{cases}
	0,
	&
	\text{if } \tilde z_i^{k+1} > 0,
	\\[3pt]
	[-g_i(x^{k+1},\tilde y^{k+1})]_+,
	&
	\text{otherwise},
\end{cases}
\qquad \forall\ i = 1, \ldots, q.
\]
The corrected tuple \((x^{k+1},\tilde y^{k+1},\tilde z^{k+1},\tilde s^{k+1})\) initializes the inner loop at outer iteration \(k+1\).

\paragraph*{Truncation-Bound Update}
At the conclusion of each outer iteration, update
\begin{equation}\label{update_Mzk_slack}
	M_{z,k+1}
	:=
	\max\left\{
	M_{z,k}+\Delta_z,\,
	\|\tilde z^{k+1}\|,\,
	\|\tilde s^{k+1}\|
	\right\},
\end{equation}
where \(\Delta_z>0\) is a prescribed constant. By \eqref{update_Mzk_slack}, \(M_{z,k+1}\ge M_{z,k}+\Delta_z\), and hence $M_{z,k} \to +\infty$ as $k \to \infty$.

Algorithm~\ref{alg2} summarizes the projected slack-penalty method.

\begin{algorithm}[!htb]
	\caption{Inexact Regularized Gap-Function Bilevel Slack-Penalty Method (iG-BSPM)}
	\label{alg2}
	\small
	\begin{algorithmic}[1]
		\STATE{\textbf{Input}: Initial iterates \((x^0,\tilde y^0,\tilde z^0,\tilde s^0,\tilde\theta^0)\) with \((x^0,\tilde y^0)\in\Omega\), \((\tilde z^0,\tilde s^0)\in\mathcal C_0\), \(\tilde\theta^0\in\mathbb R^m\), and \(R(\tilde\theta^0;x^0,\tilde y^0,\tilde z^0)\le \zeta^\theta_{0,0}\), penalty parameters \(\sigma_0=(\sigma_{0,1},\sigma_{0,2})\in\mathbb R_{++}^2\), constants \(\varrho_\rho>0\), \(c_\rho>0\), and \(\tilde c>0\), truncation parameters \(M_{z,0}>0\) and \(\Delta_z>0\), regularization parameters \(\gamma_1,\gamma_2>0\), line-search parameters \(c_\alpha,c_\beta>0\), \(\kappa_L>1\), and \(\widehat L>0\), together with a rule producing initial line-search estimates \(0<\widehat L_{k,\ell}^{xy,0}\le\widehat L\) and \(0<\widehat L_{k,\ell}^{zs,0}\le\widehat L\) for all \(k,\ell\), a sequence \(\{\tau_k\}\subset\mathbb R_{++}\) satisfying \(\tau_k\to0\), an initial tolerance \(\zeta^\theta_{0,0}\in\mathbb R_+\), a tail inexactness array \(\{\zeta^\theta_{k,\ell}\}_{k\ge0,\,\ell\ge1}\subset\mathbb R_+\) satisfying \(\sum_{\ell=1}^{\infty}(\zeta^\theta_{k,\ell})^2<+\infty\) for every \(k\) and \(\lim_{k\to\infty}\sum_{\ell=1}^{\infty}(\zeta^\theta_{k,\ell})^2=0\), and a stopping tolerance \(\varepsilon_{\mathrm{stop}}>0\).}
		\FOR{$k=0,\ 1,\cdots$}
		\STATE{Initialize the inner loop with \((x^{k,0},y^{k,0},z^{k,0},s^{k,0},\theta^{k,0}):=(x^k,\tilde y^k,\tilde z^k,\tilde s^k,\tilde\theta^k)\).}
		\STATE{\alglabel{Inner loop}}
		\FOR{$\ell=0,1,\ldots$}
		\STATE{Construct \(d_{xy}^{k,\ell}\) according to \eqref{dxy_new}; choose \(\alpha_{k,\ell}\) by the computable line search \eqref{linesearch_x}, and update \((x^{k,\ell+1},y^{k,\ell+1})\in\Omega\) by \eqref{update_xy_slack_text}.}
		\STATE{Find an inexact approximation \(\theta^{k,\ell+1/2}\in\mathbb R^m\) satisfying \eqref{theta_slack_half}.}
		\STATE{Construct \(d_{zs}^{k,\ell}\) according to \eqref{d_zs}; choose \(\beta_{k,\ell}\) by the computable line search \eqref{linesearch_zs}, and update \((z^{k,\ell+1},s^{k,\ell+1})\in\mathcal C_k\) by \eqref{update_zs_full}.}
		\STATE{Find an inexact approximation \(\theta^{k,\ell+1}\in\mathbb R^m\) satisfying \eqref{theta_slack}.}
		\IF{$\|(x^{k,\ell+1},y^{k,\ell+1},z^{k,\ell+1},s^{k,\ell+1})-(x^{k,\ell},y^{k,\ell},z^{k,\ell},s^{k,\ell})\|\le \tau_k/(\sigma_{k,1}\sigma_{k,2}M_{z,k})$ and $\sigma_{k,1}M_{z,k}\max\{\zeta^\theta_{k,\ell},\zeta^\theta_{k,\ell+1}\}\le\tau_k$}
		\STATE{\textbf{break}}
		\ENDIF
		\ENDFOR
		\STATE{Set \((x^{k+1},y^{k+1},z^{k+1},s^{k+1}):=(x^{k,\ell+1},y^{k,\ell+1},z^{k,\ell+1},s^{k,\ell+1})\), \(\theta^{k+1}:=\theta^{k,\ell+1}\), and \(\zeta^\theta_{k+1,0}:=\zeta^\theta_{k,\ell+1}\).}
		\STATE{Compute \(t^{k+1}\) by \eqref{def_t_slack}. }
		\STATE{If \(\max\{\tau_k,t^{k+1},\|g(x^{k+1},y^{k+1})+s^{k+1}\|\}\le \varepsilon_{\mathrm{stop}}\), set \((\tilde y^{k+1},\tilde z^{k+1},\tilde s^{k+1}):=(y^{k+1},z^{k+1},s^{k+1})\) and stop.}
		\STATE{\alglabel{Penalty update and correction step}}
		\STATE{\textbf{Update penalty parameters and apply feasibility correction:}}
		\STATE{\textbf{Case 1:} If \eqref{penalty_residual_test_slack} holds, set \(\sigma_{k+1}:=\sigma_k\) and \((\tilde y^{k+1},\tilde z^{k+1},\tilde s^{k+1},\tilde\theta^{k+1}):=(y^{k+1},z^{k+1},s^{k+1},\theta^{k+1})\).}
		\STATE{\textbf{Case 2:} If \eqref{penalty_residual_test_slack} fails but \eqref{correction_not_needed_slack} holds, set \(\sigma_{k+1}:=(\sigma_{k,1}+\varrho_\rho,\sigma_{k,2}+\varrho_\rho)\), \((\tilde y^{k+1},\tilde z^{k+1},\tilde s^{k+1},\tilde\theta^{k+1}):=(y^{k+1},z^{k+1},s^{k+1},\theta^{k+1})\).}
		\STATE{\textbf{Case 3:} If both \eqref{penalty_residual_test_slack} and \eqref{correction_not_needed_slack} fail, set \(\sigma_{k+1}:=(\sigma_{k,1}+\varrho_\rho,\sigma_{k,2}+\varrho_\rho)\); compute \((\tilde y^{k+1},\tilde z^{k+1},\tilde\theta^{k+1})\) such that $(x^{k+1},\tilde y^{k+1},\tilde z^{k+1})\in\Omega\times\mathbb R_+^q$, $R(\tilde\theta^{k+1};x^{k+1},\tilde y^{k+1},\tilde z^{k+1})
		\le \zeta^\theta_{k+1,0}$,	and \eqref{inexacty_yz_slack} holds, and construct \(\tilde s^{k+1}\) by the componentwise rule in the feasibility correction procedure. }	
		\STATE{Update \(M_{z,k+1}\) by \eqref{update_Mzk_slack}.}
		\ENDFOR
		\RETURN \((x^{k+1},\tilde y^{k+1},\tilde z^{k+1},\tilde s^{k+1})\).
	\end{algorithmic}
\end{algorithm}

\subsection{Preliminaries for Convergence Analysis}

We use the following constants in this subsection. Since \(\Omega\) is compact under Assumption~\ref{asup:bound}, let \(L_f\) and \(L_g\) be Lipschitz moduli of \(\nabla f\) and \(\nabla g\) on \(\Omega\), respectively, and let \(M_g:=\sup_{(x,y)\in\Omega}\|\nabla g(x,y)\|<+\infty\). Then \(M_g\) is also a Lipschitz modulus of \(g\) on \(\Omega\).

We first record that \(\theta^*(x,y,z)\) and the inexact points satisfying the residual bound \(R(\theta;x,y,z)\le\zeta^\theta\) remain uniformly bounded whenever \(z\) is restricted to a bounded set. 

\begin{lemma}\label{lem:theta-bounded}
	Suppose that Assumptions~\ref{asup_stationarity_lower} and~\ref{asup:bound} hold. Let \(Z\subset\mathbb R_+^q\) be a bounded set. Then, for any \(\zeta^\theta>0\), there exists \(M_\theta>0\) such that, for any \((x,y,z)\in\Omega\times Z\) and any \(\theta\in\mathbb R^m\) satisfying \(R(\theta;x,y,z)\le \zeta^\theta\), we have \(\|\theta\|\le M_\theta\).
\end{lemma}

\begin{proof}
	Fix \(\zeta^\theta>0\). For any \((x,y,z)\in\Omega\times Z\), let \(\theta^*=\theta^*(x,y,z)\). By Assumption~\ref{asup_stationarity_lower} and \(z\in\mathbb R_+^q\), the mapping \(\theta\mapsto\mathcal L(x,\theta,z)+\|\theta-y\|^2/(2\gamma_1)\) is \(1/\gamma_1\)-strongly convex. Hence \(\theta^*\) is uniquely defined and satisfies
	\[
		\nabla_y\mathcal L(x,\theta^*,z)+\frac{1}{\gamma_1}(\theta^*-y)=0.
	\]
	Applying the \(1/\gamma_1\)-strong monotonicity of its gradient at \(\theta^*\) and \(y\), we obtain
	\[
		\frac{1}{\gamma_1}\|\theta^*-y\|^2\le-\left\langle\nabla_y\mathcal L(x,y,z),\theta^*-y\right\rangle\le\|\nabla_y\mathcal L(x,y,z)\|\|\theta^*-y\|.
	\]
	Since \(\Omega\times Z\) is bounded and the relevant gradients are continuous, there exist \(M_y,M_{\mathcal L}>0\) such that \(\|y\|\le M_y\) and \(\|\nabla_y\mathcal L(x,y,z)\|\le M_{\mathcal L}\) on \(\Omega\times Z\). Hence \(\|\theta^*(x,y,z)\|\le M_y+\gamma_1M_{\mathcal L}\). If \(R(\theta;x,y,z)\le\zeta^\theta\), Lemma~\ref{error_bound} gives \(\|\theta-\theta^*(x,y,z)\|\le\gamma_1R(\theta;x,y,z)\le\gamma_1\zeta^\theta\). Therefore \(\|\theta\|\le M_y+\gamma_1M_{\mathcal L}+\gamma_1\zeta^\theta\), which proves the assertion.
\end{proof}

The next lemma establishes the Lipschitz continuity of the solution mapping \(z\mapsto\theta^*(x,y,z)\) on bounded subsets of \(\mathbb R_+^q\),
uniformly with respect to \((x,y)\in\Omega\).

\begin{lemma}\label{lem:theta-z-lipschitz}
	Suppose that Assumptions~\ref{asup_stationarity_lower} and~\ref{asup:bound} hold. Let \(Z\subset\mathbb R_+^q\) be a bounded set. Then, there exists $L_{\theta,z} \ge 0$ such that, for any fixed \((x,y)\in\Omega\) and any \(z_1,z_2\in Z\), 
	\[
	\|\theta^*(x,y,z_1)-\theta^*(x,y,z_2)\|
	\le
	L_{\theta,z}  \|z_1-z_2\|.
	\]
\end{lemma}

\begin{proof}
	By Lemma~\ref{lem:theta-bounded}, the points \(\theta^*(x,y,z)\), \((x,y,z)\in\Omega\times Z\), are bounded. Hence there exists \(M_{g,Z}>0\) such that \(\|\nabla_y g(x,\theta)\|\le M_{g,Z}\) for all relevant \((x,\theta)\). Fix \((x,y)\in\Omega\) and \(z_1,z_2\in Z\), and set \(\theta_i:=\theta^*(x,y,z_i)\), \(i=1,2\). The first-order optimality conditions give
	\[
		\nabla_y\mathcal L(x,\theta_i,z_i)+\frac{1}{\gamma_1}(\theta_i-y)=0,
		\qquad i=1,2.
	\]
	Applying the \(1/\gamma_1\)-strong monotonicity of this gradient at \(\theta_1\) and \(\theta_2\) yields
	\[
	\frac{1}{\gamma_1}\|\theta_1-\theta_2\|^2
	\le
	-\left\langle
	\nabla_y g(x,\theta_2)^\top(z_1-z_2),
	\theta_1-\theta_2
	\right\rangle
	\le
	M_{g,Z}\|z_1-z_2\|\|\theta_1-\theta_2\|.
	\]
	This implies \(\|\theta_1-\theta_2\|\le\gamma_1M_{g,Z}\|z_1-z_2\|\). Hence the assertion holds with \(L_{\theta,z}:=\gamma_1M_{g,Z}\).
\end{proof}

The next lemma gives a separate estimate for how the \((x,y)\)-gradient of the exact regularized gap function changes when only \(z\) is perturbed.

\begin{lemma}\label{lem:Gxy-z-lip}
	Suppose that Assumptions~\ref{asup_stationarity_lower} and~\ref{asup:bound} hold. Let \(Z\subset\mathbb R_+^q\) be a bounded set. Then, there exists \(L_{\nabla_{xy}\mathcal G}^{z}>0\) such that, for any fixed \((\bar x,\bar y)\in\Omega\) and any \(z_1,z_2\in Z\),
	\begin{equation}\label{lem9_eq5}
		\begin{aligned}
			\|\nabla_{x,y}\mathcal G_\gamma(\bar x,\bar y,z_1)
			-\nabla_{x,y}\mathcal G_\gamma(\bar x,\bar y,z_2)\|
			\le L_{\nabla_{xy}\mathcal G}^{z}\|z_1-z_2\|.
		\end{aligned}
	\end{equation}
\end{lemma}

\begin{proof}
	Let \(M_Z:=\sup_{z\in Z}\|z\|<+\infty\). By Lemma~\ref{lem:theta-bounded} applied to \(Z\), there exists \(M_{\theta,Z}>0\) such that \(\|\theta^*(x,y,z)\|\le M_{\theta,Z}\) for all \((x,y,z)\in\Omega\times Z\). Define
	\[
	\widehat\Omega_{\theta,Z}:=
	\left\{(x,\theta)\in\mathbb R^{n+m}\ \middle|\ 
	(x,y)\in\Omega\ \text{for some }y,\ \|\theta\|\le M_{\theta,Z}\right\}.
	\]
	Then \(\widehat\Omega_{\theta,Z}\) is bounded. Let \(M_{g,\theta,Z}\) be a uniform bound on \(\|\nabla g\|\) over \(\widehat\Omega_{\theta,Z}\), and let \(L_{\mathcal L,Z}>0\) be a Lipschitz modulus of \(\nabla_x\mathcal L(\cdot,\cdot,z)\), uniform for \(z\in Z\), on \(\widehat\Omega_{\theta,Z}\).
	
	Fix \((\bar x,\bar y)\in\Omega\) and \(z_1,z_2\in Z\). Set \(\theta_i^*:=\theta^*(\bar x,\bar y,z_i)\) and \(\lambda_i^*:=\lambda^*(\bar x,\bar y,z_i)\), \(i=1,2\). By the projection formula for \(\lambda^*\) and the nonexpansiveness of the projection onto \(\mathbb R_+^q\),
	\[
	\|\lambda_1^*-\lambda_2^*\|
	\le
	\|(z_1+\gamma_2g(\bar x,\bar y))-(z_2+\gamma_2g(\bar x,\bar y))\|
	\le
	\|z_1-z_2\|.
	\]
	Moreover, Lemma~\ref{lem:theta-z-lipschitz} gives the existence of \(L_{\theta,z}\) such that
	$
	\|\theta_1^*-\theta_2^*\|\le L_{\theta,z}\|z_1-z_2\|.
	$
	Using the \(x\)-component of the gradient formula \eqref{Ggradient}, we have
	\[
	\begin{aligned}
		&\quad \nabla_x\mathcal G_\gamma(\bar x,\bar y,z_1)-\nabla_x\mathcal G_\gamma(\bar x,\bar y,z_2)\\
		&=
		\nabla_xg(\bar x,\bar y)^\top(\lambda_1^*-\lambda_2^*)
		-
		\Bigl(
		\nabla_x\mathcal L(\bar x,\theta_1^*,z_1)
		-
		\nabla_x\mathcal L(\bar x,\theta_2^*,z_1)
		\Bigr)
		-
		\nabla_xg(\bar x,\theta_2^*)^\top(z_1-z_2).
	\end{aligned}
	\]
	Combining this identity with \(\|\lambda_1^*-\lambda_2^*\|\le\|z_1-z_2\|\) and \(\|\theta_1^*-\theta_2^*\|\le L_{\theta,z}\|z_1-z_2\|\) gives
	\begin{equation}\label{lem9_x_lip}
		\|\nabla_x\mathcal G_\gamma(\bar x,\bar y,z_1)-\nabla_x\mathcal G_\gamma(\bar x,\bar y,z_2)\|\le	\left(M_g+M_{g,\theta,Z}+L_{\mathcal L,Z}L_{\theta,z}\right)\|z_1-z_2\|.
	\end{equation}
	Similarly, the \(y\)-component of \eqref{Ggradient} gives
	\begin{equation}\label{lem9_y_lip}
		\|\nabla_y\mathcal G_\gamma(\bar x,\bar y,z_1)-\nabla_y\mathcal G_\gamma(\bar x,\bar y,z_2)\|\le M_g \|\lambda_1^*-\lambda_2^*\| + \frac{L_{\theta,z}}{\gamma_1}\|z_1-z_2\|\le \left(M_g+ \frac{L_{\theta,z}}{\gamma_1}\right)\|z_1-z_2\|.
	\end{equation}
	Combining \eqref{lem9_x_lip} and \eqref{lem9_y_lip} yields \eqref{lem9_eq5} for some constant \(L_{\nabla_{xy}\mathcal G}^{z}>0\).
\end{proof}

We now prepare the quadratic upper estimates for the computable function used in the inner loop.

\begin{lemma}\label{lem10}
	Suppose that Assumptions~\ref{asup_stationarity_lower}, \ref{asup:bound}, and~\ref{asup:F_lis_cts} hold. Let $\sigma = (\sigma_1, \sigma_2) > 0$, $\zeta^\theta \ge 0$, and let $Z \subset \mathbb{R}_+^q$ be a bounded set. Then, there exists a constant $L_{\hat\psi}^{xy} > 0$ such that, for all $(x, y), (\bar{x}, \bar{y}) \in \Omega$, all $\bar{z}, \bar{s} \in Z$, and any $\theta \in \mathbb{R}^m$ satisfying $R(\theta; \bar{x}, \bar{y}, \bar{z}) \le \zeta^\theta$, we have
	\begin{equation}\label{lem10_eq1}
		\hat{\psi}_{\sigma}(x, y, \bar{z}, \bar{s}; \theta)
		\le
		\hat{\psi}_{\sigma}(\bar{x}, \bar{y}, \bar{z}, \bar{s}; \theta)
		+
		\left\langle
		\nabla_{x,y} \hat{\psi}_{\sigma}(\bar{x}, \bar{y}, \bar{z}, \bar{s}; \theta), (x, y) - (\bar{x}, \bar{y})
		\right\rangle
		+ \frac{L_{\hat\psi}^{xy}}{2} \|(x, y) - (\bar{x}, \bar{y})\|^2.
	\end{equation}
	Furthermore, there exists a constant $L_{\hat\psi}^{zs} > 0$ such that, for all $(\bar{x}, \bar{y}) \in \Omega$, all $z, s, \bar{z}, \bar{s} \in Z$, and any $\theta \in \mathbb{R}^m$ satisfying $R(\theta; \bar{x}, \bar{y}, \bar{z}) \le \zeta^\theta$, we have
	\begin{equation}\label{lem10_eq2}
		\hat{\psi}_{\sigma}(\bar{x}, \bar{y}, z, s; \theta)
		\le
		\hat{\psi}_{\sigma}(\bar{x}, \bar{y}, \bar{z}, \bar{s}; \theta)
		+
		\left\langle
		\nabla_{z,s} \hat{\psi}_{\sigma}(\bar{x}, \bar{y}, \bar{z}, \bar{s}; \theta), (z, s) - (\bar{z}, \bar{s})
		\right\rangle
		+ \frac{L_{\hat\psi}^{zs}}{2} \|(z, s) - (\bar{z}, \bar{s})\|^2.
	\end{equation}
\end{lemma}

\begin{proof}
	By Lemma~\ref{lem:theta-bounded}, all points \(\theta\) satisfying the prescribed residual bound lie in a bounded set, uniformly over \((\bar x,\bar y,\bar z)\in\Omega\times Z\). We show that, uniformly over all such \(\bar z,\bar s\in Z\) and \(\theta\), the mapping \((x,y)\mapsto\hat\psi_\sigma(x,y,\bar z,\bar s;\theta)\) has a Lipschitz continuous gradient on \(\Omega\).
	Recall that
	\[
		\hat{\psi}_{\sigma}(x,y, z, s;\theta)
	=
	\frac{1}{\sigma_1}F(x,y)
	+
	f(x,y)
	+
	\mathcal{P}_{\gamma_2}(x,y, z)-\mathcal L(x,\theta,z)-
	\frac{1}{2\gamma_1}\|\theta-y\|^2
		+
		\frac{\sigma_2}{2}\|g(x,y)+ s\|^2.
	\]
	
	We first consider \(\mathcal P_{\gamma_2}\). Let \(M_\lambda:=\sup_{\bar z\in Z,(x,y)\in\Omega}\|(\bar z+\gamma_2g(x,y))_+\|<+\infty\). For any \((x_i,y_i)\in\Omega\), set \(\lambda_i:=(\bar z+\gamma_2g(x_i,y_i))_+\), \(i=1,2\). The nonexpansiveness of the projection onto \(\mathbb R_+^q\) gives
	\[
	\|\lambda_1 - \lambda_2\| \le \gamma_2 \|g(x_1,y_1) - g(x_2,y_2)\| \le \gamma_2 M_g \|(x_1,y_1) - (x_2,y_2)\|.
	\]
	Using the chain rule, the gradient is given by $
	\nabla_{x,y}\mathcal{P}_{\gamma_2}(x,y,\bar{z}) = \nabla g(x,y)^\top(\bar{z} + \gamma_2 g(x,y))_+.
	$
	Consequently,
	\[
	\begin{aligned}
		\left\| \nabla_{x,y}\mathcal{P}_{\gamma_2}(x_1,y_1,\bar{z}) - \nabla_{x,y}\mathcal{P}_{\gamma_2}(x_2,y_2,\bar{z}) \right\| 
		&\le \|\nabla g(x_1,y_1) - \nabla g(x_2,y_2)\| \|\lambda_1\| + \|\nabla g(x_2,y_2)\| \|\lambda_1 - \lambda_2\| \\
		&\le \left(L_g M_\lambda + \gamma_2 M_g^2\right) \|(x_1,y_1) - (x_2,y_2)\|.
	\end{aligned}
	\]
	Next, let \(L_{\mathcal L}^x>0\) be a uniform Lipschitz modulus of \(x\mapsto\nabla_x\mathcal L(x,\theta,\bar z)\) for admissible \(\theta\) and \(\bar z\in Z\). The function \(-\|\theta-y\|^2/(2\gamma_1)\) has a Lipschitz continuous gradient with respect to $y$, with modulus $1/\gamma_1$. 
	
	For the slack penalty, let \(M_Z:=\sup_{s\in Z}\|s\|\) and \(M_0:=\sup_{(x,y)\in\Omega}\|g(x,y)\|\). The same estimate as above, with \((\bar z+\gamma_2g(x,y))_+\) replaced by \(g(x,y)+\bar s\), gives
	\[
	\begin{aligned}
		&\quad \left\| \nabla_{x,y}\left(\frac{1}{2}\|g(x_1,y_1) + \bar{s}\|^2\right) - \nabla_{x,y}\left(\frac{1}{2}\|g(x_2,y_2) + \bar{s}\|^2\right) \right\|\\
		&\le \left(L_g(M_0 + M_Z) + M_g^2\right) \|(x_1,y_1) - (x_2,y_2)\|.
	\end{aligned}
	\]
	Combining the preceding estimates, for fixed \(\sigma=(\sigma_1,\sigma_2)\) we may take
	\begin{equation}\label{def_L^{xy}}
		L_{\hat\psi}^{xy}:=\frac{L_F}{\sigma_1}+L_f+
		\left(L_gM_\lambda+\gamma_2M_g^2\right)
		+
		L_{\mathcal L}^{x}
		+
		\frac{1}{\gamma_1}+
		\sigma_2\left(L_g(M_0+M_Z)+M_g^2\right).
	\end{equation}
	The descent lemma \cite[Lemma~5.7]{beck2017first} then yields \eqref{lem10_eq1} with Lipschitz modulus \(L_{\hat\psi}^{xy}\).

	We next prove \eqref{lem10_eq2}. For fixed \((\bar x,\bar y)\), the terms depending on \((z,s)\) are \(\mathcal P_{\gamma_2}(\bar x,\bar y,z)-z^\top g(\bar x,\theta)+\sigma_2\|g(\bar x,\bar y)+s\|^2/2\). Using the maximization representation
	\[
	\mathcal{P}_{\gamma_2}(\bar{x}, \bar{y}, z)
	=
	\max_{\lambda \in \mathbb{R}_+^q}
	\left\{
	\lambda^\top g(\bar{x}, \bar{y})
	-
	\frac{1}{2\gamma_2}\|\lambda - z\|^2
	\right\},
	\]
	since $-\mathcal{P}_{\gamma_2}$ is the partial minimum over $\lambda \in \mathbb{R}_+^q$ of a jointly convex function in $(\lambda, z)$, the function $z \mapsto \mathcal{P}_{\gamma_2}(\bar{x}, \bar{y}, z)$ is concave on $\mathbb{R}_+^q$. Consequently, the first-order concavity condition yields
	\[
	\mathcal{P}_{\gamma_2}(\bar{x}, \bar{y}, z) - z^\top g(\bar{x}, \theta)
	\le  \mathcal{P}_{\gamma_2}(\bar{x}, \bar{y}, \bar{z}) - \bar{z}^\top g(\bar{x}, \theta)+ \left\langle \nabla_z\mathcal{P}_{\gamma_2}(\bar{x}, \bar{y}, \bar{z}) - g(\bar{x}, \theta), z - \bar{z} \right\rangle.
	\]
	Next, the $s$-dependent term is a simple quadratic function, which satisfies the exact identity
	\[
	\frac{\sigma_2}{2}\|g(\bar{x}, \bar{y}) + s\|^2 - \frac{\sigma_2}{2}\|g(\bar{x}, \bar{y}) + \bar{s}\|^2
	= \sigma_2 \left\langle g(\bar{x}, \bar{y}) + \bar{s}, s - \bar{s} \right\rangle + \frac{\sigma_2}{2}\|s - \bar{s}\|^2.
	\]
	
	Combining the preceding two estimates and applying the definition of $\nabla_{z,s}\hat{\psi}_{\sigma}$, we obtain
	\[
	\hat{\psi}_{\sigma}(\bar{x}, \bar{y}, z, s; \theta) - \hat{\psi}_{\sigma}(\bar{x}, \bar{y}, \bar{z}, \bar{s}; \theta)
	\le \left\langle \nabla_{z,s}\hat{\psi}_{\sigma}(\bar{x}, \bar{y}, \bar{z}, \bar{s}; \theta), (z, s) - (\bar{z}, \bar{s}) \right\rangle  + \frac{\sigma_2}{2}\|s - \bar{s}\|^2.
	\]
	Then, the desired estimate \eqref{lem10_eq2} follows by choosing constant $L_{\hat\psi}^{zs}=\sigma_2$. 
\end{proof}

\subsection{Convergence Analysis of the Inner Loop}
In this subsection, we analyze the inner loop performed during a fixed outer iteration. Let the outer index $k$ be fixed. For notational simplicity, throughout this subsection, we write $(x^\ell,y^\ell,z^\ell,s^\ell):=(x^{k,\ell},y^{k,\ell},z^{k,\ell},s^{k,\ell})$.

The preceding local estimates imply finite termination of line searches and the following descent estimate.

\begin{lemma}\label{lem_descent}
	Suppose that Assumptions~\ref{asup_stationarity_lower}, \ref{asup:bound}, and~\ref{asup:F_lis_cts} hold. Fix \(k\ge0\), and let\linebreak[4] \(\{(x^{k,\ell},y^{k,\ell},z^{k,\ell},s^{k,\ell},\theta^{k,\ell})\}_{\ell\ge0}\), together with the intermediate points \(\{\theta^{k,\ell+1/2}\}_{\ell\ge0}\), be generated by the inner loop of Algorithm~\ref{alg2} with fixed parameters \(\sigma_k\) and \(M_{z,k}\). Assume further that \(\sum_{\ell=0}^{\infty} (\zeta^\theta_{k,\ell})^2 < \infty\). Then, the line search procedures for determining the step sizes $\alpha_\ell$ and $\beta_\ell$ are well-defined; that is, the inequalities \eqref{linesearch_x} and \eqref{linesearch_zs} are satisfied for the finite indices $j=j_\ell^{xy}$ and $j=j_\ell^{zs}$, respectively. Moreover, for every inner iteration $\ell$,
	\begin{equation}\label{eq:psi_descent_combined}
		\begin{aligned}
			&\quad\psi_{\sigma_k}(x^{\ell+1},y^{\ell+1},z^{\ell+1},s^{\ell+1}) + c_\alpha \|(x^{\ell+1},y^{\ell+1})-(x^\ell,y^\ell)\|^2 + c_\beta \|(z^{\ell+1},s^{\ell+1})-(z^\ell,s^\ell)\|^2  \\
			&\le \psi_{\sigma_k}(x^\ell,y^\ell,z^\ell,s^\ell) + 2 \gamma_1 (\zeta^\theta_{k,\ell})^2.
		\end{aligned}
	\end{equation}
\end{lemma}

\begin{proof}
	Fix an inner iteration \(\ell\) and a candidate backtracking index \(j\). To simplify notation in the line-search analysis, we temporarily write the trial points \((x^{\ell+1,j},y^{\ell+1,j})\) and \((z^{\ell+1,j},s^{\ell+1,j})\) as \((x^{\ell+1},y^{\ell+1})\) and \((z^{\ell+1},s^{\ell+1})\), respectively. These points belong to \(\Omega\times\mathcal C_k\). Since $R(\theta^\ell;x^\ell,y^\ell,z^\ell)\le \zeta^\theta_{k,\ell}$, $R(\theta^{\ell+1/2};x^{\ell+1},y^{\ell+1},z^\ell)\le \zeta^\theta_{k,\ell}$, and \(\sup_{\ell\ge0}\zeta^\theta_{k,\ell}<+\infty\), Lemma~\ref{lem:theta-bounded} gives a bound, independent of \(\ell\) and the trial step sizes, for \(\theta^\ell\) and \(\theta^{\ell+1/2}\). Hence, Lemma~\ref{lem10}, applied with \(\sigma=\sigma_k\) and $
	\{ z\in\mathbb R_+^q \mid 0\le z_i\le M_{z,k},\ i=1,\ldots,q\}$, provides constants \(L_{\hat\psi}^{xy},L_{\hat\psi}^{zs}>0\), independent of \(\ell\) and the trial step sizes.
	
	For the \((x,y)\)-update, Lemma~\ref{lem10} and the definition of \(d_{xy}^\ell\) in \eqref{dxy_new} give 
	\begin{equation}\label{eq:xy_hatpsi_upper}
		\begin{aligned}
			&\quad\hat\psi_{\sigma_k}
			(x^{\ell+1},y^{\ell+1},z^\ell,s^\ell;\theta^\ell)
			-
			\hat\psi_{\sigma_k}
			(x^\ell,y^\ell,z^\ell,s^\ell;\theta^\ell)
			\\
			&\le
			\left\langle
			d_{xy}^\ell,
			(x^{\ell+1},y^{\ell+1})-(x^\ell,y^\ell)
			\right\rangle
			+
			\frac{L_{\hat\psi}^{xy}}{2}
			\left\|
			(x^{\ell+1},y^{\ell+1})-(x^\ell,y^\ell)
			\right\|^2.
		\end{aligned}
	\end{equation}
	The projection formula \eqref{update_xy_slack_text} and the convexity of $\Omega$ imply the inequality
	\begin{equation}\label{xy_descent_prelim}
		\left\langle d_{xy}^\ell, (x^{\ell+1},y^{\ell+1})-(x^\ell,y^\ell) \right\rangle + \frac{1}{\alpha_\ell} \|(x^{\ell+1},y^{\ell+1})-(x^\ell,y^\ell)\|^2 \le 0.
	\end{equation}
	
	Combining \eqref{eq:xy_hatpsi_upper} with \eqref{xy_descent_prelim}, we obtain
	\begin{equation}\label{xy_combined_descent}
		\hat{\psi}_{\sigma_k}(x^{\ell+1},y^{\ell+1},z^\ell,s^\ell;\theta^\ell) - \hat{\psi}_{\sigma_k}(x^\ell,y^\ell,z^\ell,s^\ell;\theta^\ell)
		\le - \left( \frac{1}{\alpha_\ell} - \frac{L_{\hat\psi}^{xy}}{2} \right) \|(x^{\ell+1},y^{\ell+1})-(x^\ell,y^\ell)\|^2.
	\end{equation}
	
	Since \(\Omega\) is compact and \(\{z^\ell\}\) and \(\{\theta^\ell\}\) are bounded, the twice continuous differentiability of \(f\) and \(g_i\), ($i=1,\ldots,q$), implies that, for some \(L_R^{xy}>0\) independent of \(\ell\), the mapping $(x,y) \mapsto \nabla_y\mathcal L(x,\theta,z) + (\theta-y)/\gamma_1$ is uniformly Lipschitz continuous on \(\Omega\) over all admissible \((\theta,z)\). Using the definition of $R$, the reverse triangle inequality $\big|\|a\|-\|b\|\big| \le \|a-b\|$, and the Lipschitz estimate above, we obtain
	\begin{equation}\label{eq:theta_residual_step_bound}
		\begin{aligned}
			\frac{\gamma_1}{2}R(\theta^\ell;x^{\ell+1},y^{\ell+1},z^\ell)^2
			&\le
			\frac{\gamma_1}{2}
			\Bigl(
			R(\theta^\ell;x^\ell,y^\ell,z^\ell)
			+
			L_R^{xy}\|(x^{\ell+1},y^{\ell+1})-(x^\ell,y^\ell)\|
			\Bigr)^2  \\
			&\le
			\gamma_1R(\theta^\ell;x^\ell,y^\ell,z^\ell)^2
			+
			\gamma_1(L_R^{xy})^2
			\|(x^{\ell+1},y^{\ell+1})-(x^\ell,y^\ell)\|^2 ,
		\end{aligned}
	\end{equation}
	where the second inequality follows from $(a+b)^2\le 2a^2+2b^2$. Adding \eqref{xy_combined_descent} to \eqref{eq:theta_residual_step_bound}, and since $R(\theta^\ell; x^\ell, y^\ell, z^\ell) \le \zeta^\theta_{k,\ell}$, we have
	\begin{equation}\label{eq:xy_linesearch_sufficient_bound}
		\begin{aligned}
			&\quad\hat{\psi}_{\sigma_k}(x^{\ell+1},y^{\ell+1},z^\ell,s^\ell;\theta^\ell) + \frac{\gamma_1}{2} R(\theta^\ell;x^{\ell+1},y^{\ell+1},z^\ell)^2 \\
			&\le \hat{\psi}_{\sigma_k}(x^\ell,y^\ell,z^\ell,s^\ell;\theta^\ell) - \Big( \frac{1}{\alpha_\ell} - \frac{L_{\hat\psi}^{xy}}{2} - \gamma_1(L_R^{xy})^2 \Big) \|(x^{\ell+1},y^{\ell+1})-(x^\ell,y^\ell)\|^2 + \gamma_1 (\zeta^\theta_{k,\ell})^2.
		\end{aligned}
	\end{equation}
	
	Let \(L_\alpha^*:=L_{\hat\psi}^{xy}/2+\gamma_1(L_R^{xy})^2\). If \(\widehat L_{k,\ell}^{xy,j}\ge L_\alpha^*\), then \(1/\alpha_\ell^j-L_\alpha^*=\widehat L_{k,\ell}^{xy,j}+c_\alpha-L_\alpha^*\ge c_\alpha\), so \eqref{eq:xy_linesearch_sufficient_bound} implies \eqref{linesearch_x}. Since \(\widehat L_{k,\ell}^{xy,j}=\widehat L_{k,\ell}^{xy,0}\kappa_L^j\) with \(\kappa_L>1\), the \((x,y)\)-line search terminates finitely. For the accepted step, \eqref{linesearch_x}, Lemma~\ref{error_bound}, and \(\hat{\psi}_{\sigma_k}(x^\ell,y^\ell,z^\ell,s^\ell;\theta^\ell)\le\psi_{\sigma_k}(x^\ell,y^\ell,z^\ell,s^\ell)\) yield
	\begin{equation}\label{eq:xy_psi_descent_accepted}
		\psi_{\sigma_k}(x^{\ell+1},y^{\ell+1},z^\ell,s^\ell) \le \psi_{\sigma_k}(x^\ell,y^\ell,z^\ell,s^\ell) - c_\alpha \|(x^{\ell+1},y^{\ell+1})-(x^\ell,y^\ell)\|^2 + \gamma_1 (\zeta^\theta_{k,\ell})^2.
	\end{equation}
	
	For the \((z,s)\)-update, a similar argument applies. Since the current point \((z^\ell,s^\ell)\) belongs to \(\mathcal C_k\) and the trial point is a selected global Euclidean projection, the projection inequality gives
	\begin{equation}\label{zs_projection_descent_prelim}
		\left\langle d_{zs}^\ell,
		(z^{\ell+1},s^{\ell+1})-(z^\ell,s^\ell)\right\rangle
		+
		\frac{1}{2\beta_\ell}
		\|(z^{\ell+1},s^{\ell+1})-(z^\ell,s^\ell)\|^2
		\le0.
	\end{equation}
	Combining \eqref{zs_projection_descent_prelim} with Lemma~\ref{lem10}, \eqref{d_zs}, and \eqref{update_zs_full}, we obtain
	\begin{equation}\label{eq:zs_hatpsi_descent}
		\begin{aligned}
			&\quad \hat{\psi}_{\sigma_k}(x^{\ell+1},y^{\ell+1},z^{\ell+1},s^{\ell+1};\theta^{\ell+1/2})\\
			&\le \hat{\psi}_{\sigma_k}(x^{\ell+1},y^{\ell+1},z^\ell,s^\ell;\theta^{\ell+1/2})-\left(\frac{1}{2\beta_\ell}-\frac{L_{\hat\psi}^{zs}}{2}\right)\|(z^{\ell+1},s^{\ell+1})-(z^\ell,s^\ell)\|^2.
		\end{aligned}
	\end{equation}
	Moreover, since $\nabla_y\mathcal L(x^{\ell+1},\theta^{\ell+1/2},z^{\ell+1})-\nabla_y\mathcal L(x^{\ell+1},\theta^{\ell+1/2},z^\ell)
	=\nabla_y g(x^{\ell+1},\theta^{\ell+1/2})^\top(z^{\ell+1}-z^\ell)$, the boundedness established above gives \(L_R^z>0\), independent of \(\ell\), such that
	\begin{equation}\label{eq:zs_residual_trial_bound}
		\frac{\gamma_1}{2}R(\theta^{\ell+1/2};x^{\ell+1},y^{\ell+1},z^{\ell+1})^2
		\le\gamma_1(\zeta^\theta_{k,\ell})^2+\gamma_1(L_R^z)^2\|(z^{\ell+1},s^{\ell+1})-(z^\ell,s^\ell)\|^2.
	\end{equation}
	Combining \eqref{eq:zs_hatpsi_descent} and \eqref{eq:zs_residual_trial_bound} yields
	\begin{equation}\label{eq:zs_linesearch_sufficient_bound}
		\begin{aligned}
			&\quad \hat{\psi}_{\sigma_k}(x^{\ell+1},y^{\ell+1},z^{\ell+1},s^{\ell+1};\theta^{\ell+1/2})
			+\frac{\gamma_1}{2}R(\theta^{\ell+1/2};x^{\ell+1},y^{\ell+1},z^{\ell+1})^2\\
			&\le\hat{\psi}_{\sigma_k}(x^{\ell+1},y^{\ell+1},z^\ell,s^\ell;\theta^{\ell+1/2})
			-\left(\frac{1}{2\beta_\ell}-\frac{L_{\hat\psi}^{zs}}{2}-\gamma_1(L_R^z)^2\right)
			\|(z^{\ell+1},s^{\ell+1})-(z^\ell,s^\ell)\|^2+\gamma_1(\zeta^\theta_{k,\ell})^2.
		\end{aligned}
	\end{equation}
	Let \(L_\beta^*:=L_{\hat\psi}^{zs}+2\gamma_1(L_R^z)^2+c_\beta\). By \eqref{eq:zs_linesearch_sufficient_bound}, \eqref{linesearch_zs} holds whenever \(\widehat L_{k,\ell}^{zs,j}\ge L_\beta^*\); hence, the \((z,s)\)-line search terminates finitely. For the accepted step, \eqref{linesearch_zs}, \(\hat{\psi}_{\sigma_k}(\cdot;\theta^{\ell+1/2})\le\psi_{\sigma_k}\), and Lemma~\ref{error_bound} yield
	\begin{equation}\label{eq:zs_psi_descent_accepted}
		\psi_{\sigma_k}(x^{\ell+1},y^{\ell+1},z^{\ell+1},s^{\ell+1})
		\le
		\psi_{\sigma_k}(x^{\ell+1},y^{\ell+1},z^\ell,s^\ell)
		-
		c_\beta\|(z^{\ell+1},s^{\ell+1})-(z^\ell,s^\ell)\|^2
		+
		\gamma_1(\zeta^\theta_{k,\ell})^2 .
	\end{equation}
	Adding \eqref{eq:xy_psi_descent_accepted} and \eqref{eq:zs_psi_descent_accepted} gives \eqref{eq:psi_descent_combined}.
\end{proof}

The descent estimate implies finite termination of the inner loop.
\begin{lemma}\label{lem_inner_finite_slack}
	Suppose that Assumptions~\ref{asup_stationarity_lower}, \ref{asup:bound}, and~\ref{asup:F_lis_cts} hold. Fix \(k\ge0\), and let\linebreak[4] \(\{(x^{k,\ell},y^{k,\ell},z^{k,\ell},s^{k,\ell},\theta^{k,\ell})\}_{\ell\ge0}\), together with the intermediate points \(\{\theta^{k,\ell+1/2}\}_{\ell\ge0}\), be generated by the inner loop of Algorithm~\ref{alg2} with fixed parameters \(\sigma_k\) and \(M_{z,k}\). Assume further that \(\sigma_k\in\mathbb R_{++}^2\), \(\tau_k>0\), \(M_{z,k}>0\), and
	$
	\sum_{\ell=0}^{\infty}(\zeta^\theta_{k,\ell})^2<+\infty .
	$
	Then the inner-loop termination criterion \eqref{inner_stop_slack} is satisfied after finitely many inner iterations. Consequently, the termination index of the \(k\)-th inner loop in Algorithm~\ref{alg2} is well defined and finite.
\end{lemma}

\begin{proof}
	By Lemma~\ref{lem_descent}, the line searches terminate finitely and \eqref{eq:psi_descent_combined} holds. Since \(\psi_{\sigma_k}\) is continuous on the compact set \(\Omega\times\mathcal C_k\), define \(\underline\psi_k:=\min_{\Omega\times\mathcal C_k}\psi_{\sigma_k}>-\infty\). Summing \eqref{eq:psi_descent_combined} over \(\ell=0,\ldots,N-1\) yields
	\[
	\begin{aligned}
		&\quad\sum_{\ell=0}^{N-1}
		\left[
		c_\alpha
		\|(x^{\ell+1},y^{\ell+1})-(x^\ell,y^\ell)\|^2
		+
		c_\beta
		\|(z^{\ell+1},s^{\ell+1})-(z^\ell,s^\ell)\|^2
		\right]
		\\
		&\le
		\psi_{\sigma_k}(x^0,y^0,z^0,s^0)
		-
		\psi_{\sigma_k}(x^N,y^N,z^N,s^N)
		+
		2\gamma_1
		\sum_{\ell=0}^{N-1}(\zeta^\theta_{k,\ell})^2\\
		&\le
		\psi_{\sigma_k}(x^0,y^0,z^0,s^0)
		-
		\underline\psi_k
		+
		2\gamma_1
		\sum_{\ell=0}^{\infty}(\zeta^\theta_{k,\ell})^2
		<+\infty .
	\end{aligned}
	\]
	Letting \(N\to\infty\) and using \(c_\alpha>0\) and \(c_\beta>0\), we obtain
	\[
	\sum_{\ell=0}^{\infty}
	\|(x^{\ell+1},y^{\ell+1},z^{\ell+1},s^{\ell+1})-(x^\ell,y^\ell,z^\ell,s^\ell)\|^2
	<+\infty,\ \|(x^{\ell+1},y^{\ell+1},z^{\ell+1},s^{\ell+1})-(x^\ell,y^\ell,z^\ell,s^\ell)\|\to0 .
	\]
	Also \(\zeta^\theta_{k,\ell}\to0\). Since \(\sigma_k\), \(M_{z,k}\), and \(\tau_k\) are fixed during the \(k\)-th inner loop, both conditions in \eqref{inner_stop_slack} hold for all sufficiently large \(\ell\). Hence the inner loop terminates finitely.
\end{proof}

\subsection{Convergence to Approximate KKT Points}

We next prove feasibility of accumulation points generated by Algorithm~\ref{alg2}. 

\begin{proposition}\label{prop:slack_feasibility} 
	Suppose that Assumptions~\ref{asup_stationarity_lower}, \ref{asup:bound}, \ref{asup:F_lis_cts}, and~\ref{asup4} hold. Let \(\{(x^k,y^k,z^k,s^k)\}\) be generated by Algorithm~\ref{alg2}. Assume further that $\tau_k \to 0$ as $k \to \infty$, and that the sequence of inexactness parameters $\{\zeta^\theta_{k,\ell}\}$ satisfies
	\[
	\sum_{\ell} (\zeta^\theta_{k,\ell})^2 < \infty \quad \text{for all } k, \qquad \text{and} \qquad \lim_{k \to \infty} \sum_{\ell} (\zeta^\theta_{k,\ell})^2 = 0.
	\]
	Then, every accumulation point $(\bar x, \bar y, \bar z, \bar s)$ of $\{(x^k, y^k, z^k, s^k)\}$ is feasible for the slack reformulation \eqref{reformulation_slack}. That is, it satisfies $(\bar x, \bar y) \in \Omega$, $(\bar z, \bar s) \in \mathcal C$, $\mathcal G_\gamma(\bar x, \bar y, \bar z) = 0$, and $g(\bar x, \bar y) + \bar s = 0$. 
\end{proposition}

\begin{proof}
	Let \((\bar x,\bar y,\bar z,\bar s)\) be any accumulation point of $\{(x^{k+1}, y^{k+1},z^{k+1},s^{k+1})\}$, and let \(\ell_k\) be the finite termination index of the \(k\)-th inner loop, whose existence follows from Lemma~\ref{lem_inner_finite_slack}. Thus
	\[
	(x^{k+1},y^{k+1},z^{k+1},s^{k+1})
	=
	(x^{k,\ell_k+1},y^{k,\ell_k+1},z^{k,\ell_k+1},s^{k,\ell_k+1}).
	\]
	The closedness of \(\Omega\) and \(\mathcal C\), together with \((x^{k+1},y^{k+1})\in\Omega\) and \((z^{k+1},s^{k+1})\in\mathcal C_k\subseteq\mathcal C\), gives \((\bar x,\bar y,\bar z,\bar s)\in\Omega\times\mathcal C\).
	
	Under Assumption~\ref{asup:bound}, \(F\) is bounded below on \(\Omega\). Let $\underline F:=\inf_{(x,y)\in\Omega}F(x,y)>-\infty$, and define 
	\[
	\tilde{\psi}_{\sigma_k}(x,y,z,s)
	:=
	\psi_{\sigma_k}(x,y,z,s)-\frac{\underline F}{\sigma_{k,1}}
	=
	\frac{F(x,y)-\underline F}{\sigma_{k,1}}
	+\mathcal G_\gamma(x,y,z)
	+\frac{\sigma_{k,2}}{2}\|g(x,y)+s\|^2 .
	\]
	The three terms in \(\tilde{\psi}_{\sigma_k}\) are nonnegative on \(\Omega\times\mathcal C\), because \(\mathcal G_\gamma\ge0\) and \(\sigma_{k,1},\sigma_{k,2}>0\).
	
	We estimate the corrected point used to initialize the \(k\)-th inner loop. Let \(\ell_{k-1}\) be the termination index of the \((k-1)\)-st inner loop. If \eqref{penalty_residual_test_slack} holds at the end of the \((k-1)\)-st outer iteration, then
	\[
		\hat{\mathcal G}_\gamma(x^k,\tilde y^k,\tilde z^k;\tilde\theta^k)
		+
		\frac{\sigma_{k-1,2}}{2} \|g(x^k,\tilde y^k)+\tilde s^k\|^2
		=t^k
		\le c_\rho\tau_{k-1}.
	\]
	Using Lemma~\ref{error_bound} and \(R(\tilde\theta^k;x^k,\tilde y^k,\tilde z^k)\le\zeta^\theta_{k-1,\ell_{k-1}+1}\), we obtain 
	\begin{equation}\label{eq:penalty_test_exact_bound}
		\mathcal G_\gamma(x^k,\tilde y^k,\tilde z^k)
		+
		\frac{\sigma_{k-1,2}}{2}\|g(x^k,\tilde y^k)+\tilde s^k\|^2
		\le
		c_\rho\tau_{k-1}
		+
		\frac{\gamma_1}{2}(\zeta^\theta_{k-1,\ell_{k-1}+1})^2.
	\end{equation}
	If the penalty residual test \eqref{penalty_residual_test_slack} fails but the correction is not needed, then \eqref{correction_not_needed_slack} gives
	\[
		\hat{\mathcal G}_\gamma(x^k,\tilde y^k,\tilde z^k;\tilde\theta^k)
		+
		\frac{\sigma_{k-1,2}}{2}\|g(x^k,\tilde y^k)+\tilde s^k\|^2
		\le
		\frac{\tilde c}{\sigma_{k-1,1}},
	\]
	and Lemma~\ref{error_bound} gives
	\begin{equation}\label{eq:no_correction_exact_bound}
		\mathcal G_\gamma(x^k,\tilde y^k,\tilde z^k)
		+
		\frac{\sigma_{k-1,2}}{2}	\|g(x^k,\tilde y^k)+\tilde s^k\|^2
		\le
		\frac{\tilde c}{\sigma_{k-1,1}}
		+
		\frac{\gamma_1}{2}(\zeta^\theta_{k-1,\ell_{k-1}+1})^2.
	\end{equation}
	If the correction is performed, then \eqref{inexacty_yz_slack}, Lemma~\ref{error_bound}, and the definition of \(\tilde s^k\) imply
	\begin{equation}\label{eq:correction_exact_bound}
		\begin{aligned}
			&\quad	\mathcal G_\gamma(x^k,\tilde y^k,\tilde z^k) +
			\frac{\sigma_{k-1,2}}{2}	\|g(x^k,\tilde y^k)+\tilde s^k\|^2 \\ &\le	\hat{\mathcal G}_\gamma(x^k,\tilde y^k,\tilde z^k;\tilde\theta^k) +  \frac{\gamma_1}{2} R(\tilde\theta^k;x^k,\tilde y^k,\tilde z^k)^2 +
			\frac{\sigma_{k-1,2}}{2}	\|g(x^k,\tilde y^k)+\tilde s^k\|^2 
			\le
			\frac{\tilde c}{\sigma_{k-1,1}}.
		\end{aligned}
	\end{equation}
	
	If \(\{\sigma_k\}\) is bounded, then \eqref{update_rho_slack} implies that \eqref{penalty_residual_test_slack} fails only finitely many times. Hence, for all sufficiently large \(k\), the penalty residual test holds. Applying \eqref{eq:penalty_test_exact_bound} with its index \(k\) replaced by \(k+1\) yields
	\begin{equation}\label{eq:bounded_penalty_terminal_residual}
		\mathcal G_\gamma(x^{k+1},y^{k+1},z^{k+1})
		+
		\frac{\sigma_{k,2}}{2}	\|g(x^{k+1},y^{k+1})+s^{k+1}\|^2
		\le
		c_\rho\tau_k
		+
		\frac{\gamma_1}{2}(\zeta^\theta_{k,\ell_k+1})^2
		\to0 .
	\end{equation}
	
	It remains to consider the case in which \(\{\sigma_k\}\) is unbounded. By  \eqref{update_rho_slack}, this is equivalent to \(\sigma_{k,1}\to+\infty\) and \(\sigma_{k,2}\to+\infty\). We claim that $\tilde{\psi}_{\sigma_k}(x^k,\tilde y^k,\tilde z^k,\tilde s^k)\to0$. Indeed, by boundedness of \(F\) on \(\Omega\), there exists \(M_F>0\) such that
	\(F(x,y)-\underline F\le M_F\) for all \((x,y)\in\Omega\). In all possible branches,
	\eqref{eq:penalty_test_exact_bound}, \eqref{eq:no_correction_exact_bound}, and
	\eqref{eq:correction_exact_bound} imply
		\[
	\tilde{\psi}_{\sigma_k}(x^k,\tilde y^k,\tilde z^k,\tilde s^k)
	\le
	\frac{M_F}{\sigma_{k,1}}
	+
	\frac{\sigma_{k,2}}{\sigma_{k-1,2}}
	\left[
	\max\left\{
	c_\rho\tau_{k-1},
	\frac{\tilde c}{\sigma_{k-1,1}}
	\right\}
	+
	\frac{\gamma_1}{2}
	(\zeta^\theta_{k-1,\ell_{k-1}+1})^2
	\right].
	\]
	By the penalty update rule, $1\le\sigma_{k,2}/\sigma_{k-1,2}\le1+\varrho_\rho/\sigma_{k-1,2}\to1$. The right-hand side tends to zero because \(1/\sigma_{k-1,1}\to0\), \(\tau_k\to0\), and \(\sum_{\ell}(\zeta^\theta_{k,\ell})^2\to0\). This proves $\tilde{\psi}_{\sigma_k}(x^k,\tilde y^k,\tilde z^k,\tilde s^k)\to0$.
	
	Summing \eqref{eq:psi_descent_combined} in the \(k\)-th inner loop gives
	\begin{equation}\label{eq:terminal_merit_bound}
		\tilde{\psi}_{\sigma_k}(x^{k+1},y^{k+1},z^{k+1},s^{k+1})
		\le
		\tilde{\psi}_{\sigma_k}(x^k,\tilde y^k,\tilde z^k,\tilde s^k)
		+
		2\gamma_1
		\sum_{\ell=0}^{\ell_k}(\zeta^\theta_{k,\ell})^2.
	\end{equation}
	Combining $\tilde{\psi}_{\sigma_k}(x^k,\tilde y^k,\tilde z^k,\tilde s^k)\to0$, \eqref{eq:terminal_merit_bound}, and $\sum_{\ell}(\zeta^\theta_{k,\ell})^2 \to0$, as $k \to \infty$  yields\linebreak[4] $\tilde{\psi}_{\sigma_k}(x^{k+1},y^{k+1},z^{k+1},s^{k+1})\to0$. Since all terms in \(\tilde{\psi}_{\sigma_k}\) are nonnegative, it follows that
	\[
	\mathcal G_\gamma(x^{k+1},y^{k+1},z^{k+1})\to0,
	\qquad
	\|g(x^{k+1},y^{k+1})+s^{k+1}\|\to0 .
	\]
	Together with \eqref{eq:bounded_penalty_terminal_residual}, this gives the same vanishing residual property in both cases. Passing to the convergent subsequence and using the continuity of \(\mathcal G_\gamma\) and \(g\) gives \(\mathcal G_\gamma(\bar x,\bar y,\bar z)=0\) and \(g(\bar x,\bar y)+\bar s=0\).
\end{proof}

\begin{remark}
	A simple class of the sequence \(\zeta^\theta_{k,\ell}\) satisfying the above conditions is given by  \(\zeta^\theta_{k,\ell}=\tilde{\zeta}^\theta_k\hat{\zeta}^\theta_\ell\), where \(\tilde{\zeta}^\theta_k\to0\) and \(\sum_{\ell=0}^{\infty}(\hat{\zeta}^\theta_\ell)^2<\infty\). For instance, one may take \(\tilde{\zeta}^\theta_k=(k+1)^{-1}\) and \(\hat{\zeta}^\theta_\ell=2^{-\ell}\).
\end{remark}

The next theorem shows that the accumulation points of the iterates generated by Algorithm~\ref{alg2} satisfy the approximate KKT condition in Definition~\ref{def:two_parameter_approximate_kkt}.

\begin{theorem}\label{thm:slack_algorithm_approximate_kkt}
	Suppose that Assumptions~\ref{asup_stationarity_lower}, \ref{asup:bound}, \ref{asup:F_lis_cts}, and~\ref{asup4} hold, and let \(\{(x^k,y^k,z^k,s^k)\}\) be generated by Algorithm~\ref{alg2}. Assume that $\tau_k \to 0$ as $k \to \infty$, and that the sequence of inexactness parameters $\{\zeta^\theta_{k,\ell}\}$ satisfies
	\[
	\sum_{\ell} (\zeta^\theta_{k,\ell})^2 < \infty \quad \text{for all } k, \qquad \text{and} \qquad \lim_{k \to \infty} \sum_{\ell} (\zeta^\theta_{k,\ell})^2 = 0.
	\]
Then, for every accumulation point \((\bar x,\bar y,\bar z,\bar s)\) of the sequence of iterates $\{(x^k, y^k, z^k, s^k)\}$, the pair \((\bar x,\bar y)\) satisfies the approximate KKT condition for $\mathrm{(GP)}_{\mathcal C}$ in Definition~\ref{def:two_parameter_approximate_kkt}.

	Moreover, assume that the MFCQ holds at $\bar y$ for the lower-level constraint system $\{y\mid g(\bar x,y)\le0\}$, and that the MFCQ holds at $(\bar x,\bar y)$ for the upper-level inequality system $\{(x,y)\mid G(x,y)\le0\}$. Let $\bar w:=(\bar x,\bar y,\bar z)$. Then:
	\begin{itemize}
		\item If \(\{ \sigma_k \}\) is bounded, then \(\bar w\) is S-stationary for \eqref{eq:kkt-reform}.
		\item If \(\sigma_k\to+\infty\), and assume that the MPCC-MFCQ holds at \(\bar w\) for \eqref{eq:kkt-reform}, then \(\bar w\) is M-stationary for \eqref{eq:kkt-reform}.
	\end{itemize}
\end{theorem}

\begin{proof}
	Let \((\bar x,\bar y,\bar z,\bar s)\) be any accumulation point of the sequence \(\{(x^{k+1},y^{k+1},z^{k+1},s^{k+1})\}\) generated by the algorithm. Then there exists a subsequence \(\{k_j\}\), and write \(\ell_j:=\ell_{k_j}\), such that
	\[
	(x^{k_j+1},y^{k_j+1},z^{k_j+1},s^{k_j+1})
	:=(x^{k_j,\ell_j+1},y^{k_j,\ell_j+1},z^{k_j,\ell_j+1},s^{k_j,\ell_j+1})
	\to
	(\bar x,\bar y,\bar z,\bar s).
	\]	
	By Proposition~\ref{prop:slack_feasibility}, \((\bar x,\bar y,\bar z,\bar s)\) is feasible for \eqref{reformulation_slack}. Moreover, \(M_{z,k_j}\to+\infty\), whereas the convergent sequence \(\{(z^{k_j,\ell_j+1},s^{k_j,\ell_j+1})\}\) is bounded. Hence, for all sufficiently large \(j\), $z_i^{k_j,\ell_j+1}<M_{z,k_j}$, and $s_i^{k_j,\ell_j+1}<M_{z,k_j}$, for all $i=1,\ldots,q$. Therefore, the additional upper-bound constraints defining \(\mathcal C_{k_j}\) are inactive. Thus,
	\[
		\mathcal N_{\mathcal C_{k_j}}
		(z^{k_j,\ell_j+1},s^{k_j,\ell_j+1})
		=
		\mathcal N_{\mathcal C}
		(z^{k_j,\ell_j+1},s^{k_j,\ell_j+1}).
	\]

	We next derive the approximate first-order inclusion at the point \((x^{k,\ell_k+1},y^{k,\ell_k+1},z^{k,\ell_k+1},s^{k,\ell_k+1})\). For notational simplicity, write \(\ell:=\ell_k\). The stopping rule \eqref{inner_stop_slack} gives
	\[
	\begin{aligned}
		\left\|
		(x^{k,\ell+1},y^{k,\ell+1},z^{k,\ell+1},s^{k,\ell+1})
		-
		(x^{k,\ell},y^{k,\ell},z^{k,\ell},s^{k,\ell})
		\right\|
		&\le
		\frac{\tau_k}{\sigma_{k,1}\sigma_{k,2}M_{z,k}},\\
		\sigma_{k,1}M_{z,k}
		\max\{\zeta^\theta_{k,\ell},\zeta^\theta_{k,\ell+1}\}&\le\tau_k.
	\end{aligned}
	\]
	Along this subsequence,
	\((x^{k,\ell+1},y^{k,\ell+1},z^{k,\ell+1},s^{k,\ell+1})\)
	is bounded. Since \(\sigma_{k,1}\), \(\sigma_{k,2}\), and \(M_{z,k}\) have positive lower bounds and \(\tau_k\to0\), the first condition in \eqref{inner_stop_slack} implies that
	\[
	\left\|
	(x^{k,\ell+1},y^{k,\ell+1},z^{k,\ell+1},s^{k,\ell+1})
	-
	(x^{k,\ell},y^{k,\ell},z^{k,\ell},s^{k,\ell})
	\right\|\to0
	\]
	along this subsequence. Hence the preceding inner iterates
	\((x^{k,\ell},y^{k,\ell},z^{k,\ell},s^{k,\ell})\) are also bounded. Therefore, there exists a bounded set \(Z\subset\mathbb R_+^q\) containing \(z^{k,\ell}\), \(z^{k,\ell+1}\), \(s^{k,\ell}\), and \(s^{k,\ell+1}\) for all sufficiently large \(k\) along the subsequence. Moreover, \(\zeta^\theta_{k,\ell}\to0\) and \(\zeta^\theta_{k,\ell+1}\to0\). By Lemma~\ref{lem:theta-bounded}, for all sufficiently large \(k\), the inexact points \(\theta^{k,\ell}\), \(\theta^{k,\ell+1/2}\), and the exact points \(\theta^*(x^{k,\ell+1},y^{k,\ell+1},z^{k,\ell})\), \(\theta^*(x^{k,\ell+1},y^{k,\ell+1},z^{k,\ell+1})\) all belong to a common bounded set. Thus the relevant \(\theta\) of the subsequence lie in a bounded set, there exists a finite constant \(L_{\mathcal L}^{\theta}>0\) such that
	\[
	\|\nabla_x\mathcal L(x,\theta_1,z)-\nabla_x\mathcal L(x,\theta_2,z)\|
	\le
	L_{\mathcal L}^{\theta}\|\theta_1-\theta_2\|
	\]
	for all relevant \(x\), \(z\), and \(\theta_1,\theta_2\). Moreover, Lemmas~\ref{lem:theta-z-lipschitz}-\ref{lem_descent}, applied on the above bounded sets, provide finite constants \(L_R^{xy}\), \(L_{\nabla_{xy}\mathcal G}^{z}\), \(M_{g,Z}\), \(L_{\theta,z}\), and \(L_R^z\) all independent of \(k\) along the selected subsequence.
	
For each sufficiently large \(k\) along the selected subsequence, for notational simplicity, set \(\ell:=\ell_k\) and suppress the outer index \(k\) from all inner-loop quantities along the selected subsequence.
	
	The optimality condition for the \((x,y)\)-projection
	\eqref{update_xy_slack_text} gives
	\[
	-
	d_{xy}^{\ell}
	-
	\frac{1}{\alpha_{\ell}}
	\bigl((x^{\ell+1},y^{\ell+1})-(x^{\ell},y^{\ell})\bigr)
	\in
	{\mathcal N}_{\Omega}(x^{\ell+1},y^{\ell+1}).
	\]
	Multiplying this inclusion by \(\sigma_{k,1}\) and adding
	\(\sigma_{k,1}\nabla_{x,y}\psi_{\sigma_k}
	(x^{\ell+1},y^{\ell+1},z^{\ell+1},s^{\ell+1})\), we obtain
	\begin{equation}\label{eq:slack-xy-projected}
		\xi_{xy}^k\in
		\sigma_{k,1}\nabla_{x,y}\psi_{\sigma_k}
		(x^{\ell+1},y^{\ell+1},z^{\ell+1},s^{\ell+1})
		+
		{\mathcal N}_{\Omega}(x^{\ell+1},y^{\ell+1}),
	\end{equation}
	where
	\[
	\xi_{xy}^k
	:=
	\sigma_{k,1}
	\left(
	\nabla_{x,y}\psi_{\sigma_k}
	(x^{\ell+1},y^{\ell+1},z^{\ell+1},s^{\ell+1})
	-
	d_{xy}^{\ell}
	\right)
	-
	\frac{\sigma_{k,1}}{\alpha_{\ell}}
	\bigl((x^{\ell+1},y^{\ell+1})-(x^{\ell},y^{\ell})\bigr).
	\]
	
	Similarly, the projected \((z,s)\)-update \eqref{update_zs_full} gives
	\begin{equation}\label{eq:slack-zs-projected}
		\xi_{zs}^k\in
		\sigma_{k,1}\nabla_{z,s}\psi_{\sigma_k}
		(x^{\ell+1},y^{\ell+1},z^{\ell+1},s^{\ell+1})
		+
		{\mathcal N}_{\mathcal C_k}(z^{\ell+1},s^{\ell+1}),
	\end{equation}
	where
	\[
	\xi_{zs}^k
	:=
	\sigma_{k,1}
	\left(
	\nabla_{z,s}\psi_{\sigma_k}
	(x^{\ell+1},y^{\ell+1},z^{\ell+1},s^{\ell+1})
	-
	d_{zs}^{\ell}
	\right)
	-
	\frac{\sigma_{k,1}}{\beta_{\ell}}
	\bigl((z^{\ell+1},s^{\ell+1})-(z^{\ell},s^{\ell})\bigr).
	\]

	We now estimate \(\xi_{xy}^k\). The definition of \(\xi_{xy}^k\) gives
	\begin{equation}\label{th3_xy_err}
		\begin{aligned}
			\|\xi_{xy}^k\|
			&\le
			\sigma_{k,1}
			\Big\|
			\nabla_{x,y}\psi_{\sigma_k}
			(x^{\ell+1},y^{\ell+1},z^{\ell+1},s^{\ell+1})
			-
			\nabla_{x,y}\psi_{\sigma_k}
			(x^{\ell+1},y^{\ell+1},z^\ell,s^\ell)
			\Big\|\\
			&\quad+
			\sigma_{k,1}
			\Big\|
			\nabla_{x,y}\psi_{\sigma_k}
			(x^{\ell+1},y^{\ell+1},z^\ell,s^\ell)
			-
			\nabla_{x,y}\hat\psi_{\sigma_k}
			(x^{\ell+1},y^{\ell+1},z^\ell,s^\ell;\theta^\ell)
			\Big\|  \\
			&\quad+
			\sigma_{k,1}
			\Big\|
			\nabla_{x,y}\hat\psi_{\sigma_k}
			(x^{\ell+1},y^{\ell+1},z^\ell,s^\ell;\theta^\ell)
			-
			\nabla_{x,y}\hat\psi_{\sigma_k}
			(x^\ell,y^\ell,z^\ell,s^\ell;\theta^\ell)
			\Big\|\\
			&\quad +
			\frac{\sigma_{k,1}}{\alpha_\ell}
			\|(x^{\ell+1},y^{\ell+1})-(x^\ell,y^\ell)\| .
		\end{aligned}
	\end{equation}
	For the first term, the \(z\)-Lipschitz estimate for \(\nabla_{x,y}\mathcal G_\gamma\) in Lemma \ref{lem:Gxy-z-lip} and the boundedness of \(\nabla g\) yield 
	\begin{equation*}\label{th3_xy_err_1}
		\begin{aligned}
			&\quad \Big\|
			\nabla_{x,y}\psi_{\sigma_k}
			(x^{\ell+1},y^{\ell+1},z^{\ell+1},s^{\ell+1})
			-
			\nabla_{x,y}\psi_{\sigma_k}
			(x^{\ell+1},y^{\ell+1},z^\ell,s^\ell)
			\Big\|  \\
			&\le
			\Big\|
			\nabla_{x,y}\mathcal G_\gamma
			(x^{\ell+1},y^{\ell+1},z^{\ell+1})
			-
			\nabla_{x,y}\mathcal G_\gamma
			(x^{\ell+1},y^{\ell+1},z^\ell)\Big\|+
			\sigma_{k,2}
			\Big\|\nabla_{x,y}g(x^{\ell+1},y^{\ell+1})^\top
			(s^{\ell+1}-s^\ell)
			\Big\| \\
			&\le
			L_{\nabla_{xy}\mathcal G}^{z}
			\|z^{\ell+1}-z^\ell\|
			+
			\sigma_{k,2}M_g
			\|s^{\ell+1}-s^\ell\|.
		\end{aligned}
	\end{equation*}
	For the second term, using the constant \(L_{\mathcal L}^{\theta}\) defined above, Lemma~\ref{error_bound}, and \(R(\theta^\ell;x^\ell,y^\ell,z^\ell)\le \zeta^\theta_\ell\), we have
	\[
	\begin{aligned}
		&\quad \Big\|
		\nabla_{x,y}\psi_{\sigma_k}
		(x^{\ell+1},y^{\ell+1},z^\ell,s^\ell)
		-
		\nabla_{x,y}\hat\psi_{\sigma_k}
		(x^{\ell+1},y^{\ell+1},z^\ell,s^\ell;\theta^\ell)
		\Big\|  \\
		&\le
		\left(L_{\mathcal L}^{\theta}+\frac{1}{\gamma_1}\right)
		\left\|
		\theta^*(x^{\ell+1},y^{\ell+1},z^\ell)
		-
		\theta^\ell
		\right\| \\
		&\le
		(1+\gamma_1L_{\mathcal L}^{\theta})
		\left[
		\zeta^\theta_\ell
		+
		L_R^{xy}
		\|(x^{\ell+1},y^{\ell+1})-(x^\ell,y^\ell)\|
		\right].
	\end{aligned}
	\]
	For the third term, the \(L_{\hat\psi}^{xy}\)-Lipschitz continuity of \((x,y)\mapsto\nabla_{x,y}\hat\psi_{\sigma_k} (x,y,z^\ell,s^\ell;\theta^\ell)\) gives
	\begin{equation*}\label{th3_xy_err_3}
		\Big\|
			\nabla_{x,y}\hat\psi_{\sigma_k}
			(x^{\ell+1},y^{\ell+1},z^\ell,s^\ell;\theta^\ell)
			-
			\nabla_{x,y}\hat\psi_{\sigma_k}
			(x^\ell,y^\ell,z^\ell,s^\ell;\theta^\ell)
			\Big\|\le
			L_{\hat\psi}^{xy}
			\|(x^{\ell+1},y^{\ell+1})-(x^\ell,y^\ell)\|.
	\end{equation*}	
	Combining the preceding estimates with \eqref{th3_xy_err}, we obtain
	\begin{equation}\label{th3_xy_err_all}
		\begin{aligned}
			\|\xi_{xy}^k\|
			&\le
			\sigma_{k,1}L_{\nabla_{xy}\mathcal G}^{z}
			\|z^{\ell+1}-z^\ell\|
			+
			\sigma_{k,1}\sigma_{k,2}M_g
			\|s^{\ell+1}-s^\ell\|+
			\sigma_{k,1}(1+\gamma_1L_{\mathcal L}^{\theta})\zeta^\theta_\ell  \\
			&\quad +
			\sigma_{k,1}
			\left[
			(1+\gamma_1L_{\mathcal L}^{\theta})L_R^{xy}
			+
			L_{\hat\psi}^{xy}
			+
			\frac{1}{\alpha_\ell}
			\right]
			\|(x^{\ell+1},y^{\ell+1})-(x^\ell,y^\ell)\|.
		\end{aligned}
	\end{equation}

	By the definition of \(L_{\hat\psi}^{xy}\) in \eqref{def_L^{xy}}, there exists \(C_{xy}>0\) such that $L_{\hat\psi}^{xy}\le C_{xy}M_{z,k}(1+\sigma_{k,2})$. Moreover, the accepted backtracking index satisfies $\widehat L_{\ell}^{xy,j_{\ell}^{xy}}\le\max\{\widehat L_{\ell}^{xy,0},\kappa_L L_\alpha^*\}$, $L_\alpha^*:=L_{\hat\psi}^{xy}/2+\gamma_1(L_R^{xy})^2$. Since \(L_R^{xy}\) and the initial trial values are uniformly bounded along the selected subsequence, while \(\sigma_{k,2}\) and \(M_{z,k}\) are bounded away from zero, enlarging \(C_{xy}\) if necessary gives $(L_{\hat\psi}^{xy}+\alpha_\ell^{-1})=(L_{\hat\psi}^{xy}+\widehat L_{\ell}^{xy,j_{\ell}^{xy}}+c_\alpha)\le C_{xy}\sigma_{k,2}M_{z,k}$. Substituting this estimate and the stopping condition \eqref{inner_stop_slack} into \eqref{th3_xy_err_all}, we obtain
	\[
	\|\xi_{xy}^k\|\le \left(\frac{L_{\nabla_{xy}\mathcal G}^{z}}{\sigma_{k,2}M_{z,k}} + \frac{M_g}{M_{z,k}} +\frac{1+\gamma_1L_{\mathcal L}^{\theta}}{M_{z,k}}+\frac{(1+\gamma_1L_{\mathcal L}^{\theta})L_R^{xy}}{\sigma_{k,2}M_{z,k}} +C_{xy} \right)\tau_k\to 0.
	\]

	For the \((z,s)\)-part, using \eqref{Ggradient} and the definition of
	\(d_{zs}^{\ell}\), the triangle inequality gives
	\begin{equation}\label{th3_zs_err}
		\begin{aligned}
			\|\xi_{zs}^k\|
			&\le
			\sigma_{k,1}
			\Big\|
			\nabla_z\mathcal P_{\gamma_2}(x^{\ell+1},y^{\ell+1},z^{\ell+1})
			-
			\nabla_z\mathcal P_{\gamma_2}(x^{\ell+1},y^{\ell+1},z^\ell)
			\Big\|+
			\sigma_{k,1}\sigma_{k,2}\|s^{\ell+1}-s^\ell\|\\
			&\quad +
			\sigma_{k,1}
			\Big\|
			g(x^{\ell+1},\theta^*(x^{\ell+1},y^{\ell+1},z^{\ell+1}))
			-
			g(x^{\ell+1},\theta^{\ell+1/2})
			\Big\|+
			\frac{\sigma_{k,1}}{\beta_\ell}
			\|(z^{\ell+1},s^{\ell+1})-(z^\ell,s^\ell)\|.
		\end{aligned}
	\end{equation}
	For the first term, the \(2/\gamma_2\)-Lipschitz continuity of
	\(\nabla_z\mathcal P_{\gamma_2}\) gives
	\[
	\Big\|
	\nabla_z\mathcal P_{\gamma_2}(x^{\ell+1},y^{\ell+1},z^{\ell+1})
	-
	\nabla_z\mathcal P_{\gamma_2}(x^{\ell+1},y^{\ell+1},z^\ell)
	\Big\|
	\le
	\frac{2}{\gamma_2}\|z^{\ell+1}-z^\ell\|.
	\]
	For the third term, the \(M_{g,Z}\)-Lipschitz continuity of \(g(x,\cdot)\), Lemmas~\ref{error_bound} and~\ref{lem:theta-z-lipschitz}, and the residual \eqref{theta_slack_half} yield
	\[
		\Big\|
		g(x^{\ell+1},\theta^*(x^{\ell+1},y^{\ell+1},z^{\ell+1}))
		-
		g(x^{\ell+1},\theta^{\ell+1/2})
		\Big\|\le
		M_{g,Z}
		\left[
		L_{\theta,z}\|z^{\ell+1}-z^\ell\|
		+
		\gamma_1\zeta^\theta_\ell
		\right].
	\]
	Combining these estimates with \eqref{th3_zs_err}, we obtain
	\begin{equation}\label{th3_zs_err_all}
		\begin{aligned}
			\|\xi_{zs}^k\|
			&\le
			\sigma_{k,1}
			\left(
			\frac{2}{\gamma_2}
			+
			M_{g,Z}L_{\theta,z}
			\right)
			\|z^{\ell+1}-z^\ell\|+
			\sigma_{k,1}\gamma_1M_{g,Z}\zeta^\theta_\ell
			+
			\sigma_{k,1}\sigma_{k,2}\|s^{\ell+1}-s^\ell\|\\
			&\quad+
			\frac{\sigma_{k,1}}{\beta_\ell}
			\|(z^{\ell+1},s^{\ell+1})-(z^\ell,s^\ell)\|.
		\end{aligned}
	\end{equation}
	
	By Lemma~\ref{lem10}, the constant can be
	chosen as \(L_{\hat\psi}^{zs}=\sigma_{k,2}\), and the accepted
	backtracking index satisfies $\widehat L_{\ell}^{zs,j_{\ell}^{zs}}\le\max\{\widehat L_{\ell}^{zs,0},\kappa_LL_\beta^*\}$, $L_\beta^*:=\sigma_{k,2}+2\gamma_1(L_R^z)^2+c_\beta$. Along the selected subsequence, \(L_R^z\) and \(\widehat L_{\ell}^{zs,0}\) are bounded. Therefore, there exists \(C_{zs}>0\) such that $1/\beta_\ell=\widehat L_{\ell}^{zs,j_{\ell}^{zs}}+c_\beta	\le	C_{zs}\sigma_{k,2}$. Substituting this estimate and the stopping condition
	\eqref{inner_stop_slack} into \eqref{th3_zs_err_all}, we obtain
	\[
	\|\xi_{zs}^k\|
	\le
	\left(
	\frac{
		2/\gamma_2+M_{g,Z}L_{\theta,z}
	}{
		\sigma_{k,2}M_{z,k}
	}
	+
	\frac{
		1+\gamma_1M_{g,Z}+C_{zs}
	}{
		M_{z,k}
	}
	\right)\tau_k.
	\]
	Therefore the displayed bounds for all terms in \(\xi_{zs}^k\), together with \(\tau_k\to0\), yield \(\xi_{zs}^k\to0\).
	
	Set \(\varepsilon^j:=(\xi_{xy}^{k_j},\xi_{zs}^{k_j})\). Then
	\(\varepsilon^j\to0\). Combining the projected optimality conditions
	\eqref{eq:slack-xy-projected} and \eqref{eq:slack-zs-projected} yields
	\[
	\begin{aligned}
		\varepsilon^j&\in
		\nabla F(x^{k_j+1},y^{k_j+1})
		+
		\sigma_{k_j,1}\nabla \mathcal G_\gamma(x^{k_j+1},y^{k_j+1},z^{k_j+1})+
		\frac{\sigma_{k_j,1}\sigma_{k_j,2}}{2}\nabla \left( \|g(x^{k_j+1}, y^{k_j+1}) + s^{k_j+1}\|^2\right) \\
		&\quad +
		\mathcal N_{\Omega\times\mathcal C}
		(x^{k_j+1},y^{k_j+1},z^{k_j+1},s^{k_j+1}).
	\end{aligned}
	\]
	With these definitions
	\[
	(x^j,y^j,z^j,s^j):=(x^{k_j+1},y^{k_j+1},z^{k_j+1},s^{k_j+1}),
	\quad
	\rho^j:=(\rho_1^j,\rho_2^j)
	:=\left(\sigma_{k_j,1},\frac{\sigma_{k_j,1}\sigma_{k_j,2}}{2}\right),
	\]
	the convergence and feasibility relations established above, together with \(\varepsilon^j\to0\), verify every condition in Definition~\ref{def:two_parameter_approximate_kkt}. Therefore \((\bar x,\bar y)\) satisfies the two-parameter slack-penalty approximate KKT condition.

	It remains to prove the S- and M-stationarity conclusions. Under the additional MFCQ assumptions stated above, the approximate KKT condition just verified allows us to invoke Theorem~\ref{thm:two_parameter_M_stationarity} with \(\rho_1^j=\sigma_{k_j,1}\) and \(\rho_2^j=\sigma_{k_j,1}\sigma_{k_j,2}/2\). If \(\{\sigma_k\}\) is bounded, then both \(\{\rho_1^j\}\) and \(\{\rho_2^j\}\) are bounded, and Theorem~\ref{thm:two_parameter_M_stationarity} yields S-stationarity of \(\bar w\). If \(\sigma_k\to+\infty\), then, along the selected subsequence, \(\rho_1^j\to+\infty\), \(\rho_2^j\to+\infty\), and \(\rho_2^j/\rho_1^j=\sigma_{k_j,2}/2\to+\infty\). Therefore, under MPCC-MFCQ at \(\bar w\), the same theorem yields M-stationarity of \(\bar w\). This proves the stated stationarity conclusions.
\end{proof}

\section{Conclusion}

This paper studied the regularized gap-function reformulation of bilevel optimization, which belongs to the broader class of value-function-type reformulations. Such reformulations are intrinsically degenerate and generally fail to satisfy standard constraint qualifications. Consequently, the multiplier sequence associated with the regularized gap-function constraint may be unbounded, while the limiting stationarity in this regime has received relatively limited attention. We established that the resulting accumulation points are C-stationary for the associated KKT-based MPCC reformulation, provided that MFCQ holds for the upper- and lower-level constraint systems and MPCC-MFCQ holds at the limiting MPCC point. To recover M-stationarity, we introduced a slack-based two-parameter formulation that preserves exact multiplier--slack complementarity. Under the same constraint qualifications, the resulting accumulation points are M-stationary. Finally, we developed an inexact bilevel slack-penalty method and established M-stationarity of its accumulation points under the stated algorithmic assumptions.

\bibliographystyle{plain}
\bibliography{references}

\end{document}